\documentclass[11pt, reqno]{amsart}
\usepackage{amssymb}
\usepackage[mathcal]{euscript}
\usepackage{upgreek}
\usepackage{cancel}

\usepackage{amsmath, amsthm, amscd, amsfonts}
\usepackage{mathrsfs}

\usepackage{bbm, bm}
\usepackage[left= 3.4 cm, right= 3.4 cm, top= 2.5 cm, bottom=2.6 cm, foot= 1.1 cm, marginparwidth=2.7 cm, marginparsep=0.5 cm]{geometry}
\usepackage{abstract}
\usepackage{color}
\DeclareFontFamily{OT1}{pzc}{}
\DeclareFontShape{OT1}{pzc}{m}{it}{<-> s * [1.200] pzcmi7t}{}
\DeclareMathAlphabet{\mathpzc}{OT1}{pzc}{m}{it}

\usepackage{hyperref}
\hypersetup{colorlinks=true, linkcolor=magenta, citecolor=blue}
	
	\newcommand{\mb}{\mathbb}
	
	\newcommand{\mf}{\mathfrak}
	\newcommand{\mc}{\mathcal}
	\newcommand{\ms}{\mathscr}
	\newcommand{\ov}{\overline}
	
	\newcommand{\ud}{\underline}
	\newcommand{\wt}{\widetilde}
	\newcommand{\mz}{\mathpzc}
	\makeatletter % `@' now normal "letter"
	\@addtoreset{equation}{section}
	\makeatother  % `@' is restored as "non-letter"

	\numberwithin{equation}{section}

\begin{document}

\title{Analysis of gauged Witten equation}

\author[Tian]{Gang Tian}
\address{
Department of Mathematics \\
Princeton University \\
Fine Hall, Washington Road \\
Princeton, NJ 08544 USA
}
\email{tian@math.princeton.edu}

\author[Xu]{Guangbo Xu}
\address{
Department of Mathematics \\
University of California, Irvine\\
Irvine, CA 92697 USA 
}
\email{guangbox@math.uci.edu}

\date{\today}
	
	\newtheorem{thm}{Theorem}[section]
	\newtheorem{lemma}[thm]{Lemma}
	\newtheorem{cor}[thm]{Corollary}
	\newtheorem{prop}[thm]{Proposition}
		\newtheorem{conj}[thm]{Conjecture}
	
	\theoremstyle{definition}
	\newtheorem{defn}[thm]{Definition}
	
	\theoremstyle{remark}
	\newtheorem{rem}[thm]{Remark}
	\newtheorem{hyp}[thm]{Hypothesis}
	\newtheorem{example}[thm]{Example}
	
\setcounter{tocdepth}{1}
% set the table of contents up to sections

\maketitle

\begin{abstract}
The gauged Witten equation was essentially introduced by Witten in his formulation of gauged linear $\sigma$-model (GLSM) in \cite{Witten_LGCY}. GLSM is a physics theory which explains the so-called Landau-Ginzburg/Calabi-Yau correspondence. This is the first paper in a series towards a mathematical construction of GLSM. In this paper we study some analytical properties of the gauged Witten equation for a Lagrange multiplier type superpotential. It contains the asymptotic property of finite energy solutions, the linear Fredholm property, the uniform $C^0$-bound, and the compactness of the moduli space of solutions over a fixed smooth $r$-spin curve with uniform energy bound.


\end{abstract}

\tableofcontents

\section{Introduction}

In this paper we study a system of elliptic partial differential equations over a Riemann surface, called the gauged Witten equation, which originated from physicists' study of superstring theory. This is the first piece of work in a series which aims at a rigorous construction of Witten's gauged linear $\sigma$-model (\cite{Witten_LGCY}), and which, from a mathematical point of view, generalizes both the theory of symplectic vortex equation (see \cite{Cieliebak_Gaio_Salamon_2000}, \cite{Mundet_thesis}, \cite{Mundet_2003}) and the theory of Witten equation (see \cite{FJR1, FJR3, FJR2}). It is also a new member of the collection of mathematical works related to quantum field theory, which has greatly influenced mathematics in the past few decades. Therefore we would like to explain our motivation from a historical perspective, and many related works will be recalled.

Two celebrated members of this collection are Gromov-Witten theory and gauge theory. Gromov-Witten theory, inspired by Gromov's work on $J$-holomorphic curves \cite{Gromov_1985} and Witten's interpretation \cite{Witten_sigma_model}, has been constructed rigorously by mathematicians (\cite{Ruan_Tian}, \cite{Ruan_96},  \cite{Li_Tian_2}, \cite{Li_Tian}, \cite{Fukaya_Ono} etc.). The field-theoretic correlation function, called the Gromov-Witten invariant, has become a fundamental tool in symplectic topology as well as in algebraic geometry. On the other hand, for last few decades, a lot of exciting results about gauge theory have been proven, notably, Atiyah-Bott's famous work \cite{Atiyah_Bott}, Uhlenbeck's compactness theorem \cite{Uhlenbeck_82}, Taubes' construction of self-dual connections (\cite{Taubes_SDYM}) and Donaldson theory on differentiable 4-manifolds (\cite{Donaldson_Kronheimer}). 

The coupling of gauge theory and $\sigma$-model is fundamental in physics, which also has been adapted by mathematicians. Many interesting examples came from ``dimensional reduction'' of four-dimensional gauge theory to dimension 2, where conformal invariance brings richer structures. For example, one considers a rank $n$ Hermitian vector bundle $E$ over a Riemann surface $\Sigma$, and consider the (linear) vortex equation on a pair $(A, u)$:
\begin{align}\label{equation11}
\left\{ \begin{array}{ccc}
D_A^{0,1} u & = & 0,\\
{\bm i} * F_A + \left( u\otimes u^* - \tau {\rm Id}_E\right) & = & 0.
\end{array} \right.
\end{align}
Here $A$ (gauge field) is a unitary connection on $E$, $u$ (matter field) is a smooth section of $E$, and $\tau$ is a constant parameter (see \cite{Bradlow_stable_pairs}). The vortex equation carries a new feature: the moduli space undergoes a birational-like transformation when $\tau$ varies (see for example \cite{Bradlow_birational}). In the language of algebraic geometry, this is called the variation of GIT quotient. Interesting results have been proved by utilizing this feature (cf. \cite{Thaddeus_pairs}), which share similar spirit of the Landau-Ginzburg/Calabi-Yau correspondence we will soon review.

Observing that the term $(u\otimes u^* - \tau {\rm Id}_E)$ is of the form of a moment map of the $U(n)$-action on ${\mb C}^n$, the vortex equation can be generalized to a symplectic manifold $X$ with a Hamiltonian $G$-action. This was firstly studied by Mundet in his thesis (cf. \cite{Mundet_thesis}, \cite{Mundet_2003}) and Cieliebak-Gaio-Salamon (\cite{Cieliebak_Gaio_Salamon_2000}). The equation is now called the {\bf symplectic vortex equation}. Using the moduli space of solutions to the symplectic vortex equation, certain invariants of Hamiltonian $G$-manifolds, called the gauged (or Hamiltonian) Gromov-Witten invariants can be defined (see \cite{Mundet_2003}, \cite{Cieliebak_Gaio_Mundet_Salamon_2002}, \cite{Mundet_Tian_Draft} etc.). On the other hand, such invariants are closely related to the Gromov-Witten invariants of the symplectic quotient of $X$: in the ``adiabatic limit'' the symplectic vortex equation reduces to $J$-holomorphic curves in the symplectic quotient (see \cite{Gaio_Salamon_2005}). Therefore, the gauged Gromov-Witten invariants also relate the Gromov-Witten invariants of different symplectic/GIT quotients (cf. \cite{Woodward_1, Woodward_2, Woodward_3}, \cite{Gonzalez_Woodward_wall_crossing} for the algebraic case).

Another important ingredient in field theory is the potential energy. Via localization, such field theories are closely related to the geometry of the ``singularity'' of the potential function. If the potential is a holomorphic function on a complex manifold, then such a theory is usually referred to as a Landau-Ginzburg theory. LG theories are naturally related to the study of singularities in topology and algebraic geometry.

In \cite{Witten_spin}, Witten proposed an elliptic equation associated to a quasi-homogeneous polynomial $W$ in $n$ complex variables (now called {\bf Witten equation}), which was motivated from physicists' study of matrix models of two dimensional quantum gravity. His equation takes a simple form as a ``complex gradient flow equation'':
\begin{align}\label{equation12}
{\partial u_i \over \partial \ov{z}} + \ov{ \partial_i W} (u_1, \ldots, u_n) = 0.
\end{align}
In particular for polynomials which define the simple singularities (which have the famous A-D-E classification), Witten conjectured that certain generating functions satisfy the generalized KdV hierarchies. This generalized his earlier conjecture about $A_1$-singularity and KdV hierarchy (\cite{Witten_conjecture}), which was proved by Kontsevich \cite{Kontsevich_92} (and later by Mirzakhani \cite{Mirzakhani}). For higher type $A$ singularities, generalized Witten's conjecture were proved by various people (Jarvis-Kimura-Vaintrob \cite{JKV}, Lee \cite{YPLee}, and Faber-Shadrin-Zvonkine \cite{FSZ}) using algebro-geometric method; while for type $D$ and type $E$ singularities, algebraic method seemed to be insufficient. In a series of papers (\cite{FJR1, FJR3, FJR2}), Fan-Jarvis-Ruan used analytic method to study the moduli space of Witten equation (\ref{equation12}) for general nondegenerate quasi-homogeneous polynomials, and proved generalized Witten's conjecture for $D_n$-singularities for even $n \geq 6$ and all type $E$-singularities. This much more systematic approach is referred to as the Landau-Ginzbug A-model theory, which can be viewed as a quantum theory about singularities.

Around 1990s, physicists discovered a correspondence between the ``Landau-Ginzburg model'' and the nonlinear $\sigma$-model of Calabi-Yau hypersurfaces (\cite{GVW}, \cite{Martinec}, \cite{VafaW}). It becomes a famous mathematical conjecture, often referred to as the Landau-Ginzburg/Calabi-Yau correspondence (LG/CY for short). The conjecture can be vaguely stated as follows.
\begin{conj}
The Landau-Ginzburg theory of a quasi-homogeneous superpotential $W$ of Calabi-Yau type is isomorphic to the nonlinear $\sigma$-model of the weighted projective hypersurface defined by $W$ in a certain sense.
\end{conj}
This conjecture is certainly one of the most important problems in studying mathematical aspects of 2-dimensional quantum field theories. It has many deep implications, e.g., simpler method of calculating Gromov-Witten invariants of Calabi-Yau manifolds and applications to mirror symmetry, etc..

Witten (\cite{Witten_LGCY}) observed that this correspondence can be explained as a phase transition via the variation of the Fayet-Iliopoulos D-term (something similar to the $\tau$ in (\ref{equation11})) in certain ``master theory''. This master theory, usually referred to as the {\bf gauged linear $\sigma$-model} (GLSM), 
flows in low energy to the LG and CY models respectively in different phases. Let us illustrate Witten's idea through the following important example.

More precisely, if $Q$ is a quintic polynomial in variables $x = (x_1, \ldots, x_5)$, then Witten proposed to study (\ref{equation12}) for $W(x, p) = pQ(x)$. Moreover, on the $(x, p)$-space there is an $S^1$-action with weight $(1, 1, 1, 1, 1, -5)$ under which $W$ is invariant. Then the equation carries a gauge invariance with respect to this action. Moreover, this action is Hamiltonian with moment map proportional to
\begin{align}\label{equation13}
\mu(x, p) = - 5 |p|^2 + \sum_{i=1}^5 |x_i|^2 + \tau.
\end{align}
For $\tau>0$, the ``classical vacuum'' is $\left( {\rm Crit} W \cap \mu^{-1}(0)\right) / S^1$, which is the same as the singularity defined by $Q$; for $\tau<0$, the classical vacuum $\left( {\rm Crit} W \cap \mu^{-1}(0) \right) /S^1$ is the quintic hypersurface in ${\mb P}^4$ defined by $Q$. The variation of $\tau$ parametrizes the phase transition therefore the two theories are related.

GLSM has been extensively used by physicists in their research, for example, in the study of mirror symmetry (cf. \cite{Hori_Vafa}). Mathematicians have been also thinking about its mathematical formulation and applications: How to construct them mathematically? How can it be applied to studying mirror symmetry?
For instance,  in \cite{Chang_Li} and \cite{Chang_Li_Li}, J. Li {\it et al.} studied the Gromov-Witten theory of a quintic hypersurface and the narrow case of Landau-Ginzburg theory by using cosection localization, which they believe to lead to an algebraic approach to GLSM and LG/CY correspondence. Fan-Jarvis-Ruan also have a project towards GLSM.

The purpose of our series of papers is to establish a mathematical theory of GLSM, at least, in some simple cases. Our approach is symplecto-geometric and uses geometric analysis. We will start our series by solving some serious technical problems, among which the most crucial one is the compactness of its moduli space. There are well-known difficulties we need to overcome in solving these problems. Our framework also includes the gauged Gromov-Witten theory as a special case where the superpotential is zero. We hope, via certain adiabatic limits, our construction can relates the work of Fan-Jarvis-Ruan on one side and the nonlinear $\sigma$-model on the other side, so it can give a good mathematical understanding of the LG/CY correspondence.

\subsection{Main results of this paper}

Now we briefly describe our main set-up and result of this first paper in our series. Let $(X, \omega, J)$ be a noncompact K\"ahler manifold (with ``bounded geometry'' at infinity), $Q: X \to {\mb C}$ be a nonzero holomorphic function which is homogeneous with respect to a ${\mb C}^*$-action on $X$. Consider $\wt{X} = X \times {\mb C}$ and the function $W: \wt{X} \to {\mb C}$ given by $W (x, p) = p Q(x)$. $W$ is invariant under another ${\mb C}^*$-action on $\wt{X}$. Let $G = S^1 \times S^1$, and there is a moment map $\mu: \wt{X} \to ({\rm Lie} G)^*$ for the $G$-action.

Let $\Sigma$ be a compact Riemann surface with punctures. The gauged Witten equation is roughly a union of the Witten equation and the vortex equation, which reads
\begin{align}\label{equation14}
\left\{  \begin{array}{ccc} \ov\partial_A u + \nabla W(u) & = & 0; \\
* F_A + \mu(u) & = & 0.
\end{array} \right.
\end{align}
The variables of this system are $A$ and $u$, where $A$ is a connection on a $G$-bundle $P \to \Sigma$ and $u$ is a section of the associated bundle $P \times_G \wt{X}$. In fact, such a system can be defined for a quite general class of superpotentials $W$ on a general K\"ahler manifold, which is not necessarily a Lagrange multiplier type one. But in this paper we only consider a special class, because of difficulties in proving compactness for general superpotentials.

The gauged Witten equation generalizes both the symplectic vortex equation (\ref{equation11}) and the Witten equation (\ref{equation12}). It is also the classical equation of motion with respect to the following energy functional. For each pair $(A, u)$, using the superpotential $W$, its energy is defined as
\begin{align}\label{equation15}
E(A, u) = {1\over 2} \Big( \big\|d_A u \big|_{L^2(\Sigma)}^2 + \big\| F_A \big\|_{L^2(\Sigma)}^2 + \big\| \mu(u) \big\|_{L^2(\Sigma)}^2 \Big) + \big\|\nabla W(u) \big\|_{L^2(\Sigma)}^2.
\end{align}

(\ref{equation14}) is not Fredholm in general because $W$ may have degenerate critical points. On a cylindrical end $[0, +\infty)\times S^1$ of the punctured surface with cylindrical coordinates $(s, t)$, the unperturbed equation is essentially the following Floer-type equation
\begin{align}\label{equation16}
{\partial u \over \partial s} + J  {\partial u \over \partial t} +  \nabla W (u) = 0.
\end{align}
To have a Fredholm operator we have to modify $W$ on cylindrical ends. In Section \ref{section2} we discuss the perturbation of the equation on the cylindrical ends at the ``broad'' punctures of $\Sigma$, so that after perturbation, $W$ becomes a holomorphic Morse function. After perturbation, (\ref{equation14}) gives a nonlinear Fredholm operator. In Section \ref{section3} we prove
\begin{thm}
Every bounded solution (see Definition \ref{defn41}) to the perturbed version of (\ref{equation14}) converges to a critical point of the perturbed $W$ at each cylindrical end, and the energy density decays exponentially. Modulo gauge transformation, the linearization of the left-hand-side of the perturbed version of (\ref{equation14}) is a linear Fredholm operator between certain Sobolev spaces (see Theorem \ref{thmx52}). Its Fredholm index is given by (\ref{equationx58}).
\end{thm}

There are certain difficulties in formulating this problem properly: First, to write down the Witten equation (\ref{equation12}) for a superpotential $W$ on a general Riemann surface which has no global holomorphic coordinate, one needs an extra structure (the $W$-structure) on the Riemann surface. For example, if $W$ is a generic homogeneous polynomial of degree $r$, then a natural choice of such a structure is an $r$-spin structure, i.e., an $r$-th root of the canonical bundle of the Riemann surface. (see \cite{FJR2} for a comprehensive study of $W$-structures and their moduli.) Based on Fan-Jarvis-Ruan's work, we realized that the purpose of having a $W$-structure is to lift the superpotential to the fibre bundle. For GLSM, $W$ is invariant under the action copy of ${\mb C}^*$. Therefore we have to make the $W$-structure consistent with another line bundle so that $W$ can be lifted and we can write (\ref{equation12}) globally on a Riemann surface.

Another difficulty is how to set up a proper perturbation scheme of the gauged Witten equation (\ref{equation14}). In Fan-Jarvis-Ruan's framework, $W$ is a nondegenerate quasi-homogeneous polynomial and the perturbation in \cite{FJR3} was done by adding a small generic holomorphic function $\epsilon f$ to $W$ so that $W+ \epsilon f$ becomes a holomorphic Morse function. Using a cut-off function one can extend the perturbation to the whole Riemann surface. On the other hand, the beautiful Picard-Lefschetz theory about isolated hypersurface singularities guarantees that generic perturbations can give topological information about the singularity. For general $W$ with non-isolated critical points, there is no Picard-Lefschetz theory and we don't know if generic perturbations can unwrap interesting topology. This is one reason why we restrict to the case of superpotentials of Lagrange multiplier type (i.e., $W = pQ$). In this case we perturb $pQ(x)$ to $p(Q(x) -a) + F(x)$, so that the topology of the regular hypersurface $Q^{-1}(a)$ will be relevant, and objects like vanishing cycles appear again. 

\subsection{Difficulties in proving compactness}

The most important technical result of the current paper is the compactness of solutions to the perturbed gauged Witten equation. The theorem reads (see Theorem \ref{thm65})
\begin{thm}
Let $\vec{\mc C}$ be a rigidified $r$-spin curve (see Definition \ref{defn210}). Then any sequence $(A^{(i)}, u^{(i)})$ of stable solutions to the perturbed gauged Witten equation on $\vec{\mc C}$ with $\sup_i E (A^{(i)}, u^{(i)} ) < \infty$, modulo gauge transformation, has a convergent subsequence with respect to the natural Gromov-type topology. 
\end{thm}
Its proof occupies the last three sections (Section \ref{section6}--\ref{section8}) of the paper. Moreover, in order to use the compactness theorem, we need to prove that the energy of solutions with fixed homology class is uniformly bounded (Theorem \ref{thm44}). This requires a delicate control on the contribution of the perturbation term, for which we have to include a non-local parameter in the perturbation term (Definition \ref{defn215}) and impose a few more properties (see Hypothesis \ref{hyp28}). 

The main issue in proving compactness is to establish a uniform $C^0$-bound on solutions. Since the target space is noncompact, this is not automatic and usually one has to assume conditions about the geometry of the target space at infinity. For example, in Gromov-Witten theory one can assume the existence of a plurisubharmonic function on the manifold; in the case of symplectic vortex equation, there is also an analogous, $G$-equivariant version of this convexity assumption (see \cite[Section 2.5]{Cieliebak_Gaio_Mundet_Salamon_2002}). The uniform bound then follows from a strong maximal principle argument.

In our situation, if the equation is unperturbed, the solutions are holomorphic and they are special solutions to the symplectic vortex equation. So one can prove the $C^0$-bound in the same way as in \cite{Cieliebak_Gaio_Mundet_Salamon_2002}. The difficulty lies in the perturbed case, where the perturbation term disturbs the control. Even worse, in our case, the gradient $\nabla W$ is not a proper map, so $\nabla W(u)$ cannot control $u$ (such a control \cite[Theorem 5.8]{FJR1} is a crucial technical ingredient in the compactness theorem of Fan-Jarvis-Ruan).

We take a different route. We prove that for a sequence of solutions $(A_i, u_i)$ with uniform energy bound, if $u_i$ blows up near some point on the Riemann surface, then there must be an energy concentration (Corollary \ref{cor72}). Such a quantization property implies that the sequence are uniformly bounded up to blowing up at finitely many points. Then we argue that the blowing up contradicts with a local maximal principle. 

The establishment of this energy quantization property is lengthy due to the complicated behavior of the superpotential $W$ at infinity. The critical point set ${\rm Crit} W$ is a stratified space, and near infinity of the target space $\wt{X}$, ${\rm Crit} W$ has components of different nature. If the blow-up of solutions happens away from ${\rm Crit} W$, then the energy quantization is easy to achieve; if the blow-up happens near ${\rm Crit} W$, then in general, we can prove the energy concentration only when it is near a component of ${\rm Crit} W$ of Bott type. However, since in our main example $W$ is the Lagrange multiplier of a homogeneous polynomial, whose critical point set has necessarily a degenerate component, considering only Bott type critical loci is not enough. For the degenerate component, we have to use the special structure of the Lagrange multiplier; this is another (and a more important) technical reason why we have to restrict to such type of superpotentials. On the other hand, this part of argument is purely local and it may shed some light on more general cases.

Once $C^0$-bound is established, the remaining part of the proof of the compactness problem is straightforward. In this paper we assume that the target space is aspherical so that we can rule out sphere bubbles. On the other hand, on the cylindrical ends the solutions may undergo a Morse-Floer type degeneration, similar to the situation of \cite[Section 4]{FJR3}. In this situation we have to consider ``solitons'', which are solutions to (\ref{equation16}) on the cylinder ${\mb R} \times S^1$ with $W$ properly perturbed. A stable solution to the perturbed gauged Witten equation is the concatenation of a usual solution with (broken) solitons attached to the cylindrical ends. The construction of a stable solution in a subsequence limit follows from standard arguments.

In this paper we only consider the compactification of the moduli space for a fixed complex structure on the Riemann surface $\Sigma$. The compactification with degenerating complex structures will be much more complicated because the variations of holonomies on the forming nodes can give extra pieces of the limiting stable objects like the situation of \cite{Mundet_Tian_2009}, and it awaits further consideration.

\subsection{A formal definition of the GLSM correlation functions}

Our main goal of this series of papers is to define the correlation functions of the gauged linear $\sigma$-model. For this purpose we have to work out the transversality problem of the moduli space and prove that the correlation functions are independent of many choices we made in defining them. The details of constructing  the virtual cycle and proving its properties will be given in a forthcoming paper \cite{Tian_Xu_3}. Assuming the existence of virtual cycle, we announced the definition of the correlation function in \cite{Tian_Xu_2}, which we sketch here.

The correlation function can be defined for general Lagrange multiplier type superpotentials with appropriate assumptions on the pair $(X, Q)$ (see \cite{Tian_Xu_2}). For simplicity we sketch it for the case of (the Lagrange multiplier of) a quintic polynomial in 5 variables. Let $Q: {\mb C}^5 \to {\mb C}$ be a nondegenerate quintic polynomial and $W = pQ: {\mb C}^6 \to {\mb C}$ be the superpotential of GLSM. The {\bf state space} is the direct sum of the narrow sectors and the broad sector. For $\upgamma^{(k)} = \exp \left( {2k\pi {\bm i} \over 5} \right) \in {\mb Z}_5$ for $k = 1, 2, 3, 4$, the $\upgamma^{(k)}$-sector (which is narrow) of the state space ${\ms H}_k$ is a one-dimensional rational vector space, generated by one vector $\alpha_k$ of degree $2k -2$. For for $\upgamma^{(0)} = 1$, the broad sector ${\ms H}_0$ has pure degree $5$, and is isomorphic to the cohomology group
\begin{align*}
{\ms H}_0 = H^3 \big( \ov{X}_Q; {\mb Q} \big).
\end{align*}
Here $\ov{X}_Q \subset {\mb P}^4$ is the quintic hypersurface defined by $Q$. For each $a \in {\mb C}^*$, ${\ms H}_0$ can be identified canonically with the ${\mb Z}_5$-invariant part of the cohomology $H^4 \big( Q^{-1}(a); {\mb Q} \big)$. A perfect pairing can be defined on ${\ms H}_0$ so it is also identified with ${\ms H}_4 \big( Q^{-1}(a); {\mb Q} \big)^{{\mb Z}_5}$, i.e., the invariant part of the space of vanishing cycles. 

We denote by ${\ms H}_{\rm GLSM}$ the direct sum of broad and narrow sectors. The correlation function is the collection of multi-linear maps
\begin{align}\label{equation17}
\left\langle \ \cdots\ \right\rangle_{g, n}^d: \big( {\ms H}_{\rm GLSM} \big)^{\otimes n}  \to {\mb Q},\ g, n, d \in {\mb Z}, g \geq 0,\ 2g-2+n >0.
\end{align}
To define the correlation function, we need to do certain virtual integration on the moduli space of solutions to the perturbed gauged Witten equation. Here for simplicity, we omit the discussion about gravitational descendents.

In this simplified situation, the topological data we need to fix is the degree of the additional $S^1$-bundle $P_1$. This corresponds to the degrees of holomorphic curves in the quintic 3-fold $\ov{X}_Q$. For each rigidified $5$-spin curve, the perturbation data at broad punctures is given by the choice of $a \in {\mb C}^*$ and a linear function $F(x_1, \ldots, x_5)$. Then denote by ${\mc W}_{g, n}^d$ the moduli space of solutions to the perturbed gauged Witten equation over a genus $g$, $n$-marked rigidified $r$-spin curve, of degree $d$. The moduli space can be subdivided as the disjoint union of moduli spaces
\begin{align*}
{\mc W}_{g, n}^d \big( \vec{\upgamma}, \vec{\upkappa} \big).
\end{align*}
Here $\upgamma = ( \upgamma_1, \ldots, \upgamma_n ) \in ( {\mb Z}_5 )^n$ describes the monodromies of the $r$-spin structure at the $n$ punctures; $\vec{\upkappa} = (\upkappa_{i_1}, \ldots, \upkappa_{i_b} )$ describes the asymptotics at the broad punctures where each $\upkappa_{i_\alpha}$ is a critical point of the Lagrange multiplier $\wt{W} = p(Q -a) + F$, or equivalently of the function $F_a:= F|_{Q^{-1}(a)}$. We assume that each ${\mc W}_{g, n}^d \left( \vec{\upgamma}, \vec{\upkappa} \right)$ has a good compactification, over which we have a well-defined virtual cycle. Then, we can define the virtual counting
\begin{align*}
\# {\mc W}_{g, n}^d \left( \vec{\upgamma}, \vec{\upkappa} \right) \in {\mb Q}
\end{align*}
which is zero if the virtual dimension of ${\mc W}_{g, n}^d \left( \vec{\upgamma}, \vec{\upkappa} \right)$ is not zero. The correlation function will just be a linear combination of the virtual numbers.

For each critical point $\upkappa$ of $F_a$, its unstable submanifold with respect to the flow of the real part of $F_a$ is a 4-dimensional cycle in $Q^{-1}(a)$ relative to infinity, denoted by $[\upkappa] \in H_4 ( Q^{-1}(a), \infty)$. We define the correlation
\begin{align}\label{equation18}
\big\langle \theta_1, \ldots, \theta_n \big\rangle_{g, n}^d := \sum_{\vec{\upgamma}} \sum_{ \vec{\upkappa}} \# {\mc W}_{g, n}^d \big( \vec{\upgamma}, \vec{\upkappa} \big) \big( \theta_{i_1}^* \cap [\upkappa_{i_1}] \big) \cdots \big( \theta_{i_b}^* \cap [\upkappa_{i_b}] \big).
\end{align}
Here we assume that each $\theta_i\in {\ms H}_{\rm GLSM}$ is homogeneous, i.e., coming from a single sector and if $\theta_i$ is a narrow state, then $\theta_i$ is the generator of the corresponding narrow sector. The first summation in (\ref{equation18}) runs over all possible combination of monodromies $\vec{\upgamma} =(\upgamma_1, \ldots, \upgamma_n )$ of an $r$-spin structure over a genus $g$, $n$-marked stable curve, such that if $\theta_i \in {\ms H}_k$, then $\upgamma_i = \upgamma^{(k)}$; the second summation runs over all combinations of critical points $\vec{\upkappa} = ( \upkappa_{i_1}, \ldots, \upkappa_{i_b})$ of $\wt{W}$; $\theta_{i_1}^*, \ldots, \theta_{i_b}^* \in H_4 \big( Q^{-1}(a); {\mb Q} \big)^{{\mb Z}_5})$ are the duals of the broad states $\theta_{i_1}, \ldots, \theta_{i_b}$. The correlator (\ref{equation17}) is then defined by extending (\ref{equation18}) linearly.

\subsection{Organization of the paper}

In Section \ref{section2}, we give the basic set-up of the gauged Witten equation, including the basic assumptions, and how to perturb the equation. In Section \ref{section3}, we consider the asymptotic behavior of bounded solutions to the perturbed gauged Witten equation. In Section \ref{section4} we study the linear Fredholm property of the perturbed Witten equation and compute the index of the linearized operator.

In Section \ref{section6}--\ref{section8}, we consider the compactification of the moduli space when the complex structure of the Riemann surface $\Sigma$ is fixed. In Section \ref{section6} we first define the stable objects which are possible geometric limits of a sequence of solutions, and then state the compactness theorem. In Section \ref{section7} we establish the energy quantization about blowing-up of solutions. In Section \ref{section8} we establish the uniform $C^0$-bound and prove the compactness theorem.	
 
In Appendix \ref{appendixa} we provide some basic analytical results which are used in this paper. In Appendix \ref{appenxib} we include some basic facts about equivariant topology.

\subsection{Acknowledgements}

We would like to thank Simons Center for Geometry and Physics for hospitality during our visit in 2013. We thank Kentaro Hori, David Morrison, Edward Witten for useful conversations on GLSM. The second author would like to thank Chris Woodward for helpful discussions. The revision of this paper were partially made during the second author's visit to Institute for Advanced Study and he would like to thank Helmut Hofer for hospitality.

\section{The gauged Witten equation and perturbations}\label{section2}

\subsection{The target space}\label{subsection21}

Let $(X, \omega, J)$ be a K\"ahler manifold and $Q: X \to {\mb C}$ is a holomorphic function, with a single critical point $\star \in X$. We assume that there exists a Hamiltonian $S^1$-action with moment map $\mu_0 : X \to {\bm i}{\mb R}$. Here we identify ${\bm i} {\mb R} \simeq {\rm Lie} S^1$ with its dual space by the standard metric on ${\mb R}$. Then for the generator ${\bm i}$ of ${\rm Lie} S^1$, we denote its infinitesimal action by ${\mc X}_0 \in \Gamma (TX)$. 

We suppose that the $S^1$-action extends to a holomorphic ${\mb C}^*$-action. We also assume that $Q$ is homogeneous of degree $r$, $r>1$ with respect to this ${\mb C}^*$-action. This means for $x \in X$ and $\xi \in {\mb C}^*$, 
\begin{align*}
Q(\xi x) = \xi^r Q(x).
\end{align*}
Let $X_Q:= Q^{-1}(0)$, which is smooth away from $\star$. For any $\upgamma \in {\mb Z}_r$, let $X_\upgamma \subset X$ be the fixed point set of $\upgamma$ and $\wt{X}_\upgamma= X_{\upgamma}\times {\mb C}$.

The GLSM target space is the product $\wt{X} = X \times {\mb C}$, whose coordinates are denoted by $(x, p)$. The factor ${\mb C}$ has the standard K\"ahler structure so that it induces a product K\"ahler structure, which, for simplicity, is still denoted by $(\omega, J)$ on $\wt{X}$. We lift the ${\mb C}^*$-action on $X$ trivially to $\wt{X}$. The {\bf superpotential} is the holomorphic function
\begin{align*}
W: \wt{X} \to {\mb C},\ W(x, p) = p Q(x).
\end{align*}
$W$ is also of degree $r$ with respect to the ${\mb C}^*$-action because $W(\xi (x,p)) = W(\xi x, p) = \xi^r W(x, p)$. 

On the other hand, let $G_1 = S^1$ and we consider the $G_1^{\mb C} = {\mb C}^*$-action on $\wt{X}$, given by
\begin{align*}
\zeta (x, p) = ( \zeta x, \zeta^{-r} p).
\end{align*}
$W$ is then $G_1^{\mb C}$-invariant. We use $G_0$ (resp. $G_0^{\mb C}$) to denote the copy of $S^1$ (resp. ${\mb C}^*$) which acts on $X$ and denote $G = G_0 \times G_1$ (resp. $G^{\mb C} = G_0^{\mb C} \times G_1^{\mb C}$). Then the $G$-action on $\wt{X}$ is Hamiltonian, with a moment map
\begin{align*}
\mu (x, p) = \Big( \mu_0 (x), \mu_0 (x) + {{\bm i}r \over 2}|p|^2 - \tau \Big).
\end{align*}
Here $\tau \in {\bm i} {\mb R}$ is a constant, which we fix from now on. We denote 
\begin{align*}
\mu_1(x, p) = \mu_0(x) + {{\bm i} r\over 2}|p|^2 - \tau
\end{align*}
which is a moment map for the $G_1$-action. Let ${\mf g}_i$ be the Lie algebra of $G_i$ for $i = 0, 1$ and ${\mf g} = {\mf g}_0 \oplus {\mf g}_1$. For any $\xi = (\xi_0, \xi_1) \in {\mf g}$, we denote by ${\mc X}_\xi = {\mc X}_{\xi_0} + {\mc X}_{\xi_1}\in \Gamma (T\wt{X})$ the infinitesimal action of $\xi$.

We make the following assumptions on the structures, which are all satisfied by the typical example of nondegenerate homogeneous polynomials on ${\mb C}^n$ of degree at least 2.
\begin{hyp}
\hspace{2em}
\begin{enumerate}
\item[({\bf X1})] $(X, \omega)$ is symplectically aspherical.

\item[({\bf X2})] The Riemannian curvature of $X$ is uniformly bounded; the complex structure $J$ is uniformly continuous on $X$ with respect to the K\"ahler metric in the sense of Definition \ref{defna1}.

\item[({\bf X3})] The moment map $\mu_0$ is proper and there exists $c >0$ such that for any $x \in X$,
\begin{align*}
{1\over c} {\bm i} \mu_0(x) -  c \leq \big| {\mc X}_0(x) \big|^2 \leq c {\bm i} \mu_0(x) + c.
\end{align*}

\item[({\bf X4})] As a real quadratic form on $TX$, we have
\begin{align*}0 \leq \nabla^2 \big( {\bm i}\mu_0 \big) \leq r. 
\end{align*}
\end{enumerate}
\end{hyp}\label{hyp21}

\begin{rem}
({\bf X1}) is imposed in order to simplify the proof of compactness; it can be removed. ({\bf X2}) is a bounded geometry at infinity assumption of $X$, which is used to prove the uniform $C^0$-bound of solutions (see Section \ref{section7}-\ref{section8}). The precise upper bound of ({\bf X4}) seems to be too strong but it is satisfied by all quasi-homogeneous polynomials on ${\mb C}^n$ of positive degrees. The condition ({\bf X3}) implies certain convexity about the geometry of $\wt{X}$ near infinity. \end{rem}

\begin{lemma}\label{lemma23}
For each $h \in {\bm i}{\mf g}$, we denote 
\begin{align}\label{equation21}
\big( |h|_{\wt{X}} \big)^2 = \big\| (e^h)^* \omega \|_{L^\infty(\wt{X})}.
\end{align}
There is a constant $c>0$ such that for any $h \in {\bm i} {\mf g}$,
\begin{align*}
\big( |h|_{\wt{X}} \big)^2 \leq c|h|.
\end{align*}
\end{lemma}
\begin{proof}
For any two tangent vector fields $Y, Z$, we have
\begin{align*}
\begin{split}
\big({\mc L}_{J{\mc X}} \omega \big)(Y, Z) = &\ Y \omega (J{\mc X}, Z) - Z \omega( J{\mc X}, Y) - \omega( J{\mc X}, [Y, Z]) \\
= &\ Z \langle {\mc X}, Y \rangle - Y \langle {\mc X}, Z \rangle + \langle {\mc X}, [Y, Z]\rangle \\
= &\ \langle Y, \nabla_Z {\mc X}\rangle - \langle Z, \nabla_Y {\mc X} \rangle \\
= &\ 2 \langle Y, \nabla_Z {\mc X}\rangle.
\end{split}
\end{align*}
The last inequality follows from the fact that ${\mc X}$ is Killing. Notice that $|\nabla {\mc X}| = |\nabla^2\mu|$. Then by ({\bf X4}) of Hypothesis \ref{hyp21}, $\big| {\mc L}_{J{\mc X}} \omega \big|$ is uniformly bounded throughout $\wt{X}$. Then the lemma follows from the fact that $J{\mc X}$ is the infinitesimal ${\bm i} {\mf g}$-action.
\end{proof}

For any $b \in (0,1)$, define ${\mc F}_{b}: \wt{X} \to {\mb R}$ by
\begin{align*}
{\mc F}_b(x, p):= \mu_0(x) \cdot \Big({ {\bm i} (1- b) \over r} \Big) + {b \over 2} |p|^2 = \mu \cdot \Big(  - {{\bm i} \over r}, {{\bm i}b \over r} \Big).
\end{align*}
\begin{lemma}\label{lemma24}
For any $b \in (0, 1)$, ${\mc F}_b: \wt{X} \to {\mb R}$ is a proper function and is bounded from below. Moreover, there exist a constant $c_0 >0$, a choice of $b_0 \in (0, 1)$ and $\lambda_0 >0$ such that 
\begin{align}\label{equation22}
\big\langle \nabla {\mc F}_{b_0}, J{\mc X}_{(\lambda_0 \mu_0, \mu_1)} \big\rangle \geq {1\over c_0} \big| \mu(u) \big|^2 - c_0.
\end{align}
\end{lemma}

\begin{proof}
The properness and the fact that ${\mc F}_b$ is bounded from below follow immediately from ({\bf X3}) of Hypothesis \ref{hyp21}. On the other hand, denoting $\rho  = |p|$, we have
\begin{align*}
\begin{split}
J{\mc X}_{\mu_0} (x, p) = &\ \Big(  -{\bm i} \mu_0(x)  J{\mc X}_0 (x) , 0 \Big),\\
J{\mc X}_{\mu_1} (x, p) = &\ \Big( \big( - {\bm i} \mu_0(x) + {r \over 2} |p|^2 + {\bm i} \tau \big) J {\mc X}_0(x),  r \big(  -{\bm i} \mu_0(x) + {r\over 2}|p|^2 + {\bm i} \tau      \big) \rho {\partial \over \partial \rho}  \Big).
\end{split}
\end{align*}
Then by ({\bf X3}) of Hypothesis \ref{hyp21}, we have
\begin{align*}
\begin{split}
\Big\langle \nabla {\mc F}_b, J{\mc X}_{\mu_0} \Big\rangle = &\ \Big\langle - \big({ 1- b \over r} \big)J{\mc X}_0, - {\bm i} \mu_0 J {\mc X}_0 \Big\rangle\\
= &\ \big({1-b \over r} \big) {\bm i}\mu_0 |{\mc X}_0|^2\\
\geq &\ \big( {1-b \over c r} \big) |{\mc X}_0|^4 - \big({1-b \over r} \big) |{\mc X}_0|^2;\\
\Big\langle \nabla {\mc F}_b, J{\mc X}_{\mu_1} \Big\rangle = &\ \Big\langle - \big({1- b \over r} \big) J {\mc X}_0 , \big( - {\bm i} \mu_0(x) + {r \over 2} |p|^2 + {\bm i} \tau \big) J {\mc X}_0(x) \Big\rangle\\
&\ + \Big\langle b \rho{\partial \over \partial \rho}, r \big(  -{\bm i} \mu_0(x) + {r\over 2}|p|^2 + {\bm i} \tau      \big) \rho {\partial \over \partial \rho} \Big\rangle\\
= &\ \big({1-b \over r}\big) \big( {\bm i} \mu_0 - {\bm i} \tau \big) |{\mc X}_0|^2 + r b |p|^4 - \big( {1-b \over 2} \big) |p|^2 |{\mc X}_0|^2 \\
&\ - r b {\bm i}\mu_0 |p|^2 + br {\bm i} \tau |p|^2\\
\geq &\ \big({1-b \over c r} \big) |{\mc X}_0|^4 + rb |p|^4 - \Big( \big({1- b \over 2}\big) + rb c \Big) |p|^2 |{\mc X}_0|^2 \\
&\ - \big({1- b \over r} \big) ( 1 + {\bm i} \tau) |{\mc X}_0|^2  + br ( {\bm i} \tau - c^2) |p|^2 .
\end{split}
\end{align*}
It suffices to choose $\lambda$ and $b$ so that the quartic part of $\langle \nabla {\mc F}_b, J{\mc X}_{(\lambda \mu_0, \mu_1)} \rangle$ is a positive definite form in $|{\mc X}_0|^2$ and $|p|^2$, i.e., to guarantee that the quadratic form
\begin{align*}
( 1+ \lambda) \big( {1- b \over c r} \big) A^2 - \Big( \big({1- b\over 2} \big) + rb c \Big) AB + rb B^2
\end{align*} 
is positive definite. This is equivalent to 
\begin{align*}
\Big( \big({1- b\over 2} \big) + rb c \Big) < 4 b ( 1+ \lambda) \big( {1- b \over c } \big).
\end{align*}
It holds for certain $b = b_0 \in (0, 1)$ and $\lambda = \lambda_0 >0$. Then $c_0>0$ exists.  \end{proof}

We fix $b_0$ and $\lambda_0 $ and denote ${\mc F}_{b_0}$ by ${\mc F}$. We use $\lambda_0 $ to define a metric on ${\mf g}$ as
\begin{align}\label{equation23}
\big| (\xi_0, \xi_1) \big|^2 = \lambda_0^{-1} \big| \xi_0 \big|^2 + \big| \xi_1 \big|^2.
\end{align}
This metric induces an identification ${\mf g} \simeq {\mf g}^*$ and $(\lambda_0 \mu_0, \mu_1)$ can be viewed as the dual of the moment map with respect to this metric. Then (\ref{equation22}) can be rewritten as
\begin{align}\label{equation24}
\big\langle \nabla {\mc F}, J {\mc X}_{\mu^*} \big\rangle \geq {1\over c_0} \big| \mu \big|^2 - c_0.
\end{align}

Now we give the assumptions on the function $Q$. 
\begin{hyp}\label{hyp25}
\hspace{2em}
\begin{enumerate}
\item[({\bf Q1})] There is a constant $c_Q >1 $ and a $G_0$-invariant compact subset $K_Q \subset X$ such that
\begin{align*} 
x \notin K_Q \Longrightarrow  {1\over c_Q} \left| \nabla^3 Q(x) \right| \leq \left| \nabla^2 Q(x) \right| \leq  c_Q \left| \nabla Q(x) \right|.
\end{align*}
Moreover, for every $\delta>0$, there exists $c_Q(\delta)>0$ such that 
\begin{align*} 
d(x, X_Q) \geq \delta, x \notin K_Q \Longrightarrow |\nabla Q(x)| \leq c_Q(\delta) |Q(x)|.
\end{align*}

\item[({\bf Q2})] For every $\upgamma \in {\mb Z}_r$, it is easy to see that $dQ$ vanishes along the normal bundle $N_{\upgamma} \to X_{\upgamma}$. We assume that the Hessian $\nabla^2 Q$ vanishes along $N_{\upgamma}$.
\end{enumerate}
\end{hyp}

\begin{rem}
The condition ({\bf Q2}) is not essential but it helps reduce the technicality in proving the asymptotic property of solutions in Section \ref{section3}.
\end{rem}

\begin{defn}
$\upgamma \in {\mb Z}_r$ is called {\bf broad} (resp. {\bf narrow}) if $X_\upgamma \neq \{\star\}$ (resp. $X_\upgamma = \{\star\}$).
\end{defn}

\begin{hyp}\label{hyp28}For any broad $\upgamma \in {\mb Z}_r$, there are a function $F_\upgamma: X \to {\mb C}$ and $a_\upgamma \in {\mb C}^*$ satisfying the following conditions.
\begin{enumerate}
\item[({\bf P}1)] $F_\upgamma$ can be written as 
\begin{align*}
F_\upgamma = \sum_{l=1}^{s-1} F_{\upgamma;l},\ (2 \leq s \leq r)
\end{align*}
where $F_{\upgamma;l}: X \to {\mb C}$ is a holomorphic function of degree $l$ with respect to the $G_0^{\mb C}$-action on $X$. The pull-back of $F_{\upgamma;l}$ to $\wt{X}$ is still denoted by $F_{\upgamma;l}$.

\item[({\bf P2)}] Each $F_{\upgamma,l}$ is $\upgamma$-invariant. It is easy to see that $dF_{\upgamma;l}$ vanishes along the normal bundle $N_\upgamma \to X_\upgamma$. We require that for every $l$, the Hessian $\nabla^2 F_{\upgamma;l}$ vanishes along $N_\upgamma$.

\item[({\bf P3})] For $j = 0, 1, \ldots$, there exist constants $c^{(j)} >0$ such that for $l = 1, \ldots, s-1$ and $x \in X$, 
\begin{align*} 
\Big| F_{\upgamma;l} (x) \Big| \leq c^{(0)} \Big( 1 + \big| \mu_0 (x) \big| \Big)^{1\over 2},\ \Big| \nabla^j F_{\upgamma;l} \Big| \leq c^{(j)}, j \geq 1.
\end{align*}

\item[({\bf P4})] The restriction of $F_\upgamma$ to $Q^{-1}(a_\upgamma) \cap X_{\upgamma}$ is a holomorphic Morse function. This is equivalent to saying that the Lagrange multiplier $p(Q_\upgamma - a_\upgamma ) + F_\upgamma$ is a holomorphic Morse function on $\wt{X}_\upgamma$.

\item[({\bf P5})] The perturbation has no critical point at infinity, in the following sense. There exist a compact subset $\wt{K}_\upgamma \subset \wt{X}$ and a constant $c_\upgamma > 0$ such that 
\begin{align*}
(x, p) \notin \wt{K}_\upgamma \Longrightarrow \Big| \nabla \big( W - a_\upgamma p + F_\upgamma \big)(x, p) \Big| \geq c_\upgamma.
\end{align*}
\end{enumerate}
\end{hyp}

\begin{rem}
The conditions ({\bf P1})--({\bf P5}) are modelled on linear functions on ${\mb C}^n$, when $Q$ is a nondegenerate quasi-homogeneous polynomial. Like ({\bf Q2}), the second part of the condition ({\bf P2}) is not essential and can be removed.
\end{rem}

Now for each broad $\upgamma$, we fix the choice of the perturbation data $(a_\upgamma, F_\upgamma)$. We denote $F_{\upgamma; s} = - a_\upgamma p: \wt{X} \to {\mb C}$. We introduced 
\begin{align*}
W_\upgamma':= -a_\upgamma p + \sum_{l=1}^{s-1} F_{\upgamma; l} =: \sum_{l=1}^{s} F_{\upgamma; l}.
\end{align*}
We also denote $F_0 (x, p) = W(x, p) = pQ(x)$. For notational purpose, if $\upgamma$ is narrow, we take $W_\upgamma' = \sum_{l=1}^s F_{\upgamma; l}$ to be the sum of $s$ zero functions. Then we denote
\begin{align}\label{equation25}
\wt{W}_\upgamma = W + W_\upgamma' = \sum_{l=0}^s F_{\upgamma; l}.
\end{align}

For each $t\in G_0^{\mb C} = {\mb C}^*$, we denote
\begin{align*}
\wt{W}_\upgamma^{(t)}(x, p) = t^r \wt{W}_\upgamma ( t^{-1} x, p)  = W(x, p) - t^r a_\upgamma p + \sum_{l=1}^{s-1} t^{r- l} F_{\upgamma; l}(x) =: \sum_{l=0}^s F_{\upgamma; l}^{(t)}.
\end{align*}
Most of the time we will consider the case $t \in {\mb R}_+$ and use $\delta$ instead of $t$. We have
\begin{align}\label{equation26}
(x, p)\in {\rm Crit} \big( \wt{W}_\upgamma |_{\wt{X}_\upgamma} \big) \Longleftrightarrow  (tx, p) \in {\rm Crit} \big(  \wt{W}_\upgamma^{(t)}|_{\wt{X}_\upgamma} \big).
\end{align}

For $l=1, \ldots, s-1$, we denote by $\rho_l: G^{\mb C} \simeq {\mb C}^* \times {\mb C}^* \to {\mb C}^*$ the character $(\xi_0, \xi_1) \mapsto (\xi_0^l, \xi_1^l)$; we denote by $\rho_{s}: G^{\mb C} \to {\mb C}^*$ the character which is trivial on the first ${\mb C}^*$-factor and is $\xi\mapsto \xi^{-r}$ on the second; we denote by $\rho_0: G^{\mb C} \to {\mb C}^*$ the character which is $\xi \mapsto \xi^r$ on the first ${\mb C}^*$-factor and is trivial on the second. Then each $F_{\upgamma;l}: \wt{X} \to {\mb C}$ above is $\rho_l$-equivariant for $l=0, \ldots, s$. 

\subsection{The domain}\label{subsection22}

\subsubsection*{Rigidified $r$-spin curves}

We recall the notion of rigidified $r$-spin curves following \cite[Section 2.1]{FJR2}.

Let $\Sigma$ be a compact Riemann surface and ${\bm z}= \{z_1, \ldots, z_k\}$ is a finite subset of punctures (marked points). We denote $\Sigma^*:= \Sigma \setminus {\bm z}$. We can attach orbifold charts near each puncture to obtain an {\bf orbicurve} ${\mc C}$. Suppose the local group of orbifold chart near each $z_j$ is $\Gamma_j$, which is canonically isomorphic to a cyclic group ${\mb Z}_{r_j}$. Then $\Sigma$ can be viewed as the ``desingularization'' of ${\mc C}$, also denoted by $|{\mc C}|$. There is a projection $\pi_{\mc C}: {\mc C} \to \Sigma$. The orbicurve ${\mc C}$ has the log-canonical bundle ${\mc K}_{\log}\simeq \pi_{\mc C}^* K_{\log}$, where $K_{\log}\to \Sigma$ is the bundle 
\begin{align*}
K_{\log} = K_\Sigma \otimes {\mc O}(z_1) \otimes \cdots \otimes {\mc O}(z_k).
\end{align*}

\begin{defn}\label{defn210} Fix $r \in {\mb Z}$, $r \geq 3$. An {\bf $r$-spin curve} is a triple $({\mc C}, {\mc L}, \upvarphi)$ where ${\mc C}$ is an orbicurve, ${\mc L}\to {\mc C}$ is an orbibundle, and 
\begin{align*} \upvarphi: {\mc L}^{\otimes r} \to {\mc K}_{\log}
\end{align*}
is an isomorphism of orbibundles.

A {\bf rigidification} of the $r$-spin structure $({\mc L}, \upvarphi)$ at $z_j$ is a choice of an element $e_j$ of ${\mc L}|_{z_j}$ such that
\begin{align*}
\upvarphi( e_j^{\otimes r}) = {dw \over w}.
\end{align*}

We denote a rigidification at $z_j$ by a map $\upphi_j: {\mb C}/ \Gamma_j \to {\mc L}|_{z_j}$. For a choice of rigidification $\upphi_j$ for each $j$, we call the tuple $\vec{{\mc C}}:= ({\mc C}, {\mc L}, \upvarphi; {\bm \upphi}):= ({\mc C}, {\mc L}, \upvarphi; \upphi_0, \ldots, \upphi_k)$ a {\bf rigidified $r$-spin curve}. \end{defn}

In this paper, from now on, we fix a rigidified $r$-spin curve $\vec{{\mc C}}= ({\mc C}, {\mc L}, \upvarphi; {\bm \upphi})$.

It is more convenient to look at rigidifications on the smooth curve $\Sigma$. Indeed, at each marked point $z_j$, the orbibundle ${\mc L}$ has its local monodromy, which is a representation $\Gamma_j \to S^1$. As a convention, we always assume that this representation is faithful. Then since ${\mc K}_{\log}$ always has trivial monodromy, we can view $\Gamma_j$ as a subgroup of ${\mb Z}_r$. So the generator of $\Gamma_j$ can be written as $\exp \left( {2\pi{\bm i} m_j \over r} \right)$, with $m_j \in \{0, 1, \ldots, r-1\}$. Then the $r$-spin structure induces an isomorphism
\begin{align*}
 |\upvarphi|: |{\mc L}|^{\otimes r} \to K_{\log} \otimes {\mc O}\Big( - \sum_{j=0}^k m_j z_j \Big)
\end{align*}
as usual line bundles over $\Sigma$, where $|{\mc L}| \to \Sigma$ is the desingularization of ${\mc L}$. Therefore, for any choice of local coordinate $w$ around $z_j$, a rigidification induces a choice of local frame $e_j$ of $|{\mc L}|$ near $z_j$ such that
\begin{align}\label{equation27}
|\upvarphi|(e_j^{\otimes r}) = w^{m_j} {dw \over w}.
\end{align}
We denote $\lambda_j = {\bm i} m_j/ r$ (resp. $\upgamma_j = \exp (2\pi \lambda_j)$) and call it the {\bf residue} (resp. {\bf monodromy}) of the $r$-spin structure at $z_j$. We define the type of the punctures.
\begin{defn}
A puncture $z_j$ is called {\bf narrow} (resp. {\bf broad}) if $\upgamma_j\in {\mb Z}_r$ is narrow (resp. broad).
\end{defn}

We take a smooth area form $\nu$ on the closed Riemann surface $\Sigma$. Then together with the complex structure, it determines a Riemannian metric, to which we will refer as the ``smooth metric''. On the other hand, for each $z_j$, we fix a holomorphic coordinate patch $w: B_1 \to \Sigma$ with $w(0) = z_j$ and use the $\log$ function to identify the punctured $U_j = w (B_1) \setminus \{z_j\}$ with the cylinder $\Theta_+:= [0, +\infty) \times S^1$. The latter has coordinates $s + {\bm i} t = - \log w$. We can choose a different area form $\nu'$ such that $\nu = \sigma \nu'$ where the conformal factor $\sigma: \Sigma^* \to {\mb R}_+$ is a smooth function whose restriction to each $\Theta_+$ is equal to $e^{-2s}$. The metric determined by $\nu'$ and the complex structure is called the ``cylindrical metric'' on $\Sigma^*$. 

From now on, for each puncture $z_j$, we fix the coordinate $w$ centered at $z_j$, the cylindrical end $U_j$, and its identification with $\Theta_+$. For any $S \geq 0$, we denote by $U_j(S) \subset \Sigma^*$ the subset identified with $[S, +\infty) \times S^1$. 

The cylindrical metric has injectivity radius bounded from below. We choose an $r^* \in (0, 1]$ such that for every point $q \in \Sigma^*$, there exists a holomorphic coordinate
\begin{align}\label{equation28}
z_q = s + {\bm i} t: B_{r^*}(q) \to B_{r^*} \subset {\mb C}.
\end{align}
such that $z_q(0) = q$. Here $B_{r^*}(q)$ is the $r^*$-neighborhood of $q$ with respect to the cylindrical metric. Then, for each such neighborhood $B_{r^*}(q)$, the area form $\Omega$ can be written as
\begin{align*}
\nu = \sigma_q (z_q) {{\bm i}\over 2} dz_q \wedge d\ov{z_q}.
\end{align*}
Then by shrinking $r^*$ properly, we have
\begin{align}\label{equation29}
\sup_{p \in \Sigma^*} \sup_{z_q \in B_{r^*}} \sigma_q (z_q) < \infty,\ \forall q\ \sup_{B_{r^*}} \sigma_q \leq  2 \inf_{B_{r^*}} \sigma_q.
\end{align}

We require that, if $B_{r^*}(q) \subset U_j$, then $z_q$ is the restriction of the cylindrical coordinate $ s+ {\bm i} t$ (after proper translation) to $B_{r^*}(q)$.

\subsubsection*{Adapted Hermitian metrics}

We define the weighted Sobolev space $W_\delta^{k, p}$ to be the Banach space completion of $C^\infty_0(\Sigma^*)$ with respect to the norm
\begin{align*}
\big\| f \big\|_{W_\delta^{k, p}(\Sigma^*)} := \big\| \sigma^{- {\delta \over 2}} f \big\|_{W^{k, p}(\Sigma^*)} 
\end{align*}
where the latter Sobolev norm is taken with respect to the cylindrical metric on $\Sigma^*$. It is similar to define the Sobolev space $W_\delta^{k, p}(U_j)$ for each cylindrical end $U_j$. We denote by ${\mz U}_\delta^{2, p}$ the space of $W_{loc}^{2, p}$-maps $g: \Sigma^* \to S^1$ such that for each $j$, 
\begin{align*}
g|_{U_j} = \exp \big( {\bm i}\xi_j \big),\ \xi_j \in W_\delta^{2, p}(U_j).
\end{align*}

From now on we fix $p>2$.
\begin{defn}
A $W_{loc}^{2,p}$-Hermitian metric $H$ on $|{\mc L}||_{\Sigma^*}$ is called {\bf adapted} if there is $\delta>0$ such that 
\begin{align*}
\log \big( |w|^{- {m_j \over r}} | e_j |_H \big) \in W_\delta^{2, p}(U_j).
\end{align*}
Here $m_j$ and $e_j$ are the ones in (\ref{equation27}). 
\end{defn}

Let ${\mz H}$ be the space of all adapted metrics on $|{\mc L}||_{\Sigma^*}$. There is an ${\mb R}_+$-action on ${\mz H}$ by rescaling a metric. For any $H \in {\mz H}$, denote by $P_0 (H)$ the $S^1$-frame bundle of $|{\mc L}|$ with respect to $H$. Then if $H$ is an adapted Hermitian metric, the Chern connection $A_0 (H)$ of $H$ will be a unitary connection on $P_0 (H)$ such that near each puncture, with respect to the trivialization determined by (\ref{equation27}), it can be written as
\begin{align*}
A_0 (H) = d + \lambda_j d t + \alpha_j
\end{align*}
and $\alpha_j$ is a purely imaginary valued 1-form on $U_j$ of class $W_\delta^{1, p}$ for some $\delta>0$. Note that the map $H \mapsto A_0 (H)$ is not injective but is constant on each ${\mb R}_+$-orbit of ${\mz H}$. 

Now choose (arbitrarily) a smooth element $H_0 \in {\mz H}$ as a reference, and consider the subset ${\mz H}_+ \subset {\mz H}$ consisting of metrics of the form $e^{2h_0} H_0$ with $h_0 \in W_\delta^{2, p}$ for some $\delta>0$. Then the map $H \mapsto A_0 (H)$ is injective on ${\mz H}_+$. We define
\begin{align*}
{\mz U} =\bigcup_{\delta>0} {\mz U}_\delta^{2, p},\ {\mz A}_0 = \left\{ g^* A_0 (H)\ |\  H \in {\mz H}_+,\ g\in {\mz U}  \right\}.
\end{align*}
Then every element of ${\mz A}_0$ has a unique expression as $g^* A_0 (H)$ for $g \in {\mz U}_{loc}^{2, p}$ and $H \in {\mz H}_+$.

\begin{rem}
It is necessary for us to remove the ${\mb R}_+$-action because it will cause trouble in proving compactness. On the other hand, we can release the restriction of ${\mz H}_+$ such that we can vary the value of $H$ at the punctures. However, those variations only form a finite dimensional degree of freedom, so they don't affect compactness and don't essentially change Fredholm property.
\end{rem}

Denote $P_0=P_0(H_0)$, which is the unit circle bundle of $|{\mc L}|$ with respect to the reference $H_0$. For any $H = e^{2h_0} H_0 \in {\mz H}_+$, there is a canonical isomorphism between $P_0 (H)$ and $P_0 (H_0)$, given by $v \mapsto e^{h_0} v$. Then any connection in ${\mz A}_0$ is transformed to a $S^1$-connection on $P_0$. We still denote this set of connections by ${\mz A}_0$. In particular, for every $A_0 \in {\mz A}_0$, the holomorphic line bundle structure of $P_0 \times_{S^1} {\mb C}$ determined by (the $(0,1)$-part of) $A_0$ is isomorphic to the holomorphic line bundle $|{\mc L}|$.  

Now we will choose a trivialization of $P_0$ on each $B_{r^*}(q)$ as well as on each $U_j$. On each $B_{r^*}(q)$, there is a local holomorphic section $e_q$ of $|{\mc L}|$ such that 
\begin{align}\label{equation210}
|\upvarphi|(e_q^{\otimes r}) = dz_q.
\end{align}
Here $z_q$ is the fixed one in (\ref{equation28}) and $e_q$ is unique up to a ${\mb Z}_r$-action and we just choose one of them. Then we trivialize $P_0$ over $B_{r^*}(q)$ by the local unitary frame 
\begin{align*}
\epsilon_q:= { e_q \over \|e_q\|_{H_0}}.
\end{align*}
This trivialization is denoted by 
\begin{align*}
\phi_{q, 0}: B_{r^*}(q)\times S^1 \to P_0 |_{B_{r^*}(q)}.
\end{align*}
On the other hand, on each cylindrical end $U_j$, there is a local holomorphic section $e_j$ of $|{\mc L}|$ such that 
\begin{align}\label{equation211}
|\upvarphi| (e_j^{\otimes r}) = w^{m_j} {dw \over w}.
\end{align}
$e_j$ is unique up to a ${\mb Z}_r$-action. Then we trivialize $P_0 |_{U_j}$ by the local unitary frame 
\begin{align*}
\epsilon_j:= { e_j \over \| e_j\|_{H_0} }.
\end{align*}
This trivialization is denoted by 
\begin{align*}
\phi_{j, 0}: U_j \times S^1 \to P_0 |_{U_j}.
\end{align*}

Now for each $A_0 \in {\mz A}_0$ and each $B_{r^*}(q)$ (resp. $U_j$), we define a function $h_0(A_0): B_{r^*}(q) \to {\rm Lie} {\mb C}^*$ (resp. $h_0 (A_0): U_j \to {\mb C}$) as follows. If $A_0 = g^* A_0(H)$ for $H \in {\mz H}_+$ and $g \in {\mz U}$, then for each $B_{r^*}(q)$ (resp. $U_j$), there is a unique ${\bm i} {\mb R}$-valued function $h_0':= h_0' (A_0)$ on $B_{r^*}(q)$ (resp. $U_j$) such that 
\begin{align}\label{equation212}
e^{h_0'} = g|_{B_{r^*}(q)},\  -{\bm i}h_0'(q)  \in [0, 2\pi),\ \Big( {\rm resp.}\  e^{h_0'} = g|_{U_j},\ \lim_{z \to z_j} h_0' (z) = 0 \Big).
\end{align}
On the other hand, we define 
\begin{align}\label{equation213}
h_0'':= h_0''(A_0) = {\bm i} \log \| e_q \|_{H}\ \Big( {\rm resp.}\  h_0''= {\bm i} \log \Big( \|e_j\|_{H} - |w|^{m_j \over r} \Big) \Big),
\end{align}
where $w = e^{-z}$ is the coordinate centered at $z_j$; then on either $B_{r^*}(q)$ or $U_j$, define
\begin{align}\label{equation214}
h_0:= h_0(A_0):= h_0' +  {\bm i} h_0''.
\end{align}

By the definition of the Chern connection and that of gauge transformation, on each $B_{r^*}(q)$, with respect to the trivialization $\phi_{q, 0}$ of $S|_{B_{r^*}(q)}$, if $A_0 \in {\mz A}_0$ is written as $A_0 = d + \phi_0 ds + \psi_0 dt$ for $\phi_0, \psi_0: B_{r^*}(q) \to {\bm i} {\mb R}$, then
\begin{align*}
\phi_0 = \partial_s h_0'- \partial_t h_0'',\ \psi_0 = \partial_s h_0'' + \partial_t h_0'.
\end{align*}
Similarly, if on $U_j$, $A_0 = \phi_0 ds + \psi_0 dt$, then
\begin{align*}
\phi_0 = \partial_s h_0'- \partial_t h_0'',\ \psi_0 - \lambda_j = \partial_s h_0'' + \partial_t h_0'.
\end{align*}
In either case, the curvature form of $A_0$ is equal to $\Delta h_0'' ds dt$. 

\subsubsection*{The $G_1$-bundle and connections}

We used $G_1$ to denote another copy of the group $S^1$ to distinguish from the structure group of $P_0$. We fix an arbitrary smooth $G_1$-bundle $P_1 \to \Sigma$. We denote its restriction to $\Sigma^*$ still by $P_1$ and denote by
\begin{align*}
P =  P_0 \times_{\Sigma^*} P_1 \to \Sigma^*
\end{align*}
the fibre product, which is a $G = G_0\times G_1 = S^1 \times S^1$-bundle over $\Sigma^*$. For each coordinate patch $B_{r^*}(q) \subset \Sigma^*$, we fix a trivialization $\phi_{q,1}: U_q \times G_1 \to P_1|_{U_q}$ arbitrarily. For each cylindrical end $U_j$ we can also take a trivialization $\phi_{j,1}: U_j \times G_1 \to P_1|_{U_j}$ which is the restriction of a local trivialization of $P_1$ near $z_j$. Together with the trivializations $\phi_{q, +}$ (resp. $\phi_{j, +}$), this gives a trivialization $\phi_q = (\phi_{q, 0}, \phi_{q, 1}): B_{r^*} (q) \times G \to P|_{B_{r^*}(q)}$ (resp. $\phi_j= (\phi_{j, 0}, \phi_{j, 1}): U_j \times G \to P|_{U_j}$). 

We denote ${\mz A}_1$ to be the space of $W_{loc}^{1, p}$-connections on $P_1|_{\Sigma^*}$ such that for each cylindrical end $U_j$, with respect to the trivialization of $P_1|_{U_j}$ induced from $\phi_{j, 1}$, any $A_1 \in {\mz A}_1$ can be written as $A_1 = d + \alpha_1$ where $\alpha_1$ is a ${\mf g}_1$-valued 1-form on $U_j$ of class $W_\delta^{1, p}$ for some $\delta>0$ (with respect to the cylindrical metric).

Now consider ${\mz A} = {\mz A}_0 \times {\mz A}_1$. This is a set of $G$-connections on $P$. For any $\delta>0$, denote by ${\mz G}_{1,\delta}^{2,p}$ the group of $G_1$-gauge transformations on $\Sigma^*$ of class $W_{\delta}^{2, p}$ and denote 
\begin{align*}
{\mz G}_1 = \bigcup_{\delta>0} {\mz G}_{1, \delta}^{2, p}, \ {\mz G} = {\mz U} \times {\mz G}_1.
\end{align*}
Then ${\mz G} = {\mz U} \times {\mz G}_1$ acts on ${\mz A}$ naturally. 

We would like to define functions similar to $h_0 (A_0)$ given by (\ref{equation212})--(\ref{equation214}). On $B_{r^*}(q)$, with respect to the trivialization $\phi_{q, 1}$, a $G_1$-connection $A_1 \in {\mz A}_1$ can be written as
\begin{align*}
A_1 = d + \phi_1 ds + \psi_1 dt,\ \phi_1, \psi_1: B_{r^*}(q) \to {\mf g}_1,
\end{align*}
where $s + {\bm i} t = z$ is the local coordinate. Then we define a function $h_1 = h_1' + {\bm i} h_1'':= h_1(A_1) =  h_1'(A_1) + {\bm i} h_1(A_1)'': B_{r^*}(q) \to {\mf g}^{\mb C}$ by the Cauchy integral formula 
\begin{align*}
h_1 (A_1)(z) = {1\over 4 \pi {\bm i}} \iint_{B_{r^*}(q)} \left(  { \phi_1 + {\bm i} \psi_1 \over \zeta - z} - {\phi_1 + {\bm i} \psi_1 \over \zeta} \right) d\zeta d\ov{\zeta}.
\end{align*}
Similarly, for $U_j$, we write $A_1$ as 
\begin{align*}
A_1 =d + \phi_1 ds + \psi_1 dt=  d + \vartheta dx + \varsigma dy
\end{align*}
where $w = x + {\bm i} y= e^{-z}$ is the smooth coordinate near $z_j$. Then we define 
\begin{align*}
h_1 (A_1)(z) = {1\over 4 \pi {\bm i}} \iint_{U_j} \left( {\vartheta + {\bm i} \varsigma \over w-z} - {\vartheta + {\bm i} \varsigma \over w }\right) dw d\ov{w}.
\end{align*}
Since $\phi_1 ds + \psi_1 dt$ is of class $W_\delta^{1, p}$ on $U_j$, we see that
\begin{align*}
\Big| \iint_{U_j} { \vartheta + {\bm i} \varsigma \over w} dw d\ov{w} \Big| \leq \iint_{U_j}| \phi_1 + {\bm i} \psi_1 | ds dt \leq \big\|\phi_1 + {\bm i} \psi_1 \big\|_{L_\delta^p(U_j) } \big\| e^{- \delta s} \big\|_{L^{p \over p-1}(U_j)} < \infty.
\end{align*}
Therefore $h_1$ is well-defined on $U_j$ and $\displaystyle \lim_{s \to +\infty} h_1(s, t) = 0$. On either $B_{r^*}(q)$ or $U_j$, we have
\begin{align}\label{equation215}
\phi_1 = \partial_s h_1' - \partial_t h_1'',\ \psi_0 = \partial_s h_1'' + \partial_t h_1'.
\end{align}
In particular, the curvature of $A_1$ is $F_{A_1} = \Delta h_1 (A_1)'' ds dt$.

Now for a connection $A  = (A_0, A_1)\in {\mz A}$, for $U$ being either $B_{r^*}(q)$ or $U_j$, we define
\begin{align}\label{equation216}
h_A:= (h_0, h_1)= (h_0(A_0), h_1(A_1)): U \to {\mf g}^{\mb C}.
\end{align}
This family of functions are useful when we do local analysis.

\subsubsection*{The fibre bundle}

Since $G$ acts on $\wt{X}$, we have the associated fibre bundle 
\begin{align*}
\pi: Y:= P\times_G \wt{X} \to \Sigma^*.
\end{align*}
The vertical tangent bundle $T^\bot Y \subset TY$ consists of vectors tangent to a fibre. Then since the $G$-action is Hamiltonian and preserves $J$, the K\"ahler structure on $\wt{X}$ induces a Hermitian structure on $T^\bot Y$. On the other hand, for any continuous connection $A$, the tangent bundle $TY$ splits as the direct sum of $T^\bot Y$ and the horizontal tangent bundle. The horizontal bundle is isomorphic to $\pi^* T\Sigma^*$, therefore the connection induces an almost complex structure on $Y$. Since $\wt{X}$ is K\"ahler, this almost complex structure is integrable and $Y$ becomes a holomorphic fibre bundle over $\Sigma^*$. 

We will consider sections of $Y$. A general smooth section is denoted by $u \in \Gamma(Y)$; more generally, we will consider sections $u \in \Gamma_{loc}^{1, p}(Y)$ of class $W^{1, p}_{loc}$. The group ${\mz G}$ also acts on the space of sections.

The trivialization $\phi_q: B_{r^*}(q) \times G \to P|_{B_{r^*}(q)}$ (resp. $\phi_j: U_j \times G \to P|_{U_j}$) induces a corresponding local trivialization of $Y$, which is denoted by the same symbol.

\subsection{The superpotential and gauged Witten equation}

\subsubsection*{The lift of the superpotential}

Using the $r$-spin structure $\upvarphi: {\mc L}^{\otimes r} \to {\mc K}_{\log}$ we can lift the potential function $W$ to the total space $Y$. More precisely, for each $B_{r^*}(q) \subset \Sigma^*$, let $(z_q, e_q)$ satisfy (\ref{equation210}). Let $\epsilon_{q, 1}$ be an arbitrary local frame of $P_1|_{B_{r^*}(q)}$. Then a point of $Y|_{B_{r^*}(q)}$ can be represented by $[e_q, \epsilon_{q, 1}, x]$ with the equivalence relation 
\begin{align*}
[g_0 e_q, g_1 \epsilon_{q, 1}, x] = [e_q, \epsilon_{q,1}, g_0 g_1 x],\ \forall x\in \wt{X},\ g_0 \in G_0^{\mb C},\ g_1 \in G_1.
\end{align*}
Then we define 
\begin{align*}
{\mc W}_{H_0} = \left( [e_q, \epsilon_{0, q}, x] \right) = W(x) dz_q.
\end{align*}
Then with respect to the unitary frame $\epsilon_q :=e_q/ \|e_q\|_{H_0}$ of $P_0$, we have 
\begin{align*}
{\mc W}_{H_0} \big( [\epsilon_q, \epsilon_{q, 0}, x] \big) = {\mc W}_{H_0} \big( [e_q, \epsilon_{q,0}, \|e_q\|_{H_0}^{-1} x] \big)= \left\| e_q \right\|_{H_0}^{-r} W(x) dz_q. 
\end{align*}
Then it is easy to see that the above definition is independent of the choice of the pair $(z_q, e_q)$ satisfying (\ref{equation210}) and the choice of the frame $\epsilon_{q, 0}$, so ${\mc W}_{H_0}$ is a well-defined section of the bundle $\pi^* K_\Sigma \to Y$. Moreover since $W$ is holomorphic we see that ${\mc W}_{H_0}$ is actually holomorphic with respect to the holomorphic structure on $Y$ induced from the $S^1$-connection $A_0 (H_0)$ and any $G_1$-connection $A_1$. 

Now let $H = e^{2h_0} H_0 \in {\mz H}_+$. Then we define
\begin{align*}
{\mc W}_{H} = e^{-r h_0} {\mc W}_{H_0} \in \Gamma \left( Y, \pi^* K_{\Sigma^*} \right).
\end{align*}
We see it is holomorphic with respect to the holomorphic structure on $Y$ induced from $A_0 (H)$ and any $G_1$-connection $A_1$. Moreover, for any connection $A = (A_0, A_1) \in {\mz A}$, we can express $A_0$ uniquely as $g_0^* A_0 (H)$ for some $g \in {\mz U}$ and $H\in {\mz H}_+$. Then we define
\begin{align*}
{\mc W}_A(y) = {\mc W}_{H} ( g y).
\end{align*}
Again, this is a section of $\pi^* K_\Sigma$ which is holomorphic with respect to the holomorphic structure on $Y$ induced from $A$. By the $G_1$-invariance of $W$, we also see that for any $g \in {\mz G}$, we have
\begin{align}\label{equation217}
{\mc W}_{g^* A}(y) = {\mc W}_A (gy).
\end{align}
On the other hand, using the trivialization $\phi_q: B_{r^*}(q) \times \wt{X} \to Y|_{B_{r^*}(q)}$, we have
\begin{align*}
{\mc W}_A \circ \phi_q(z, x) = e^{\rho_0(h_A(z))} W(x),
\end{align*}
where $\rho_0: G^{\mb C} \to {\mb C}^*$ is the character defined at the end of Subsection \ref{subsection21}. Similarly, the trivialization $\phi_j: U_j \times G \to P|_{U_j}$ induces a trivialization $\phi_j: U_j \times X \to Y|_{U_j}$, and
\begin{align*}
{\mc W}_A \circ \phi_j(z, x) = e^{\rho_0(h_A(z))} W(e^{\lambda_j t} x) dz = e^{\rho_0(h_A(z)) + r\lambda_j t} W(x) dz.
\end{align*}

\subsubsection*{The gauged Witten equation}

The vertical differential of ${\mc W}_A$ is a section
\begin{align*}
d{\mc W}_A \in \Gamma \Big( Y, \pi^* K_{\Sigma^*} \otimes \big( T^\bot Y \big)^* \Big).
\end{align*}
The vertical Hermitian metric on $T^\bot Y$ induces a conjugate linear isomorphism $T^\bot Y \simeq \left( T^\bot Y \right)^*$. On the other hand, the complex structure on $\Sigma^*$ induces a conjugate linear isomorphism $K_{\Sigma^*} \simeq \Lambda^{0,1}T^* \Sigma^*$. Therefore we have a conjugate linear isomorphism
\begin{align*}
\pi^* K_{\Sigma^*} \otimes \big( T^\bot Y \big)^* \simeq \pi^* \Lambda^{0,1}_{\Sigma^*} \otimes T^\bot Y.
\end{align*}
The image of $d{\mc W}_A$ under this map is called the {\bf vertical gradient} of ${\mc W}_A$, denoted by 
\begin{align*}
\nabla {\mc W}_A \in \Gamma \Big( Y, \pi^* \Lambda^{0,1}_{\Sigma^*} \otimes T^\bot Y \Big).
\end{align*}

Now we can write down {\bf the gauged Witten equation}. It is the following system on the pair $(A, u)$, where $A \in {\mz A}$ and $u \in \Gamma_{loc}^{1, p}(Y)$:
\begin{align}\label{equation218}
\left\{ \begin{array}{ccc}
\ov\partial_A u + \nabla {\mc W}_A(u) & = & 0;\\
* F_A + \mu^* (u) & = & 0.
\end{array} \right.
\end{align}
Each term in the system is defined as follows: the connection $A$ induces a continuous splitting $TY \simeq T^\bot Y \oplus \pi^* T \Sigma^*$ and $d_A u \in W^{1, p}_{loc}(T^* \Sigma^* \otimes u^* T^\bot Y )$ is the covariant derivative of $u$; the $G$-invariant complex structure $J$ induces a complex structure on $T^\bot Y$ and $\ov\partial_A u$ is the $(0, 1)$-part of $d_A u$ with respect to this complex structure. $\nabla {\mc W}_A (u)$ is the pull-back of $\nabla {\mc W}_A$ by $u$, which lies in the same vector space as $\ov\partial_A u$. $F_A\in \Omega^2(\Sigma^*) \otimes {\mf g}$ is the curvature form of $A$, $*: \Omega^2(\Sigma^*) \to \Omega^0(\Sigma^*)$ is the Hodge-star operator with respect to the smooth metric on $\Sigma$; the moment map $\mu$ lifts to a ${\mf g}$-valued function on $Y$ and $\mu^*(u)$ is the dual of $\mu(u)$ with respect to the metric defined by (\ref{equation23}).

By (\ref{equation217}) and the fact that the $G$-action is Hamiltonian, the gauged Witten equation is ${\mz G}$-invariant, in the sense that for any $(A, u)\in {\mz A} \times \Gamma_{loc}^{1, p}(Y)$ and any $g\in {\mz G}$, we have
\begin{align}\label{equation219}
\begin{split}
\ov\partial_{g^*A} (g^* u) + \nabla {\mc W}_{g^* A}( g^* u) = &\ \left( g^{-1}\right)_* \left( \ov\partial_A u + \nabla {\mc W}_A (u) \right),\\
* F_{g^* A} + \mu^* ( g^* u) = &\  {\rm Ad}_g^{-1} \left( * F_A + \mu^*(u) \right).
\end{split}
\end{align}

\begin{rem}
In this paper, all vector fields are regarded as real vector fields. So for a holomorphic function $F: X \to {\mb C}$, its gradient $\nabla F$ is the gradient of the real part of $F$.
\end{rem}

\subsection{Perturbation}

The function $W: \wt{X} \to {\mb C}$ has highly degenerate critical points. The degeneracy will cause the problem that the linearized equation doesn't give a Fredholm operator, in the presence of broad punctures. The usual way to deal with this situation is to perturb the potential ${\mc W}_A$ near the broad punctures, which is already adopted in the study of Witten equation in \cite{FJR3}. We will use the functions $F_\upgamma$ given in Hypothesis \ref{hyp28} to perturb the superpotential. 

\subsubsection*{A bounding functional on ${\mz A}$}

In this subsection we would like to construct a {\it smooth} functionals on ${\mz A}$ which can control certain Sobolev norms. The purpose of having such bounding functionals is to give uniform energy bound on solutions with fixed topological type (see the proof of Theorem \ref{thm44}).

\begin{defn}\label{defn215} For each $A = (A_0, A_1) \in {\mz A} = {\mz A}_0 \times {\mz A}_1$ and for each broad puncture $z_j$, we define
\begin{align}\label{equation220}
m_{j,A} = \sum_{l=1}^{s} \big\| e^{\rho_l (h_{j,A}) } \big\|_{L^2(U_j \setminus U_j (2))} + 1,\ \delta_{j, A} = \big( m_{j, A}\big)^{-1}.
\end{align}
Here $h_{j, A}: U_j \to {\mf g}^{\mb C}$ is the function defined by (\ref{equation216}). If $z_j$ is narrow, we define $\delta_{j, A} = 1$. 
\end{defn}

$\delta_{j, A}$ only depends on the gauge equivalence class of $A$ because a gauge transformation only changes the real part of $h_{j,A}$. Moreover, the function $A \mapsto \delta_{j,A}$ is smooth in $A \in {\mz A}$. Indeed, the map $A \mapsto h_{j,A}$ is smooth; it follows with the restriction to $U_j \setminus U_j(2)$, and a Sobolev embedding $L_1^p \to C^0$, which are both linear, hence smooth. Now $C^0(U_j \setminus U_j(2))$ is a Banach algebra, so the exponential map is smooth. It is followed by taking the $L^2$-norm of a nonzero continuous function, which is smooth.

\subsubsection*{The perturbed gauged Witten equation}

For each broad puncture $z_j$, we can lift $W_{\upgamma_j}'= \sum_{l=1}^s F_{\upgamma_j,l}$ to $Y|_{U_j}$. The trivialization $\phi_j$ gives the local frame $\epsilon_j$ of $P|_{U_j}$. We define
\begin{align*}
\begin{array}{cccc}
\displaystyle {\mc W}_{j,A}': & Y|_{U_j} & \to & ( T^* U_j )^{1,0}\\[0.2cm]
\displaystyle          & \big( [ \epsilon_j, x ] \big) & \mapsto &   \displaystyle \Big( \sum_{l=1}^{s} e^{\rho_l(h_{j,A} + \lambda_j t)} F_{\upgamma_j,l}^{(\delta_j)} ( x) \Big) {dw \over w}.
														\end{array}
														\end{align*}
Indeed, $e^{h_{j,A} + \lambda_j t} \epsilon_j$ gives a local frame of $P^{\mb C}|_{U_j}$ which is holomorphic with respect to $A$, and we have
\begin{align*}
{\mc W}_{j,A}'\Big( [e^{h_{j,A} + \lambda_j t} \epsilon_j, x ] \Big) =\Big( \sum_{l=1}^s F_{\upgamma_j,l}^{(\delta_j)} (x) \Big) {dw \over w}.
\end{align*}
This expression shows that ${\mc W}_{j,A}'$ is holomorphic with respect to the connection $A$. 

For each broad puncture $z_j$, fix a cut-off function $\beta_j$ supported in $U_j$ and $\beta|_{U_0(2)} \equiv 1$ such that with respect to the cylindrical metric,
\begin{align*}
\big| \nabla \beta_j \big| \leq 1,\ \big| \nabla^2 \beta_j \big| \leq 1
\end{align*}
Denote $\beta = \sum_{z_j\ {\rm broad}} \beta_j$. We define
\begin{align*}
\wt{\mc W}_A = {\mc W}_A + \sum_{z_j\ {\rm broad}} \beta_j {\mc W}_{j,A}'.
\end{align*}
$\wt{\mc W}_A$ is only vertically holomorphic and is holomorphic outside the support of $d\beta$. Then the {\bf perturbed gauged Witten equation} is
\begin{align}\label{equation221}
\left\{ \begin{array}{ccc}
\ov\partial_A u + \nabla \wt{\mc W}_A(u) & = & 0;\\
* F_A + \mu^* (u) & = & 0.
\end{array}\right.
\end{align}
Similar to the unperturbed case, the perturbed gauged Witten equation is gauged invariant in a similar sense as (\ref{equation219}).

\subsection{Energy} 

For $(A, u)\in {\mz A} \times \Gamma_{loc}^{1,p}(Y)$, we define its energy as
\begin{align}
E (A, u) = {1\over 2} \Big( \big\| d_A u \big\|_{L^2(\Sigma^*)}^2 + {1\over 2} \big\| F_A \big\|_{L^2(\Sigma^*)}^2 + \big\| \mu(u) \big\|_{L^2(\Sigma^*)}^2\Big)  +  \big\| \nabla \wt{\mc W}_A (u) \big\|_{L^2(\Sigma^*)}^2.
\end{align}
Here the Sobolev norms are taken with respect to the smooth metric on $\Sigma$. The sum of the first two terms is sometimes referred to as the kinetic energy, and the sum of the last two terms is sometimes referred to as the potential energy.

This energy functional generalizes the Yang-Mills-Higgs functional used in gauged Gromov-Witten theory, which can be viewed as a special case of our setting where $W = 0$. One can compare it with the bosonic part of the supersymmetric action in \cite{Witten_LGCY}.

\subsection{Regularity}

\begin{prop}\label{prop216}
Suppose $(A, u) \in {\mz A} \times \Gamma_{loc}^{1, p}(Y)$ is a solution to (\ref{equation221}). Then there exists a gauge transformation $g\in {\mz G}$ such that $g^* (A, u)$ is smooth.
\end{prop}

\begin{proof}
Suppose $A \in {\mz A}_\delta^{1, p}$ for some $\delta>0$. Let $d^*$ be the dual of $d$ with respect to the cylindrical metric and $\Delta = - d^* d$. Then $\Delta: W_\delta^{2, p}(\Sigma^*) \otimes {\mf g} \to L_\delta^p(\Sigma^*) \otimes {\mf g}$ is Fredholm. Therefore we can find a smooth element $A' \in {\mz A}$ such that $d^* (A - A')\in L_\delta^p(\Sigma^*) \otimes {\mf g}$ lies in the range of $\Delta$. Choose $h \in \Delta^{-1} \big( d^* (A - A')\big)$ and denote $g = \exp h \in {\mz G}$. Then 
\begin{align*}
d^* \big( g^* A - A' \big) = - \Delta h + d^* (A - A') = 0.
\end{align*} 
This means that $g^* A$ is in Coulomb gauge relative to $A'$. Let $\alpha = g^* A - A'$. Then
\begin{align*}
\ov\partial_{A'} u = - \nabla \wt{\mc W}_A (u) - \big({\mc X}_{\alpha} (u) \big)^{0,1},\ d\alpha = - \mu^* (u) \nu - F_{A'}.
\end{align*}
Apply the standard elliptic bootstrapping argument to the pair $(\alpha, u)$ we see that $(\alpha, u)$ is indeed smooth.
\end{proof}

So it suffices to consider smooth solutions to the perturbed gauged Witten equation. 

\section{Local and cylindrical models of gauged Witten equation}

\begin{defn}
Let $r >0$. The parameters of {\bf local models of gauged Witten equation} over $B_r$ are triples $(\beta, \sigma, \delta)$, where: $\beta: B_r \to [0,1]$, $\sigma: B_r \to [0, +\infty)$ are smooth functions satisfying
\begin{align*}
|d\beta| \leq 1,\ \sigma \leq C(\sigma), \ \sigma^+:= \sup_{B_r} \sigma \leq 2 \sigma^-:= 2\inf_{B_r} \sigma
\end{align*}
and $\delta \in (0, 1]$ is a constant. A solution to the local model with parameter $(\beta, \sigma, \delta)$ is a pair ${\bm u} = (u, h = h' + {\bm i} h'') \in C^\infty( B_r, \wt{X} \times {\mf g}^{\mb C})$ solving the equation
\begin{align}\label{equation31}
\partial_s u + {\mc X}_\phi(u) + J \big( \partial_t u + {\mc X}_\psi (u) \big) + 2 \nabla \wt{W}_h^{(\delta)}(u) = 0,\ \Delta h''  + \sigma \mu^* (u) = 0.
\end{align}
Here $(\phi, \psi): B_r \to {\mf g} \times {\mf g}$ is given by $\phi + {\bm i} \psi = 2(\partial h / \partial \ov{z})$,
\begin{align*}
\wt{W}_h^{(\delta)}(z, x) = e^{\rho_0(h(z))} F_0(x) + \beta(z) \sum_{l=1}^s e^{\rho_l(h(z))} F_l^{(\delta)}(x).
\end{align*}
Here $\wt{W} = \sum_{l=0}^s F_l$ is equal to the function $\wt{W}_\upgamma$ we specified in (\ref{equation25}) for some $\upgamma \in {\mb Z}_r$.
\end{defn}

Let $\Theta_+ = [0, +\infty) \times S^1$ and let $\sigma: \Theta_+ \to {\mb R}_+ \cup \{0\}$ be a smooth function satisfying
\begin{align}\label{equation32}
\big| \nabla^j \sigma(s, t) \big| \leq C^{(j)}(\sigma) e^{-2s},\ j = 0, 1, \ldots.
\end{align}

\begin{defn}
Let $\lambda \in  {\bm i} ({\mb Z}/r \cap [0, 1)) \subset {\mf g}_0$ and $\upgamma = \exp (2\pi \lambda)$. The parameters of {\bf cylindrical models of gauged Witten equation with residue $\lambda$} (called $\lambda$-cylindrical model for short) is a pair $(\sigma, \delta)$, where: $\sigma: \Theta_+ \to {\mb R}_+ \cup\{0\}$ is a smooth function satisfying (\ref{equation32}) and $\delta\in (0, 1]$ is a constant, such that if $\upgamma$ is narrow, then $\delta = 1$. 

A smooth solution to a $\lambda$-cylindrical model with parameter $(\sigma, \delta)$ is a map ${\bm u} = (u, h) \in C^\infty( \Theta_+, \wt{X} \times {\mf g}^{\mb C})$ which solves
\begin{align}\label{equation33}
\partial_s u + {\mc X}_\phi(u) + J \big( \partial_t u + {\mc X}_{\psi} \big) + 2 \wt{W}_{h, \lambda}^{(\delta)}(u) = 0,\ \Delta h'' + \sigma \mu^*(u) = 0.
\end{align}
and which satisfies
\begin{align}\label{equation34}
\lim_{s \to +\infty} h(s, t) = 0,\ \phi \in W_\delta^{1, p}(\Theta, {\mf g})\ {\rm for\ some\ }\delta>0.
\end{align}
Here $(\phi, \psi): \Theta_+ \to {\mf g} \times {\mf g}$ is given by $\phi + {\bm i}(\psi - \lambda) = 2( \partial h / \partial \ov{z})$ and 
\begin{align}\label{equation35}
\wt{W}_{h, \lambda}^{(\delta)}(z, x) = \sum_{l=0}^s e^{\rho_l(h(z) + \lambda t)} F_{\upgamma;l}^{(\delta)}(x)
\end{align}
where the function $\sum_{l=0}^s F_{\upgamma; l} =\wt{W}_\upgamma$ is the one we specified in (\ref{equation25}).
\end{defn}

For either the local model or the cylindrical model, we have the coordinate $z = s + {\bm i} t$. If ${\bm u} = (u, h)$ is a solution to either type of model, we abbreviate $A = \phi ds + \psi dt$, $d_A u = (\partial_s u + {\mc X}_\phi(u)) ds + (\partial_t u + {\mc X}_{\psi}(u)) dt$ and $\wt{\mc W}_A = \wt{W}_h^{(\delta)}$. The energy density of ${\bm u}$ is 
\begin{align}
e({\bm u}) = {1\over 2} \big| d_A u \big|^2 + \big| \nabla \wt{\mc W}_A(u) \big|^2 + \big| \sqrt{\sigma} \mu(u) \big|^2.
\end{align}
The total energy $E({\bm u})$ is the integral of $e({\bm u})$ over the domain (either $B_r$ or $\Theta_+$).

Suppose $(A, u)$ is a solution to the perturbed gauged Witten equation (\ref{equation221}) over the rigidified $r$-spin curve $\vec{\mc C}$. Then for any $q \in \Sigma^*$, the restriction of $(A, u)$ to $B_r(q)$ ($r \leq r^*$) gives a solution to a corresponding local model, via the local coordinate on $B_r(q)$ and the local trivialization of $P|_{B_r(q)}$. Moreover, if the puncture $z_j$ has residue $\lambda_j$, then the restriction of $(A, u)$ to any $U_j(S)\simeq \Theta_+$ ($S \geq 1$), via the trivialization $\phi_j$, gives a solution to a solution to a $\lambda_j$-cylindrical model. 

On the other hand, suppose $(u, h)$ is a solution to a $\lambda$-cylindrical model with parameter $(\sigma, \delta)$. Then on any disk $B_r \subset \Theta_+$, the function $\lambda t: B_r \to {\mf g}$ is single-valued and the pair $(u|_{B_r}, h|_{B_r} + \lambda t)$ is a solution to the local model over $B_r$ with parameter $(\beta = 1, \sigma|_{B_r}, \delta)$. 

Suppose ${\bm u} = (u, h)$ is a solution to a local model over $B_r$ and $f: U \to {\mf g}$ is a smooth function. Let $g = \exp f$. Then $g^*{\bm u} := (g^{-1} u, h + f)$ is another solution to the original model. We simply say that ${\bm u}$ and $g^* {\bm u}$ are gauge equivalent.

\subsection{Digress: holomorphic 1-forms}\label{subsection31}

To carry out local calculations, we make a digress on properties of holomorphic 1-forms on K\"ahler manifolds. These properties apply to the case of holomorphic functions directly. Within this subsection, $\wt{X}$ is a general K\"ahler manifold and $G$ is a compact Lie group. Assume that there is a Hamiltonian $G$-action on $\wt{X}$, which extends to a holomorphic $G^{\mb C}$-action. 

We say a holomorphic 1-form $\alpha$ on $\wt{X}$ is homogeneous with respect to a character $\rho: G^{\mb C} \to {\mb C}^*$ if for any $g\in G^{\mb C}$, $g^* \alpha = \rho(g) \alpha$. In particular, if $f$ is a homogeneous function, then $df$ is homogeneous with respect to the same character. On the other hand, for a holomorphic 1-form, $\alpha$, we define its metric dual $\alpha^*$ to be the {\it real} vector field satisfying that for any real vector field $Z$, 
\begin{align*}
\langle \alpha^*, Z\rangle = {\rm Re} \left( \alpha(Z)\right).
\end{align*}

\begin{lemma}\label{lemma33}
If $\alpha$ is a holomorphic 1-form which is homogeneous with respect to $\rho: G^{\mb C} \to {\mb C}^*$, then for any $\xi \in {\mf g}$,
\begin{align}\label{equation37}
\left[ \alpha^*, {\mc X}_\xi \right] = \rho(\xi) \alpha^*.
\end{align}
Moreover, for any real vector field $Z$, 
\begin{align}\label{equation38}
\left[ \nabla_Z \alpha^*, {\mc X}_\xi \right] = \rho(\xi) \nabla_Z \alpha^* + \nabla_{[Z, {\mc X}_\xi]} \alpha^*.
\end{align}
\end{lemma}

\begin{proof}
The homogeneity of $\alpha$ implies that ${\mc L}_{{\mc X}_\xi} \alpha = \rho(\xi) \alpha$. Therefore, for any $G$-invariant vector field $Z$, since ${\mc X}_\xi$ is Killing, we have 
\begin{align*}
\big\langle Z, [ \alpha^*, {\mc X}_\xi ] \big\rangle = - {\mc L}_{{\mc X}_\xi} \big( {\rm Re} ( \alpha(Z)) \big) = - \big({\rm Re}  {\mc L}_{{\mc X}_\xi} \alpha \big)(Z) = - {\rm Re} \big( \rho(\xi) \alpha \big) (Z) = \big\langle Z, \rho (\xi) \alpha^* \big\rangle.
\end{align*}
Therefore $[ \alpha^*, {\mc X}_\xi ] = \rho(\xi) \alpha^*$. To prove the second equality, we may assume that $Z$ is $G$-invariant. Then take another $G$-invariant vector field $Z'$, we see 
\begin{align*}
\begin{split}
\big\langle Z', [ \nabla_Z \alpha^* , {\mc X}_\xi ] \big\rangle = &\ - {\mc X}_\xi \big\langle Z', \nabla_Z \alpha^* \big\rangle \\
= &\ - {\mc X}_\xi Z \big\langle Z, \alpha^*  \big\rangle + {\mc X}_\xi \big\langle \nabla_Z Z', \alpha^*  \big\rangle \\
= &\ -Z {\mc X}_\xi \big\langle Z', \alpha^*  \big\rangle + {\mc X}_\xi \big\langle \nabla_Z Z', \alpha^*  \big\rangle \\
= &\ Z \big\langle Z', \rho (\xi) \alpha^*  \big\rangle -  \big\langle \nabla_Z Z', \rho (\xi) \alpha^* \big\rangle \\
= &\ \big\langle Z', \nabla_Z (\rho (\xi)) \alpha^*  \big\rangle \\
= &\ \big\langle Z', \rho (\xi) \nabla_Z \alpha^* \big\rangle.
\end{split}
\end{align*}
Therefore (\ref{equation38}) holds.
\end{proof}

\begin{lemma}\label{lemma34}
Suppose $\alpha$ is a holomorphic 1-form and $\alpha^*$ is the real vector field defined by $\langle \alpha^*, Z \rangle = {\rm Re} (\alpha(Z))$. Then we have
\begin{align*}
\nabla_{JZ} \alpha^* = - J \nabla_Z \alpha^*.
\end{align*}
\end{lemma}

\begin{proof}
It suffices to prove for any tangent vector $V$, we have
\begin{align*}
\langle \nabla_{JZ} \alpha^*, V\rangle = \langle \nabla_Z \alpha^*, J V \rangle.
\end{align*}
Indeed, the equality is bilinear in $Z$ and $V$. So it suffices to consider the case when $[Z,V] = [JZ, V] = 0$. In this case we have
\begin{align*}
\begin{split}
\big\langle \nabla_{JZ} \alpha^*, V  \big\rangle  = &\  JZ  \big\langle \alpha^*, V \big\rangle - \big\langle \alpha^*, \nabla_{JZ} V \big\rangle\\
 =&\ {\rm Re} \big( JZ \alpha(V) - \alpha (J \nabla_Z V) \big) \\
= &\ {\rm Re} \big( d\alpha (JZ, V) + V \alpha(JZ) - {\bm i} \alpha(\nabla_Z V) \big) \\
= &\ {\rm Re} \big( {\bm i} d\alpha(Z, V) - {\bm i} V \alpha(Z)  - {\bm i} \alpha( \nabla_Z V) \big) \\
= &\ {\rm Re} \big( Z \alpha(JZ) -  \alpha (\nabla_Z (JV))\big) \\
= &\  Z \big\langle \alpha^*, JV  \big\rangle - \big\langle \alpha^*, \nabla_Z (J V) \big\rangle \\
= &\ \big\langle \nabla_Z \alpha^*, JV \big\rangle.
\end{split}
\end{align*}
Here in the third and fourth equalities we used the fact that $\alpha$ is holomorphic.
\end{proof}

\subsection{Local calculations}\label{subsection32}

Consider an arbitrary smooth map $(u, \phi, \psi): U \to \wt{X}\times {\mf g} \times {\mf g}$ and denote $A = \phi ds + \psi dt$. Here $U$ is a region which is either $B_r$ or $\Theta_+$, having a coordinate $z = s + {\bm i} t$. We recall certain differential operators naturally associated to the triple $(u, \phi, \psi)$ (cf. \cite{Cieliebak_Gaio_Mundet_Salamon_2002} and \cite{Gaio_Salamon_2005} for more comprehensive treatment of such operators). For any $\xi \in \Gamma  (U, u^* T\wt{X})$, we define
\begin{align*}
D_{A, s} \xi = \nabla_s \xi + \nabla_\xi {\mc X}_\phi,\ D_{A, t} \xi = \nabla_t \xi + \nabla_\xi {\mc X}_\psi.
\end{align*}
We list some of their properties whose proofs can be found in \cite[Section 2.4]{Cieliebak_Gaio_Mundet_Salamon_2002} and \cite[Section 4]{Gaio_Salamon_2005}. 

(I) For $\xi_1, \xi_2 \in \Gamma( U, u^* T\wt{X} )$, let $\langle \xi_1, \xi_2 \rangle$ be the real inner product on $T\wt{X}$. Then
\begin{align}\label{equation39}
\begin{split}
\partial_s \langle \xi_1, \xi_2 \rangle =&\ \langle D_{A, s} \xi_1, \xi_2 \rangle + \langle \xi_1, D_{A, s} \xi_2 \rangle,\\
\partial_t \langle \xi_1, \xi_2 \rangle =&\ \langle D_{A, t} \xi_1, \xi_2 \rangle + \langle \xi_1, D_{A, t} \xi_2 \rangle.
\end{split}
\end{align}

(II) Since $J$ is integrable and $G$-invariant, we have
\begin{align}\label{equation310}
[ D_{A, s}, J ] = [ D_{A, t}, J ] = 0.
\end{align}

(III) Let $R$ be the curvature tensor of $\wt{X}$ and $F_A = \partial_s \psi - \partial_t \phi$. Denote $v_s = \partial_s u + {\mc X}_\phi$, $v_t = \partial_t u + {\mc X}_\psi$, then for $\xi \in \Gamma ( U, u^* T\wt{X} )$, we have
\begin{align}\label{equation311}
D_{A, s} v_t - D_{A, t} v_s = {\mc X}_{F_A};
\end{align}
\begin{align}\label{equation312}
[ D_{A, s}, D_{A, t} ] \xi = R(v_s, v_t) \xi + \nabla_\xi {\mc X}_{F_A}.
\end{align}

Let $h = h'+ {\bm i} h'': U \to {\mf g}^{\mb C}$ be a smooth function such that $\phi + {\bm i} \psi = 2 ( \partial h/ \partial \ov{z})$. As an application of Lemma \ref{lemma33} and \ref{lemma34}, we have
\begin{lemma}\label{lemma35}
Let $F: \wt{X} \to {\mb C}$ be a homogeneous function with respect to a character $\rho: G^{\mb C} \to {\mb C}^*$. Then we have
\begin{align*}
\begin{split}
D_{A, s} e^{\ov{\rho(h)}} \nabla F(u) = &\ e^{\ov{\rho(h)}} \big( \nabla_{v_s} \nabla F + 2 \rho ( {\bm i} \partial h'' / \partial \ov{z} ) \nabla F \big),\\ 
D_{A, t} e^{\ov{\rho(h)}} \nabla F(u) = &\ e^{\ov{\rho(h)}} \big( \nabla_{v_t} \nabla F + 2 \rho ( \partial h''/ \partial \ov{z}) \nabla F \big).
\end{split}
\end{align*}
\end{lemma}

\begin{proof}
We have
\begin{align*}
\begin{split}
D_{A,s} e^{\ov{\rho(h)}} \nabla F(u)  = &\ e^{\ov{\rho(h)}} \left( \ov{\rho (\partial_s h)} \nabla F + \nabla_s \nabla F + \nabla_{\nabla F } {\mc X}_{\phi} \right)\\
= &\  e^{\ov{\rho(h)}}   \left( \ov{\rho ( \partial_s h)} \nabla F  + \nabla_{v_s} \nabla F + [\nabla F, {\mc X}_\phi]  \right)\\
= &\  e^{\ov{\rho(h)}} \left( \rho ( - \partial_s h' + {\bm i} \partial_s h'') \nabla F + \nabla_{v_s} \nabla F  + \rho ( \partial_s h' - \partial_t h'') \nabla F \right)\\
= &\  \nabla_{v_s} e^{\ov{\rho(h)}} \nabla F + \rho ( {\bm i} \partial_s h'' - \partial_t h'' ) e^{\ov{\rho(h)}} \nabla F.
\end{split}
\end{align*}
The third equality follows from Lemma \ref{lemma33}. The formula for $D_{A, t} e^{\ov{\rho(h)}} \nabla F(u)$ follows in the same way.
\end{proof}

On the other hand, consider a vector field $Z$ along $u: U \to \wt{X}$. We have
\begin{align}\label{equation313}
\begin{split}
&\ D_{A, s} \Big( e^{\ov{\rho(h)}} \nabla_Z \nabla F \Big) \\
= &\ e^{\ov{\rho(h)}} \Big( \ov{\rho(\partial_s h)} \nabla_Z \nabla F + \nabla_s \nabla_Z \nabla F + \nabla_{\nabla_Z \nabla F} {\mc X}_\phi \Big)\\
= &\ e^{\ov{\rho(h)}} \Big(  \ov{\rho(\partial_s h) } \nabla_Z \nabla F + \big[ \nabla_Z \nabla F, {\mc X}_\phi \big] + \big( \nabla_s + \nabla_{{\mc X}_\phi} \big) \nabla_Z \nabla F \Big)\\
= &\ e^{\ov{\rho(h)}} \Big( 2 \rho ({\bm i}  \partial h''/ \partial \ov{z} ) \nabla_Z \nabla F +  \nabla_{[Z, {\mc X}_\phi]} \nabla F +    \big( \nabla_s + \nabla_{{\mc X}_\phi} \big) \nabla_Z \nabla F \Big)\\
= &\ e^{\ov{\rho (h)}} \Big(  2 \rho ({\bm i}  \partial h''/ \partial \ov{z} ) \nabla_X \nabla F +  \nabla_{\nabla_s Z + \nabla_{{\mc X}_\phi} Z + [Z, {\mc X}_\phi] } \nabla F \Big) \\
  &\ + e^{\ov{\rho (h)}} \Big( \big( \nabla_s + \nabla_{{\mc X}_\phi} \big) \nabla_Z \nabla F - \nabla_{\nabla_s Z + \nabla_{{\mc X}_\phi} Z} \nabla F \Big) \\
	= &\ e^{\ov{\rho(h)}} \Big( 2 \rho ({\bm i}  \partial h''/ \partial \ov{z} ) \nabla_Z \nabla F + \nabla_{D_{A, s} Z} \nabla F + G_F (v_s, Z) \Big)	
	\end{split}
	\end{align}
where the tensor $G_F$ is the third derivative of $F$, given by 
\begin{align*}
G_F (V, Z) = \nabla_V(\nabla_Z \nabla F) - \nabla_{\nabla_V Z} \nabla F.
\end{align*}
In deriving the third equality we used the second part of Lemma \ref{lemma33}. Similar to (\ref{equation313}),
\begin{align}\label{equation314}
D_{A, t} \Big( e^{\ov{\rho(h)}} \nabla_Z \nabla F \Big) = e^{\ov{\rho(h)}} \Big( 2 \rho( \partial h''/ \partial \ov{z}) \nabla_Z \nabla F + \nabla_{D_{A, t} Z} \nabla F + G_F (v_t, Z) \Big).	
\end{align}

Now we denote
\begin{align*}
D_A^{1, 0} = \big( D_{A, s} - J D_{A, t} \big)/2,\ D_A^{0,1} = \big( D_{A, s} + J D_{A, t} \big)/2.
\end{align*}
Then Lemma \ref{lemma35} and Lemma \ref{lemma34} imply that
\begin{align}\label{equation315}
\begin{split}
D_A^{1,0} e^{\ov{\rho(h)}} \nabla F(u) = &\ \nabla_{\ov\partial_A u} e^{\ov{\rho(h)}} \nabla F(u),\\
D_A^{0,1} e^{\ov{\rho(h)}} \nabla F(u) = &\ \nabla_{\partial_A u } e^{\ov{\rho(h)}} \nabla F(u) + 2 e^{\ov{\rho(h)}} \rho( {\bm i} \partial h''/ \partial \ov{z} ) \nabla F(u).
\end{split}
\end{align}

Suppose $(\beta, \sigma, \delta)$ parametrizes a local model. For any smooth $(u, h): B_r \to \wt{X} \times {\mf g}^{\mb C}$, abbreviate the inhomogeneous term in the first equation of (\ref{equation31}) as
\begin{align*}
\nabla \wt{\mc W}_A=  e^{\ov{\rho_0(h(z))}} \nabla W(x)  + \beta(z) \sum_{l=1}^s e^{\ov{\rho_l (h(z))}}\nabla F_l^{(\delta)} (x) = {\mc W}_A  + \beta {\mc W}_A'.
\end{align*}
Then by (\ref{equation315}), we have
\begin{align}\label{equation317}
\begin{split}
D_A^{1,0} \nabla \wt{\mc W}_A(u) = &\ \nabla_{\ov\partial_A u} \nabla \wt{\mc W}_A(u) + (\partial \beta / \partial z) \nabla {\mc W}_A'(u);
\end{split}\\
\begin{split}\label{equation318}
D_A^{0,1} \nabla \wt{\mc W}_A (u) = &\ \nabla_{\partial_A u} \nabla \wt{\mc W}_A (u) + ( \partial \beta / \partial \ov{z}) \nabla {\mc W}_A'(u) \\
 +\ 2 e^{\ov{\rho_0(h)}} \rho_0( {\bm i} \partial h''/\partial \ov{z} ) \nabla F_0^{(\delta)}(u) & +  2 \beta \sum_{l=1}^s e^{\ov{\rho_l(h)}} \rho_l ( {\bm i}  \partial h''/ \partial \ov{z}) \nabla F_l^{(\delta) } (u) .
\end{split}
\end{align}

Moreover, for $l = 0, 1, \ldots, s$, we define
\begin{align}\label{equation319}
\begin{split}
H_{A, s}^{(l)}(u, d_A u, Z) = &\ e^{\ov{\rho_l (h)}} \Big(  2 \rho_l ( {\bm i} \partial h''/\partial \ov{z})  \nabla_Z \nabla F_l^{(\delta)}  + G_{F_l^{(\delta)} } (v_s, Z) \Big),\\
H_{A, t}^{(l)}(u, d_A u, Z) = &\ e^{\ov{\rho_l (h)}}\Big(  2 \rho_l (\partial h''/ \partial \ov{z}) \nabla_Z \nabla F_l^{(\delta)}  + G_{F_l^{(\delta)}} (v_t, Z) \Big).
\end{split}
\end{align}
We define
\begin{align}\label{equation320}
\begin{split}
\wt{H}_{A, s} = H_{A, s}^{(0)} + \beta \sum_{l =1}^s H_{A, s}^{(l)}&, \ \wt{H}_{A, t} =  H_{A, t} + \beta \sum_{l =1}^s H_{A, t}^{(l)},\\
 \wt{H}_A^{0,1} = {1\over 2} &\big( \wt{H}_{A, s} + J \wt{H}_{A, t}\big).
\end{split}
\end{align}
Then by (\ref{equation313}) and (\ref{equation314}) we have
\begin{align}\label{equation321}
\begin{split}
&\ D_A^{0,1} \nabla_Z \nabla \wt{\mc W}_A (u)\\
= &\ D_A^{0,1} \big( e^{\ov{\rho(h)}} \nabla_Z \nabla F_0(u) \big) + (\partial \beta / \partial \ov{z}) \nabla_Z \nabla {\mc W}_A'(u) + \beta D_A^{0,1} \sum_{l=1}^s e^{\ov{\rho_l(h)}} \nabla_Z \nabla F_l^{(\delta)}(u) \\
= &\ \wt{H}_A^{0,1}(u, d_A u, Z) + (\partial \beta /\partial \ov{z}) \nabla_Z \nabla {\mc W}_A'(u) + \nabla_{D_A^{1,0} Z} \nabla \wt{\mc W}_A(u).
\end{split}
\end{align}

\section{Asymptotic behavior}\label{section3}

In this section we consider the asymptotic behavior of solutions to the gauged Witten equation. It suffices to consider the equation over cylindrical ends of $\Sigma^*$ and hence we can use cylindrical models introduced in the last section. 

Within this section, we fix $\lambda \in {\bm i} \big( {\mb Z}/r \cap [0, 1) \big)$ and denote $\upgamma = \exp(2\pi \lambda)$. 

\begin{defn}\label{defn41}
A solution ${\bm u}$ to a cylindrical model is called {\bf bounded}, if $E({\bm u}) < \infty$ and there is a compact subset $\wt{K} \subset \wt{X}$ such that $u(\Theta_+) \subset \wt{K}$. 

A solution $(A, u)$ to the perturbed gauged Witten equation over the rigidified $r$-spin curve $\vec{\mc C}$ is called {\bf bounded} if its restriction to any of its cylindrical ends gives a bounded solution to the corresponding cylindrical model.
\end{defn}

Our main theorems of this section are
\begin{thm}\label{thm42} Suppose ${\bm u} = (u, h)$ is a bounded solution to a $\lambda$-cylindrical model with parameters $(\sigma, \delta)$. Then there is a point $\upkappa \in \wt{X}_\upgamma$ such that 
\begin{align*}
\lim_{s \to +\infty} e({\bm u})(s, t) = 0, \lim_{s \to +\infty} e^{\lambda t} u(s, t) = \upkappa.
\end{align*}
both uniformly for $t\in S^1$.
\end{thm}

\begin{thm}\label{thm43}
For every $G$-invariant compact subset $\wt{K} \subset \wt{X}$ and every $\ud\delta\in (0, 1]$, there are constants $\epsilon (\wt{K}, \ud\delta ),  c (\wt{K}, \ud\delta), \tau (\ud\delta) > 0$ satisfying the following conditions. Suppose ${\bm u} = (u, h)$ is a bounded solution to a $\lambda$-cylindrical model parametrized by $(\sigma, \delta)$ such that $u (\Theta_+) \subset \wt{K}$ and if $\upgamma$ is broad, then $\delta \geq \ud\delta$. Then 
\begin{align*}
\big\| e({\bm u}) \big\|_{L^\infty(\Theta_+)} \leq \epsilon (\wt{K}, \ud\delta ) \Longrightarrow e({\bm u})(s, t) \leq c (\wt{K}, \ud\delta ) e^{-\tau (\ud\delta) s}.
\end{align*}
\end{thm}
The proof of Theorem \ref{thm42} and Theorem \ref{thm43} are given in Subsection \ref{subsection41}.

It follows from Theorem \ref{thm42} that we can define the evaluations of a bounded solutions the perturbed gauged Witten equation at the punctures. Indeed, let $(A, u)$ be a bounded solution to the perturbed gauged Witten equation over $\vec{{\mc C}}$. Restrict $(A, u)$ to the cylindrical end near $z_j$ with residue $\lambda_j$, we obtain a bounded solution to a $\lambda_j$-cylindrical model. Then by Theorem \ref{thm42}, we have the well-defined limit
\begin{align*}
\lim_{z \to z_j}  e^{-\lambda_j t} \phi_j^{-1} u(z) = \upkappa_j \in \wt{X}_{\upgamma_j}.
\end{align*}
We denote $ev_j(A, u) = \upkappa_j$ but indeed, the evaluation of the solution $(A, u)$ is a point on the fibre of $Y$ at $z_j$. On the other hand, for each $j =0, \ldots, k$, we have a solution to a $\lambda_j$-cylindrical model with parameters $(\sigma_j, \delta_j)$, where $\sigma_j ds dt $ is the restriction of the area form $\nu$ onto $U_j$ and $\delta_j = \delta_{j, A}$. Then the residue of $(A, u)$ at $z_j$ is
\begin{align*}
{\rm Res}_j (A, u) = \sum_{l=0}^s F_{\upgamma_j; l}^{(\delta_j)}( \upkappa_j) \in {\mb C}.
\end{align*}
Indeed, the residue ${\rm Res}_j(A, u)$ is nonzero only if $z_j$ is a broad puncture of $\vec{\mc C}$. A corollary to Theorem \ref{thm42} is the following uniform energy bound.

\begin{thm}\label{thm44} If $(A, u)$ is a bounded solution to the perturbed gauged Witten equation, then $u$ extends to a continuous orbifold section ${\mc U}$ of ${\mc Y} \to {\mc C}$, which defines a rational homology class
\begin{align*}
\big[ A, u \big] \in H^G_2 \big( \wt{X}; {\mb Z}[r^{-1}] \big).
\end{align*}
(See Appendix \ref{appenxib} for the precise meanings.) We have, 
\begin{align*}
E(A, u) = \big\langle \big[ \omega - \mu \big], \big[ A, u \big] \big\rangle +  {\rm Re} \Big( \int_\Sigma {\bm i}{\mc W}_A'(u) \wedge \ov\partial \beta - 4 \pi \sum_{j=0}^k {\rm Res}_j (A, u)\Big).
\end{align*}
Here $\big[ \omega - \mu \big] \in H^2_G \big( \wt{X}; {\mb R} \big)$ is the equivariant cohomology class represented by the equivariant symplectic form $\omega - \mu$. Moreover, there is a constant $E$ depending only on the class $\big[ A, u \big]$ such that $E(A, u) \leq E$.
\end{thm}
Its proof is given in Subsection \ref{subsection44}.

\subsection{Proof of Theorem \ref{thm42} and \ref{thm43}}\label{subsection41}

\subsubsection*{Decay of energy density}

We first prove the first half of Theorem \ref{thm42}. 

\begin{prop}\label{prop45}
For any bounded solution $(u, h)$ to a $\lambda$-cylindrical model, we have
\begin{align*}
\lim_{s \to +\infty} \big| \partial_s u + {\mc X}_\phi(u) \big| = \lim_{s \to +\infty} \big| \partial_t u + {\mc X}_\psi (u) \big| = \lim_{s \to +\infty} \big| \nabla \wt{W}_{h, \lambda}^{(\delta)}(u) \big| = 0.
\end{align*}
In particular, 
\begin{align*}
\big\| e({\bm u}) \big\|_{L^\infty(\Theta_+)} < +\infty.
\end{align*}
\end{prop}
	
\begin{proof}
We abbreviate $v_s = \partial_s u + {\mc X}_\phi(u)$, $v_t = \partial_t u + {\mc X}_\psi(u)$. The proof is based on estimating $\Delta |v_s|^2$ and $\Delta |v_t|^2$. For any $z \in {\rm Int} \Theta_+$, choose a small disk $B_r(z) \subset \Theta_+$. Then the function $\lambda t$ is single-valued on $B_r(z)$ and the restriction of $(u, h+\lambda t)$ to $B_r(z)$ gives a solution to the local model parametrized by $(\beta = 1, \sigma|_{B_r(z)}, \delta)$. Replacing $h$ by $h + \lambda t$, and using the notations introduced in Subsection \ref{subsection32}, by (\ref{equation311}) and (\ref{equation312}), we have
\begin{align}\label{equation41}
\begin{split}
&\ \big( D_{A, s}^2 + D_{A, t}^2 \big) v_s\\
 = &\ D_{A, s} \big( D_{A, s} v_s + D_{A, t} v_t \big) - \big[ D_{A, s}, D_{A, t} \big] v_t - D_{A, t} \big( D_{A, s} v_t - D_{A, t} v_s \big)\\
= &\ D_{A, s} \Big( D_{A, s} \big( -J v_t - 2 \nabla \wt{W}_{h, \lambda}^{(\delta)} \big) + D_{A, t} \big( J v_s + 2 J \nabla \wt{W}_{h, \lambda}^{(\delta)} \big) \Big) \\
&\ - R(v_s, v_t) v_t - \nabla_{v_t} {\mc X}_{F_A} - D_{A, t} {\mc X}_{F_A}\\
= &\ - J D_{A, s} {\mc X}_{F_A} - D_{A, t} {\mc X}_{F_A} - \nabla_{v_t} {\mc X}_{F_A} - R(v_s, v_t) v_t - 4 D_{A, s} D_A^{1, 0} \nabla \wt{W}_{h, \lambda}^{(\delta)}\\
= &\ J D_{A, s} \big( \sigma {\mc X}_{\mu^*}\big) + D_{A, t} \big(\sigma {\mc X}_{\mu^*}\big) + \sigma \nabla_{v_t} {\mc X}_{\mu^*} - R(v_s, v_t) v_t - 4 D_{A, s} D_A^{1, 0} \nabla \wt{W}_{h, \lambda}^{(\delta)}.
\end{split}
\end{align}
By the definition of $D_{A, s}$, $D_{A, t}$, the invariance of $\mu$ and the boundedness of ${\bm u}$, there is a constant $C ({\bm u})>0$ such that
\begin{align}
\begin{split}
&\ \big| JD_{A, s} \big( \sigma {\mc X}_{\mu^*} \big) + D_{A, t} \big( \sigma {\mc X}_{\mu^*} \big) \big| \\
 \leq &\ 2 \big| \ov\partial \sigma \big| \big| {\mc X}_{\mu^*} \big| + \big| \sigma \big| \big( \big| {\mc X}_{d \mu^* \cdot v_s}\big| + \big| {\mc X}_{d \mu^*\cdot v_t} \big| + \big| \nabla {\mc X}_{\mu^*} \big| \big| d_A u \big| \big)\\
\leq &\ C ({\bm u}) \big( 1 + \big| d_A u \big| \big).
\end{split}
\end{align}
On the other hand, by (\ref{equation317}) and (\ref{equation313}), we have 
\begin{align}\label{equation43}
 D_{A, s} D_A^{1, 0} \nabla \wt{W}_{h, \lambda}^{(\delta)} = &\  D_{A, s} \nabla_{\ov\partial_A u} \nabla \wt{W}_h^{(\delta)} =\nabla_{D_{A, s} \ov\partial_A u } \nabla \wt{W}_{h, \lambda}^{(\delta)} + \wt{H}_{A, s} (u, d_A u, \ov\partial_A u) .
\end{align}
Here $\wt{H}$ is the tensor field defined by (\ref{equation320}). We see that in the expression of $H_{A,s}^{(l )}$ in (\ref{equation319}), the tensor field $G_{F_l^{(\delta)}}$ and the Hessian of $F_l^{(\delta)}$ are uniformly bounded because $u(\Theta_+)$ is contained in a compact subset of $\wt{X}$. Moreover, the equation $\Delta h'' = -\sigma \mu^* (u)$ and the condition $\lim_{s \to +\infty} h = 0$ imply the uniform bound on $dh''$ and $e^{\ov{\rho_l (h)}}$. Therefore, abusing $C({\bm u})>0$, we have 
\begin{align}\label{equation44}
\begin{split}
\big| \wt{H}_{A, s} (u, d_A u, \ov\partial_A u ) \big| \leq  &\ C ({\bm u}) \big( 1 + \big| d_A u \big| \big) \big| \ov\partial_A u \big|;\\
\big| \nabla_{D_{A, s} \ov\partial_A u} \nabla \wt{W}_{h, \lambda}^{(\delta)} \big| \leq &\ C({\bm u}) \big( \big| D_{A, s} v_s \big| + \big| D_{A, s} v_t \big|\big).
\end{split}
\end{align}
By (\ref{equation41})-(\ref{equation44}) and abusing $C( {\bm u})$, we obtain
\begin{align*}
\begin{split}
{1\over 2} \Delta|v_s|^2 = &\ \big\langle (D_{A, s}^2 + D_{A, t}^2 ) v_s, v_s \big\rangle + \big| D_{A, s} v_s \big|^2+  \big| D_{A, t} v_s \big|^2 \\
= & \big\langle - J D_A^{1,0} {\mc X}_{F_A} - \nabla_{v_t} {\mc X}_{F_A} - R(v_s, v_t) v_t , v_s \big\rangle \\
& -4 \big\langle D_{A, s} D_A^{1, 0} \nabla \wt{W}_h^{(\delta)}, v_s \big\rangle +\big| D_A v_s\big|^2\\
\geq & - C( {\bm u}) \big( 1 + |v_s|^4 + |v_t|^4 \big) + \big| D_A v_s \big|^2 - C({\bm u}) \big(\big| D_{A, s} v_s \big| + \big| D_{A, s} v_t \big| \big) \big| v_s \big|
\end{split}
\end{align*}
In the same way, we have 
\begin{align*}
\Delta |v_t|^2 \geq - C( {\bm u}) \big( 1+ |v_s|^4 + |v_t|^4 \big)  + \big| D_A v_t \big|^2 - C({\bm u}) \big( \big| D_{A, t} v_s \big| + \big| D_{A, t} v_t \big| \big) |v_t|.
\end{align*}
Therefore, abusing $C({\bm u})$ again, we have
\begin{align*}
\Delta \big( |v_s|^2 + |v_t|^2 \big) \geq -C({\bm u}) - C({\bm u}) \big( |v_s|^2 + |v_t|^2 \big)^2.
\end{align*}
Then by the mean value estimate (Lemma \ref{lemmaa5}), there exist positive numbers $\epsilon, L >0$ depending on $C({\bm u})$, such that for any $z\in \Theta_+$ and $B_r(z) \subset \Theta_+$, we have
\begin{align*}
\int_{B_r(z)} \big( |v_s|^2 + |v_t|^2 \big) \leq \epsilon \Longrightarrow |v_s(z)|^2 + |v_t(z)|^2 \leq L \Big( r^2 + {1\over r^2} \int_{B_r(z)} \big(|v_s|^2 + |v_t|^2 \big) \Big).
\end{align*}
Since the energy of the solution is finite, this estimate implies that
\begin{align*}
\lim_{s \to +\infty} \left(| v_s(s, t)|^2 + |v_t(s, t)|^2 \right)= 0.
\end{align*}
The equation (\ref{equation33}) implies $\displaystyle \lim_{s \to +\infty} \big| \nabla \wt{W}_{h, \lambda}^{(\delta)} (u(s, t)) \big|= 0$. \end{proof}

\subsubsection*{Temporal gauge}

Suppose ${\bm u} = (u, h)$ is a bounded solution to a $\lambda$-cylindrical model with parameters $(\sigma, \delta)$ and $u(\Theta_+) \subset \wt{K}$. Then we can transform it into temporal gauge as follows. Define
\begin{align}\label{equation45}
f(s, t) = \int_s^{+\infty} \phi(v, t) dv,\ g(s, t) = \exp f(s, t).
\end{align}
By (\ref{equation34}) $f$ is finite and has limit $0$ as $s \to +\infty$. Denote $u'(s, t) = g^{-1}(s, t) u(s, t),\ h' = h + f$. We call $(u', h')$ a temporal gauge solution.

For $l = 1, \ldots, s$, we abbreviate $F_{\upgamma; l}$ by $F_l$. We denote
\begin{align*}
\wt{W}^{(\delta)} := \wt{W}^{(\delta)} (x) := \sum_{l=0}^s F_l^{(\delta)}(x),\ \wt{W}_\lambda^{(\delta)} (z, x) = \wt{W}_\lambda^{(\delta)} (e^{\lambda t} x),
\end{align*}
Note that although $e^{\lambda t}x$ is multi-valued, $\wt{W}_\lambda^{(\delta)}$ is single-valued. We denote
\begin{align*}
R_h^{(\delta)} (z, x) = \wt{W}_{h, \lambda}^{(\delta)} (z, x) - \wt{W}_\lambda^{(\delta)}(z, x).
\end{align*}
Then by the expression (\ref{equation35}), it is easy to see that, for every $l_1, l_2 \geq 0$, there is a constant $C^{l_1, l_2}(\wt{K}) >0$ only depending on the compact subset $\wt{K}$ such that
\begin{align}\label{equation46}	
\sup_{x \in \wt{K}} e^{2s} \big| \nabla_z^{(l_1)} \nabla_x^{(l_2)} R_h^{(\delta)} (z, x) \big|  \leq C^{l_1, l_2}(\wt{K}) \big| \nabla^{(l_1)} h \big|.
\end{align}
Here $\nabla^{(l)}_z$ (resp. $\nabla^{(l)}_x$) means the derivative in the $z$-direction (resp. $x$-direction) of order $l$. The norm is taken with respect to the cylindrical metric.

By elliptic regularity and the boundedness of the solution, it is easy to prove 
\begin{lemma}\label{lemma46} 
For any real number $M>0$ and any natural number $l$, there exists a constant $C^l(\wt{K}, M) >0$ satisfying the following condition. If $(u, h)$ is a smooth bounded temporal gauge solution to a cylindrical model and $u(\Theta_+) \subset \wt{K}$, $\big\| e({\bm u}) \big\|_{L^\infty} \leq  M$, then 
\begin{align*}  
\big\| h \big\|_{C^l(\Theta_+)} + \big\| du \big\|_{C^l(\Theta_+)} \leq C^l(\wt{K},M). 
\end{align*} 
\end{lemma}

\begin{proof} In radial gauge, $\partial_s \psi = - \sigma \mu^* (u)$. Then by (\ref{equation32}), $\psi$ has a uniform $C^0$-bound. Moreover, the radial gauge condition implies that 
\begin{align*}
\partial_s h = \partial_s h' + {\bm i} \partial_s h'' = \partial_t h'' + {\bm i} \partial_s h'',\ \partial_t h = \partial_t h' + {\bm i} \partial_t h'' = \psi -\lambda - \partial_s h'' + {\bm i} \partial_t h'',
\end{align*}
which is bounded by $\psi$ and $dh''$, while $dh''$ can be bounded via elliptic estimate by $\Delta h'' = F_A = - \sigma \mu^* (u)$. The uniform bound on $h$ follows from the fact that $\lim_{s \to +\infty} h = 0$. On the other hand, the bound on energy density implies uniform gradient bound on $u$. Therefore, using (\ref{equation46}) to bound the inhomogeneous term, by elliptic bootstrapping for Cauchy-Riemann equations and Sobolev embedding we obtain the uniform bounds on all derivatives of $u$. 
\end{proof}

To proceed with the proof of exponential convergence, we need the following result.
\begin{lemma}\label{lemma47}
For any natural number $l$ and any real number $M>0$, there is a constant $C^l(\wt{K}, M) >0$ such that if $(u, h)$ is a solution to a cylindrical model and $\big\| e({\bm u}) \big\|_{L^\infty} \leq M$, then for any $s \geq 0$, we have
\begin{align}\label{equation47}
\big\| \partial_s \psi \big\|_{C^l([s, +\infty)\times S^1)} \leq C^l(\wt{K}, M) e^{-2s}.
\end{align}
\begin{align*}
\big\| \partial_s u \big\|_{C^l([s, +\infty) \times S^1)} \leq C^l(\wt{K}, M) \Big( e^{-2s} + \big\|\partial_s  u \big\|_{C^0([s-1, +\infty) \times S^1)} \Big).
\end{align*}
\end{lemma}

\begin{proof}
It is easy to see that (\ref{equation47}) follows from the vortex equation $\partial_s \psi = - \sigma \mu^* (u)$ and the uniform bound on all derivatives of $u$, which is provided by Lemma \ref{lemma46}. On the other hand, apply $\nabla_s$ to the (\ref{equation33}), we obtain
\begin{align*}
2 \nabla^{0,1} \partial_s u = -  J \nabla {\mc X}_{\psi} (\partial_s u ) - J {\mc X}_{F_A} (u) - 2 \nabla^2 \wt{W}_h^{(\delta)}(\partial_s u) - 2 (\nabla_s \nabla R_h)(u).
\end{align*}
Here $2\nabla^{0,1} = \nabla_s + J \nabla_t$. Fix $p > 2$. Lemma \ref{lemma46} implies uniform bounds on $\psi$, $F_A$. Then by (\ref{equation46}) and elliptic estimate, there is a constant $b_1(\wt{K}, M)>0$ such that
\begin{align*}
\big\| \partial_s u \big\|_{W^{1, p}([s, s+1]\times S^1)} \leq b_1(\wt{K}, M) \Big( e^{-2s}  + \big\|\partial_s u \big\|_{C^0([s-1, +\infty)\times S^1)} \Big).
\end{align*}
By elliptic bootstrapping we can replace the $W^{1, p}$-norm by the $W^{k, p}$-norm and the constant $b_1(\wt{K}, M)$ by some $b_k(\wt{K}, M)$. Indeed, if it is true for $k\geq 1$, then we see, the term $J \nabla {\mc X}_\psi(\partial_s u)$ and the term $\nabla^2 \wt{W}_h^{(\delta)} (\partial_s u)$ are linear in $\partial_s u$ and all derivatives of $J \nabla {\mc X}_\psi$ and $\nabla^2 \wt{W}_h^{(\delta)}$ are uniformly bounded by Lemma \ref{lemma46}; the term $J {\mc X}_{F_A} (u)$ is linear in $F_A = \partial_s \psi ds dt$ and all derivatives of $u$ are uniformly bounded by Lemma \ref{lemma46}; finally, all derivatives of $\nabla_s R_h$ are uniformly exponentially decay by (\ref{equation46}). Therefore, (\ref{equation47}), elliptic estimate and induction hypothesis imply that there is $b_{k+1}(\wt{K}, M) >0$ such that
\begin{align*}
\big\| \partial_s u \big\|_{W^{k+1, p}([s, s+1]\times S^1)} \leq b_{k+1}(\wt{K}, M) \Big( e^{-2s} + \big\| \partial_s u \big\|_{C^0([s-1, +\infty)\times S^1)} \Big).
\end{align*}
The bound on $C^l$-norm is obtained by Sobolev embedding.
\end{proof}

\subsubsection*{Exponential decay}

Let $\wt{N}_\upgamma \to \wt{X}_\upgamma$ be the normal bundle. Let $D >0$ be a small number and let $\wt{N}_\upgamma^D \cap \wt{K}$ be the $D$-neighborhood of $\wt{X}_{\upgamma}\cap \wt{K}$. There is a small $D_0>0$ such that the exponential map identifies $\wt{N}_\upgamma^{D_0}\cap \wt{K}$ with a neighborhood of the zero section of $\wt{N}_\upgamma \cap \wt{K}$. A point in this neighborhood is denoted either by $\exp_{\ov{x}} \xi$ or $(\ov{x}, \xi)$, for $\ov{x}\in \wt{X}_\upgamma$ and $\xi \in \wt{N}_{\upgamma}|_{\ov{x}}$. 

Now we state the result about the exponential decay of the normal component, which will be proved in Subsection \ref{subsection42}. The derivative of ${\mc X}_\lambda$ in the direction of $\wt{N}_\upgamma$ defines a skew-adjoint map $d{\mc X}_\lambda^N: \wt{N}_\upgamma \to \wt{N}_\upgamma$, whose spectra are locally constant and are disjoint from ${\bm i} {\mb Z}$. We define
\begin{align*}
\tau_0 = \tau_0(\lambda):= d \big( {\bm i} {\mb Z}, {\rm Spec} ( d{\mc X}_\lambda^N ) \big) \in (0,1).
\end{align*}

\begin{prop}\label{prop48} 
For every $G$-invariant compact subset $\wt{K} \subset \wt{X}$, there exist a constant $\epsilon_1 = \epsilon_1 (\wt{K})>0$ and for every $l$, a constant $C^l(\wt{K}) >0$ satisfying the following conditions. Suppose $(u, h)$ is a smooth temporal gauge solution to a $\lambda$-cylindrical model, and $u(\Theta_+) \subset \wt{K}$. Suppose
\begin{align}\label{equation48}
 \big\| e({\bm u}) \big\|_{L^\infty(\Theta_+)} \leq (\epsilon_1 )^2.
\end{align}
Then $u(\Theta_+) \subset \wt{N}_\upgamma^{D_0} \cap \wt{K}$. Moreover, if we write $u = \exp_{\ov{u}} \xi$ where $\ov{u}: \Theta_+ \to \wt{X}_\upgamma$ and $\xi \in \Gamma \left( \ov{u}^* \wt{N}_\upgamma\right)$, then for each $l$ and every $s \geq 0$, we have
\begin{align}\label{equation49}
\| \xi \|_{C^l([s, +\infty) \times S^1)} \leq C^l(\wt{K}) e^{- {1 \over 2} \tau_0 s}.
\end{align}
\end{prop}

In the broad case, $\wt{W}^{(\delta)}: \wt{X}_\upgamma \to {\mb C}$ is a holomorphic Morse function having finitely many critical points. Then for any $\ud\delta > 0$, there exists $\tau_1 = \tau_1( \ud\delta )>0$ such that for any $\delta \in [ \ud\delta, 1]$, for any critical point $\upgamma$ of $\wt{W}^{(\delta)}|_{\wt{X}_\upgamma}$, each eigenvalue of the Hessian of $\wt{W}^{(\delta)}|_{\wt{X}_\upgamma}$ has absolute value no less than $\tau_1$.

The following two propositions will be proved in Subsection \ref{subsection43}.
\begin{prop}\label{prop49}
Suppose $\upgamma$ is broad. Then for every $G$-invariant compact subset $\wt{K}\subset \wt{X}$ and every $\ud\delta>0$, there are constants $\epsilon_2 = \epsilon_2( \wt{K}, \ud\delta )>0, C_2 = C_2(\wt{K}, \ud\delta ) >0$ satisfying the following conditions. Suppose $(u, h)$ is a bounded smooth temporal gauge solution to a $\lambda$-cylindrical model with parameters $(\sigma, \delta)$ such that $\delta \geq \ud\delta$. Suppose
\begin{align*} 
\big\| e({\bm u}) \big\|_{L^\infty(\Theta_+)} \leq ( \epsilon_2 )^2.
\end{align*}
Then $u(\Theta_+) \subset \wt{N}_\upgamma^{D_0} \cap \wt{K}$ and there is a unique critical point $\upkappa$ of $\wt{W}^{(\delta)} |_{\wt{X}_\upgamma}$ such that for all $(s, t)\in \Theta_+$,
\begin{align*}
d \big( e^{\lambda t} \ov{u}(s, t), \upkappa \big) \leq C_2 e^{- {1\over 2} \min\{ \tau_0, \tau_1 \} s}.
\end{align*}
\end{prop}

\begin{prop}\label{prop410}
Suppose $\upgamma$ is narrow. Then for every $G$-invariant compact subset $\wt{K}\subset X$, there are constants $\epsilon_3 = \epsilon_3(\wt{K} ) >0$, $C_3 = C_3 (\wt{K} ) >0$ satisfying the following conditions. Suppose $(u, h)$ is a bounded smooth temporal gauge solution to a $\lambda$-cylindrical model. Suppose
\begin{align*}
 \big\| e({\bm u}) \big\|_{L^\infty(\Theta_+)} \leq ( \epsilon_3 )^2.
\end{align*}
Then $u(\Theta_+) \subset \wt{N}_\upgamma^{D_0} \cap \wt{K}$ and there is a point $\upkappa\in \wt{X}_{\upgamma}$ such that for all $(s, t) \in \Theta_+$, 
\begin{align*}
d\big( e^{\lambda t} \ov{u}(s, t), \upkappa \big) \leq C_3 e^{-{1\over 2} \tau_0  s}.
\end{align*}
\end{prop}

\begin{proof}[Proof of Theorem \ref{thm42} and \ref{thm43}]
It is easy to see that Proposition \ref{prop45}, \ref{prop49} and \ref{prop410} imply Theorem \ref{thm42}. On the other hand, we write $u = \exp_\upkappa \xi$. By Proposition \ref{prop48}, \ref{prop49} and \ref{prop410}, $\xi$ decays exponentially. Then Theorem \ref{thm43} follows from the elliptic estimates for a Cauchy-Riemann equation in $\xi$. The choices of the constants in Theorem \ref{thm43} are obvious.
\end{proof}

\subsection{Proof of Proposition \ref{prop48}}\label{subsection42}

\begin{lemma}\label{lemma411}
For any compact $G$-invariant subset $\wt{K}\subset \wt{X}$ and $D >0$, there is an $\epsilon_4 = \epsilon_4 (\wt{K}, D) >0$ such that if a $C^1$-loop $(x, \eta): S^1 \to \wt{K} \times {\mf g}$ satisfies
\begin{align}\label{equation410}
\sup_{t \in S^1}\Big( \big| x'(t) + {\mc X}_\eta (x(t)) \big|  + \sup_{t \in S^1}  \big| \eta(t) -\lambda \big| \Big) \leq \epsilon_4,
\end{align} 
then $x(S^1) \subset \wt{N}_{\upgamma}^D$.
\end{lemma}

\begin{proof}
Define $(g, y): [0,2\pi] \to G \times \wt{X}$ by 
\begin{align*}
g(t) = \exp \Big( \int_0^t \eta(\tau) d\tau \Big),\  y(t) = g(t) x(0).
\end{align*}
Then $y'(t) = g(t)_* \left( x'(t) + {\mc X}_{\eta(t)} (x(t)) \right)$ and (\ref{equation410}) implies $d(y(2\pi), y(0)) \leq 2\pi \epsilon_4$. Then
\begin{multline*}
d ( \upgamma x(0), x(0) ) \leq d ( \upgamma x(0), y(2\pi) ) + d ( y(2\pi), x(2\pi) )\\
= d \Big( \exp ( 2\pi \lambda) x(0), \exp \Big( \int_0^{2\pi} \eta(\tau) d\tau \Big) x(0) \Big)  + d ( y(2\pi), y(0) ) \\
\leq d \Big( \upgamma x(0), \exp \Big( \int_0^{2\pi} \eta(\tau) d\tau \Big) x(0) \Big) + 2\pi \epsilon_4.
\end{multline*}
(\ref{equation410}) also implies that $\big| 2\pi \lambda - \int_0^{2\pi} \eta(\tau) d\tau \big| \leq 2\pi \epsilon_4$. Then since $x(0)$ is in a compact subset, for $\epsilon_4$ small enough, $\upgamma x(0)$ is sufficiently close to $x(0)$ so that $x(0) \in \wt{N}_{\upgamma}^{{1\over 2} D} \cap \wt{K}$. Then since $|y'(t)|$ is very small, $y([0, 2\pi])$ is contained $\wt{N}_\upgamma^D \cap \wt{K}$ for $\epsilon_4$ small enough.
\end{proof}

Now let ${\bm u} = (u, h)$ is a bounded smooth temporal gauge solution to a $\lambda$-cylindrical model with $u(\Theta_+) \subset \wt{K}$. Take $\epsilon_1 = \epsilon_1 (\wt{K}, \upgamma) > 0$ undetermined. Then if $\big\| e({\bm u}) \big\|_{L^\infty} \leq (\epsilon_1)^2$, we have
\begin{align*}
\big| \partial_t u(s, t)  + {\mc X}_{\psi(s, t)} (u(s, t)) \big| \leq \epsilon_1,
\end{align*}
\begin{align*}
\big| \psi(s, t) - \lambda \big| \leq \int_s^{+\infty} \big| \sigma(\rho, t) \mu^* (u) \big| d\rho \leq \sqrt{ \epsilon_1} \int_s^\infty \sqrt{\sigma(\rho, t)} d\rho \leq \sqrt{ \epsilon_1 C^{(0)}(\sigma)} e^{-s}.
\end{align*}
Here $C^{(0)}(\sigma)$ is the one in (\ref{equation32}). Then we can choose $\epsilon_1$ sufficiently small so that by Lemma \ref{lemma411}, the first claim of Proposition \ref{prop48} is satisfied, i.e., $u ( \Theta_+ ) \subset \wt{N}_{\upgamma}^{D_0}$. Then we can use the exponential map to write $u = \exp_{\ov{u}} \xi$ for $\ov{u}: \Theta_+ \to \wt{X}_{\upgamma}$ and $\xi \in \Gamma \left( \ov{u}^* \wt{N}_\upgamma \right)$. 

Now we consider the equation that the normal component $\xi$ of $u$ should satisfy. Let $\pi: \wt{N}_\upgamma \to \wt{X}_\upgamma$, $\pi(\exp_{\ov{x}}\xi) = \ov{x}$ be the projection. The exponential map induces a bundle isomorphism
\begin{align}\label{equation411}
T\wt{X}|_{\wt{N}_\upgamma^{D_0}} \simeq \pi^* T\wt{X}_\upgamma \oplus \pi^* \wt{N}_\upgamma.
\end{align}
For any $V \in T \wt{X}|_{\wt{N}_\upgamma^{D_0}}$, we denote by $V^T$ the tangential component and $V^N$ the normal component, with respect to the above decomposition. This decomposition respects the $G$-action, i.e. for any $g \in G$ and $(\ov{x}, \xi) \in \wt{N}_\upgamma^{D_0}$, $g (\ov{x}, \xi) = (g \ov{x}, g \xi)$. Therefore,
\begin{align*}
{\mc X}_\lambda (\exp_{\ov{x}} \xi) = \left( {\mc X}_\lambda(\ov{x}) , {\mc X}_\lambda^N( \ov{x}, \xi) \right)
\end{align*}
where the second component is linear in $\xi$. However, the decomposition (\ref{equation411}) may not respect the complex structure and we can write the complex structure as
\begin{align*}
J (\ov{x}, \xi) = \left( \begin{array}{cc} J^T(\ov{x}) & 0\\
                                          0 & J^N(\ov{x})
\end{array} \right) + R_J(\ov{x}, \xi),
\end{align*}
where $R_J$ depends smoothly on $(\ov{x}, \xi)$ and there is a constant $C_J(\wt{K})>0$ depending on the compact set $\wt{K}$ such that for $(\ov{x},\xi) \in \wt{K}$, we have
\begin{align}\label{equation412}
| R_J (\ov{x}, \xi) | \leq C_J(\wt{K}) | \xi |.
\end{align}
Lastly, by ({\bf Q2}) of Hypothesis \ref{hyp25} and ({\bf P2}) of Hypothesis \ref{hyp28}, the Hessian of $\wt{W}_h^{(\delta)}$ vanishes along the normal bundle $\wt{N}_\upgamma$. By the uniform bound on $h$ (Lemma \ref{lemma46}), there is a constant $c^N(\wt{K})$ depending only on $\wt{K}$ such that 
\begin{align}\label{equation413}
\big| \big( \nabla \wt{W}_h^{(\delta)} (\ov{x}, \xi) \big)^N \big| \leq c^N(\wt{K}) | \xi |^2.
\end{align}

Use the above notations, the normal component of (\ref{equation33}) can be written as
\begin{multline}\label{equation414}
 \nabla_s \xi + J^N(\ov{u}) ( \nabla_t \xi + {\mc X}_\lambda(\xi) )  \\
 = - \big( R_J (\ov{u}, \xi) \left( \partial_t u + {\mc X}_{ \lambda}(u) \right) \big)^N -  \big( J {\mc X}_{\psi-\lambda} (u) \big)^N - 2 \big( \nabla \wt{W}_h^{(\delta)} (\ov{u}, \xi) \big)^N.
\end{multline}

\begin{lemma}\label{lemma412}
Denote the right hand side of (\ref{equation414}) by $R(s, t)$. There exists $c_1 >0$ and for any $\rho>0$, there are constants $\varepsilon_1 = \varepsilon_1(\rho)>0$ and  $S_1 = S_1(\rho)>0$ such that if $\big\| e({\bm u}) \big\|_{L^\infty} \leq (\varepsilon_1)^2$, then for $s \geq S_1$, we have
\begin{align*}
| R(s, t) | \leq \rho |\xi| ,\ | \nabla_s R(s, t) | \leq \rho^2 |\xi| + \rho |\nabla_s \xi|,\ | \nabla_t R | \leq c_1 (  |\xi|  + |\nabla_t \xi| ).
\end{align*}
\end{lemma}

\begin{proof}
We will estimate each term in the expression of $R(s, t)$ and all constants appeared below will depend on $\wt{K}$. 

First, by the vortex equation, there is a constant $a_1>0$ such that $| \psi - \lambda |  + |\partial_s \psi | \leq a_1 e^{-2s}$. Moreover, by Lemma \ref{lemma46}, there is a constant $a_2>0$ such that $|\partial_t \psi | \leq a_2$. Then in the expression of $R$, the contribution from $(J{\mc X}_{\psi - \lambda}(u))^N$ can be bounded in the desired way since it is linear in $\xi$. Moreover, by (\ref{equation413}), Lemma \ref{lemma46} and (\ref{equation46}), the contribution of $\big( \nabla \wt{W}_h^{(\delta)}(\ov{u}, \xi) \big)^N$ can be controlled in the desired way.

On the other hand, by (\ref{equation412}), we have
\begin{multline}\label{equation415}
\big| \big( R_J(\ov{u}, \xi) \left( \partial_t u + {\mc X}_\lambda(u)\right) \big)^N \big| \leq \big| R_J(\ov{u}, \xi) \big| \big|  \partial_t u + {\mc X}_\lambda(u) \big|\\
\leq C_J(\wt{K}) \big| \xi \big| \big( \big| \partial_t u + {\mc X}_\psi(u) \big| + \big| {\mc X}_{\psi - \lambda}(u) \big| \big) \leq  C_J(\wt{K}) \big( \varepsilon_1 + a_1 e^{-2s} \big) |\xi|;
\end{multline}
applying $\nabla_s$, we have that there is a constant $a_3>0$ such that 
\begin{multline}\label{equation416}
\big| \nabla_s \big( R_J (\ov{u}, \xi) (\partial_t u + {\mc X}_\lambda(u)) \big)^N \big| \\
 \leq a_3 \big( \big| \partial_s \ov{u} \big| \big| \xi \big| + \big| \nabla_s \xi \big| \big) \big| \big(\partial_t u + {\mc X}_\lambda(u) \big)^N \big| + a_3 \big| \xi \big| \big( \big| \nabla_s \partial_t u \big| + \big|\nabla_s {\mc X}_\lambda(u) \big| \big).
\end{multline}
Then by choosing $\varepsilon_1$ sufficiently small, $S_1$ sufficiently large, and using Lemma \ref{lemma47} to control $\nabla_t \partial_s u$, we see that for $s \geq S_1$, we have
\begin{align*}
\left| \nabla_s \left( R_J(\ov{u}, \xi) (\partial_t u + {\mc X}_\lambda(u)) \right)^N \right| \leq \rho^2 |\xi| + \rho|\nabla_s \xi|.
\end{align*}
Applying $\nabla_t$ to $R_J(\partial_t u + {\mc X}_\lambda(u))^N$ and using Lemma \ref{lemma46}, we see there are constant $a_4, a_5>0$ such that 
\begin{multline*}
\left| \nabla_t R_J(\ov{u}, \xi) (\partial_t u + {\mc X}_\lambda(u))^N \right| \\
\leq a_4 \left( |\partial_t \ov{u}| |\xi| + |\partial_t \xi|\right) \left| (\partial_t u + {\mc X}_\lambda(u))^N \right| + a_4 |\xi| \left( |\nabla_t \partial_t u | + |\nabla_t {\mc X}_\lambda(u) |\right)\\
\leq a_5 |\xi| + a_5 |\nabla_t \xi|.
\end{multline*}
So the lemma is proven.
\end{proof}

\begin{lemma}\label{lemma413}
There exist $c_2>0$ and $\varepsilon_2>0$ depending only on $\wt{K}$ that satisfy the following conditions. If $u(\Theta_+) \subset \wt{N}_{\upgamma}^{D_0} \cap \wt{K}$ and $\big\| e({\bm u}) \big\|_{L^\infty} \leq (\varepsilon_2)^2$, then 
\begin{align}\label{equation417}
\| \xi \|_{L^2(\{s\}\times S^1)} \leq c_2 e^{-{1\over 2} \tau_0 s}.
\end{align}
\end{lemma}

\begin{proof}
Let $s \geq S_1$ where $S_1$ is the one in Lemma \ref{lemma412}. Let $\ov{u}_s: S^1 \to \wt{X}_\upgamma$ be the restriction of $\ov{u}$ to $\{s\}\times S^1$. We denote by ${\mc L}(s): L^2( \ov{u}_s^* \wt{N}_\upgamma) \to L^2( \ov{u}_s^* \wt{N}_\upgamma)$ the following self-adjoint operator
\begin{align*}
{\mc L}(s) \eta = J^N(\ov{u}_s) \left( \nabla_t \eta + {\mc X}_\lambda(\eta)\right).
\end{align*}
We claim that for all $s \geq S_1$, ${\mc L}(s)$ is coercive in the sense that 
\begin{align}\label{equation418}
\big\| {\mc L}(s) \eta \big\|_{L^2(S^1)}^2 \geq  (\tau_0)^2 \big\| \eta \big\|_{L^2(S^1)}^2.
\end{align}
Indeed, any $\eta \in L^2(S^1)$ can be written as Fourier series $\eta = \sum_{k \in {\mb Z}} \eta_k e^{{\bm i} k t}$. Then
\begin{align*} 
\big\| {\mc L}(s) \eta \big\|_{L^2} = \Big\|  \sum_{k\in {\mb Z}} e^{{\bm i} k t} ( - k \eta_k + {\bm i} {\mc X}_\lambda( \eta_k) ) \Big\|_{L^2} =  \Big( \sum_{k\in {\mb Z}} |  - k \eta_k + {\bm i} {\mc X}_\lambda( \eta_k) |^2 \Big)^{1\over 2} \geq \tau_0 \big\| \eta \big\|_{L^2}.
\end{align*}

On the other hand, since the covariant derivative on $\wt{N}_{\upgamma}$ preserves the complex structure $J^N$, we have
\begin{align*}
{\mc L}'(s) = J^N \left(  [\nabla_s, \nabla_t] + \nabla_s {\mc X}_\lambda \right)= J^N \left( {\sf R}^N( \partial_s \ov{u}, \partial_t \ov{u}) + \nabla_s {\mc X}_\lambda  \right).
\end{align*}
Here ${\sf R}^N$ is the curvature tensor in the normal bundle. Then ${\mc L}'(s)$ is a family of bounded operators of $L^2$, and there exists a constant $a_6>0$ depending on $\wt{K}$ such that 
\begin{align}\label{equation419}
\big\|{\mc L}'(s) \big\| \leq a_6 \big\| \partial_s \ov{u} \big\|_{L^\infty} \leq a_6 \varepsilon_2.
\end{align}
Here we used the fact of Lemma \ref{lemma46} that $\|du\|$ is uniformly bounded. Then applying $\nabla_s$ to (\ref{equation414}), we obtain
\begin{align*}
\nabla_s^2 \xi + {\mc L}(s) \nabla_s \xi + {\mc L}'(s) \xi = \nabla_s R(s, t).
\end{align*}
Denote $v(s) = \| \xi(s, \cdot) \|_{L^2(S^1)}^2$. We claim that there exist $\rho>0$, $S_2 \geq 0$ and $\varepsilon_2>0$ such that
\begin{align}\label{equation420}
s\geq S_2 \Longrightarrow v''(s) \geq  \left\| {\mc L}(s) \xi \right\|_{L^2(S^1)}^2.
\end{align}
Indeed, for any $\rho>0$, for $s\geq S_1$ where $S_1$ is the one in Lemma \ref{lemma412}, we have 
\begin{align*}
\begin{split}
{1\over 2} v''(s) = &\   \big\langle \nabla_s^2 \xi , \xi \big\rangle +  \big\| \nabla_s \xi \big\|^2\\
                  = &\  \big\langle \nabla_s R - {\mc L} (s) \nabla_s \xi - {\mc L}'(s) \xi , \xi \big\rangle + \big\| \nabla_s \xi \big\|^2\\
									= &\  \big\langle \nabla_s R - {\mc L}'(s) \xi + {\mc L}(s) ( {\mc L}(s) \xi - R), \xi \big\rangle + \big\| \nabla_s \xi \big\|^2\\
									= &\ \big\| {\mc L}(s) \xi \big\|^2 + \big\langle -R, {\mc L}(s) \xi \big\rangle +  \big\langle  \nabla_s R - {\mc L}'(s) \xi, \xi \big\rangle + \big\| \nabla_s \xi \big\|^2 \\
									\geq &\  \big\| {\mc L}(s) \xi \big\|^2 - \big\| R \big\| \big\| {\mc L}(s) \xi \big\| - \big\| \nabla_s R \big\| \big\| \xi \big\| - \big\| {\mc L}'(s) \xi \big\| \big\| \xi\big\| + \big\| \nabla_s \xi \big\|^2 \\
									\geq &\  \big\| {\mc L}(s) \xi \big\|^2 - {1\over 4} \big\|{\mc L}(s) \xi \big\|^2 - 2 \rho^2 \big\| \xi \big\|^2 - \rho \big\| \xi \big\| \big\| \nabla_s \xi \big\| - a_6 \varepsilon_2 \big\| \xi\big\|^2 + \big\| \nabla_s \xi \big\|^2 \\
									\geq &\ {3\over 4} \big\| {\mc L}(s) \xi \big\|^2  -  { 9 \over 4} \rho^2 \big\| \xi \big\|^2  - a_6 \varepsilon_2 \big\| \xi \big\|^2.
									\end{split}
\end{align*}
Here norms and inner products are the ones in the $L^2$ space, and we used (\ref{equation419}) and the estimates of Lemma \ref{lemma412}. We choose $\rho$, $\varepsilon_2$, $S_2$ so that 
\begin{align*}
{9 \over 4} \rho^2 \leq {1 \over 8} (\tau_0)^2,\ \varepsilon_2 \leq \min \big\{ \varepsilon_1(\rho), (\tau_0)^2/ 8 a_1 \big\},\ S_2 \geq S_1(\rho).
\end{align*}
Then for $s \geq S_2$, (\ref{equation420}) holds, and by (\ref{equation418}), $v''(s) \geq (\tau_0)^2 v(s)$. Thus the function 
\begin{align*}
 e^{-\tau_0 s } \big( v'(s) + \tau_0 v(s) \big) 
\end{align*}
is non-decreasing on $[S_2, +\infty)$. Since $\displaystyle \lim_{s \to \infty} v(s) = 0$, we see that for $s\geq S_2$, 
\begin{align*}
v'(s) + \tau_0  v(s) \leq 0 \Longleftrightarrow {d\over ds}  \big( e^{\tau_0 s} v(s) \big) \leq 0.
\end{align*}
Therefore $ v(s) \leq e^{- \tau_0 s} \big( v(S_2) e^{ \tau_0 S_2} \big)$. Moreover, since $v(s)$ is uniformly bounded for all $s \geq 0$, there is $c_2>0$ such that (\ref{equation417}) holds.
\end{proof}

Then in the above situation, there is $c_3 = c_3(\wt{K}) > 0$ such that
\begin{align}\label{equation421}
\| \xi \|_{L^2([s-1, s+1]\times S^1)} \leq c_3 e^{-{1\over 2} \tau_0 s}
\end{align}
To derive pointwise estimate, we apply $\nabla_s - J^N \nabla_t$ to (\ref{equation414}). Then we obtain
\begin{align}\label{equation422}
\Delta \xi = (\nabla_s - J^N \nabla_t )(\nabla_s + J^N \nabla_t )\xi =  \left( \nabla_s - J^N \nabla_t \right) \left( R - J {\mc X}_\lambda(\xi) \right).
\end{align}
Choose $z_0 = (s_0, t_0) \in [1, +\infty) \times S^1$. Then by the uniform bound on derivatives of $u$ and Lemma \ref{lemma412}, we see there is a constant $c_4 = c_4(\wt{K}) >0$ such that  
\begin{align}
{1\over 2} \Delta |\xi|^2  \geq  \langle \Delta \xi , \xi \rangle \geq - c_4 |\xi|^2 \geq - {1\over 2} \Big( {\pi \over 16 c_3^2} e^{\tau_0 s_0}|\xi|^4 +  {16 c_4^2 c_3^2 \over \pi} e^{- \tau_0 s_0 }  \Big).
\end{align}
Denote
\begin{align*}
 A = {16 c_4^2 c_3^2 \over \pi} e^{- \tau_0 s_0 },\ B = {\pi \over 16 c_3^2} e^{ \tau_0 s_0}.
\end{align*}
By (\ref{equation421}), $\int_{B_r(z_0 )} |\xi|^2 \leq \pi/ 16 B$. Then by the mean value estimate (Lemma \ref{lemmaa5}) for the differential inequality $\Delta u \geq - A - Bu^2$, for $r = 1$, we have
\begin{align}\label{equation424}
|\xi(z_0 )|^2 \leq { 8 \over \pi} \int_{B_r(z_0 )} |\xi|^2 + {A\over 4} = \Big( {8 c_3^2 \over \pi}  + {4 c_4^2 c_3^2 \over \pi} \Big)  e^{ -  \tau_0 s_0 }=: c_5 e^{- \tau_0 s}.
\end{align}
For $l \geq 1$ the estimate (\ref{equation49}) follows from elliptic estimate.

\subsection{Proof of Proposition \ref{prop49} and \ref{prop410}}\label{subsection43}

In this subsection we use the symbols $c_1, c_2, \ldots$ abusively, which could be different from the ones in the last subsection. We also use $\nabla$ to denote the Levi-Civita connection on $\wt{X}_\upgamma$.

Now suppose ${\bm u} = (u, h)$ is a temporal gauge solution to a $\lambda$-cylindrical model, satisfying (\ref{equation48}). Then $u(\Theta_+) \subset \wt{N}_\upgamma^{D_0} \cap \wt{K}$ and $u$ can be written as $u = \exp_{\ov{u}}\xi$. Then with respect to the decomposition (\ref{equation411}), the tangential direction of $ \nabla \wt{W}_\lambda^{(\delta)}(u)$ is 
\begin{align*}
\big( \nabla \wt{W}_\lambda^{(\delta)} \big)^T = \nabla \wt{W}_\lambda^{(\delta)} (\ov{u}) + R_\lambda^{(\delta), T} (\ov{u}, \xi),
\end{align*}
where the remainder $R_\lambda^{(\delta), T}( \ov{u}, \xi)$ has norm less than a constant multiple of $|\xi |$. Then if we project the first equation of (\ref{equation33}) to the tangential direction, we have
\begin{multline}\label{equation425}
\partial_s \ov{u}  +  J^T \big( \partial_t \ov{u} + {\mc X}_{\lambda}( \ov{u} )  \big) + 2 \nabla \wt{W}_\lambda^{(\delta)} (\ov{u}) \\
 = - \big( J {\mc X}_{\psi - \lambda}(u) \big)^T - \big( R_J ( \partial_t u + {\mc X}_\psi(u) ) \big)^T - 2 \big( R_h^{(\delta)} (u)\big)^T  - 2 R_\lambda^{(\delta), T} (\ov{u}, \xi).
\end{multline}
Denote the right hand side by $R_0^T$. 

\begin{lemma}\label{lemma414}
There is a constant $c_1>0$ depending only on $\wt{K}$ such that for 
\begin{align*}
\big| R_0^T ( \ov{u}, \xi) \big| + \big| \nabla_s R_0^T  (\ov{u}, \xi) \big| + \big| \nabla_t R_0^T (\ov{u}, \xi) \big|  \leq c_1  e^{-{1\over 2} \tau_0 s}.
\end{align*}
\end{lemma}
\begin{proof}
On the right hand side of (\ref{equation425}), $(J{\mc X}_{\psi- \lambda})^T$ and $\big( R_h^{(\delta)} (u)\big)^T$ decay like $e^{-2s}$ (with all derivatives) which is faster than $e^{-{1\over 2} \tau_0 s}$. The other two terms together with their derivatives can be controlled by $|\xi|$, which decays like $e^{-{1\over 2} \tau_0 s}$ (with all derivatives) by Proposition \ref{prop48}.
\end{proof}

Since the image of $\ov{u}$ is contained in $\wt{X}_\upgamma$, the map $\ov{v}:= e^{\lambda t} \ov{u}(s, t)$ is still a smooth map from $\Theta_+$ to $\wt{X}_\upgamma$, which satisfies
\begin{align}\label{equation426}
\partial_s \ov{v} + J^T(\ov{v}) \partial_t \ov{v} + \nabla \wt{W}^{(\delta)} (\ov{v})  = e^{-\lambda t} R_0^T. 
\end{align}

\subsubsection*{The broad case}

Now we suppose that $\upgamma = \exp(2\pi \lambda)$ is broad and $\delta \geq \ud\delta_0>0$. Then there exists $D_1 = D_1 (\ud\delta)>0$ such that for all $\delta \in [\ud\delta, 1]$ and any two distinct critical points $\upkappa, \upkappa'$ of $\ov{W}^{(\delta)}: \wt{X}_\upgamma \to {\mb C}$ as distance bigger than $2D_1$. Therefore we can take $\epsilon_2$ small enough so that for each such solution, there is a unique critical point $\upkappa$ of $\wt{W}^{(\delta)}$ such that $\ov{u}(\Theta_+)$ is contained in the $D_1$-neighborhood of $\upkappa$. We also assume that $D_1$ is smaller than the injectivity radius of $\wt{X}_\upgamma \cap \wt{K}$. Then we can write $\ov{v} = \exp_{\upkappa} \eta$ for $\eta \in \Gamma ( \Theta_+, T_\upkappa \wt{X}_\upgamma )$. 

The derivative of $\exp_{\upkappa}$ induces a smooth family of isomorphisms $E_2(\eta): T_{\upkappa} \wt{X}_\upgamma \to T_{\exp_{\upkappa} \eta} \wt{X}_\upgamma$. Then for the tangential part $J^T$ of the complex structure $J$, we have
\begin{align*}
J^T(\exp_{\upkappa} \eta) E_2(\eta) =  E_2(\eta) J^T( \upkappa) + B_J^T(\eta).
\end{align*}
$B_J^T$ depends smoothly on $\eta$ and there is a constant $c_{\upkappa}>0$ such that $| B_J^T (\eta) | \leq c_\upkappa |\eta|$. On the other hand, let ${\sf A}_\upkappa^{(\delta)}: T_\upkappa \wt{X}_\upgamma \to T_\upkappa \wt{X}_\upgamma$ be the Hessian of $\wt{W}^{(\delta)} |_{\wt{X}_\upgamma}$ at $\upkappa$. We can write
\begin{align*}
\nabla \wt{W}^{(\delta)} ( \exp_\upkappa \eta) =  E_2(\eta) ( {\sf A}_{\upkappa}^{(\delta)} \eta ) + R_1^T (\eta).
\end{align*}
$R_1^T$ depends smoothly on $\eta$ and we may assume $| R_1^T (\eta) | \leq c_\upkappa |\eta|^2$ for the same $c_\upkappa$. This $c_\upkappa$ can be taken uniformly for all $\delta \in [\ud\delta, 1]$.

Therefore, (\ref{equation426}) can be written as
\begin{align}\label{equation427}
E_2 \left( \partial_s \eta + J^T(\upkappa)\left(  \partial_t \eta\right) + {\sf A}_{\upkappa} \eta \right)= e^{-\lambda t} R_0^T + R_1^T.
\end{align}

The following lemma can be proved in a similar way as proving Lemma \ref{lemma412}. We leave the proof to the reader.
\begin{lemma}\label{lemma415}
There exists $c_3 = c_3 ( \wt{K}, \ud\delta, \upgamma) >0$ and for any $\rho>0$, there are constants $\varepsilon_3 = \varepsilon_3(\rho)>0$ and $S_3 = S_3(\rho)>0$ (which also depend on $\wt{K}, \ud\delta, \upgamma$) satisfying the following condition. If $\big\| e({\bm u}) \big\|_{L^\infty(\Theta_+)} \leq (\varepsilon_3)^2$, then
\begin{align*}
\big| R_1^T(s, t) \big| \leq \rho \big| \eta \big| ,\ \big| \nabla_s R_1^T (s, t) \big| \leq \rho^2 \big| \eta \big| + \rho \big| \nabla_s \eta \big|;
\end{align*}
\begin{align*}
\big| \nabla_t R_1^T \big| \leq c_3 \big(  |\eta |  + |\nabla_t \eta | \big).
\end{align*}
\end{lemma}

Let ${\mc L}_\upkappa^{(\delta)}: L^2(S^1, T_\upkappa \wt{X}_\upgamma) \to L^2(S^1, T_\upkappa \wt{X}_\upgamma)$ be the operator
\begin{align*}
{\mc L}_\upkappa(\eta) = J^T(\upkappa) \partial_t \eta + {\sf A}_\upkappa^{(\delta)} \eta.
\end{align*}
In the same way as proving (\ref{equation418}) we can show that it is self-adjoint and coercive, i.e.,
\begin{align}\label{equation428}
\big\| {\mc L}_\upkappa^{(\delta)} (\eta ) \big\|^2 \geq (\tau_1)^2 \big\| \eta \big\|^2.
\end{align}

Denoting $R^T = E_2( \upkappa, \eta )^{-1} \left( e^{- \lambda t} R_0^T + R_1^T \right)$, (\ref{equation427}) implies that
\begin{align*}
\partial_s^2 \eta = \partial_s \big( R^T - {\mc L}_\upkappa^{(\delta)} (\eta) \big) = \partial_s R^T - {\mc L}_\upkappa^{(\delta)} ( R^T) + \big( {\mc L}_\upkappa^{(\delta)} \big)^2 (\eta).
\end{align*}
Then we denote $v(s) = \big\| \eta(s, \cdot) \big\|_{L^2(S^1)}^2$. Then by Lemma \ref{lemma414} and Lemma \ref{lemma415}, we have
\begin{multline*}
{1\over 2} v''(s) = \big\langle \partial_s^2 \eta, \eta \big\rangle + \big\| \partial_s \eta \big\|^2  = \big\langle \partial_s R^T - {\mc L}_\upkappa^{(\delta)} (R^T) +\big( {\mc L}_\upkappa^{(\delta)} \big)^2 (\eta) , \eta \big\rangle + \big\| \partial_s \eta \big\|^2 \\
\geq - \rho^2 \big\| \eta \big\|^2 - c_1 e^{-{1\over 2} \tau_0 s} \big\| \eta \big\| - \rho \big\| \partial_s \eta \big\| \big\| \eta \big\| - \big\|{\mc L}_\upkappa^{(\delta)} \eta \big\| \big(  \rho \big\| \eta \big\| + c_1 e^{-{1\over 2} \tau_0 s} \big)  + \big\|{\mc L}_\upkappa^{(\delta)} \eta \big\|^2 + \big\|\partial_s \eta \big\|^2 \\
\geq {5 \over 8} \big\| {\mc L}_\upkappa^{(\delta)} \eta \big\|^2 - {9 \over 4} \rho^2  \big\| \eta \big\|^2 + 2 \big( c_1 e^{-{1\over 2} \tau_0 s} \big)^2  - c_1 e^{-{1\over 2} \tau_0 s} \big\| \eta \big\|.
\end{multline*}
Then by (\ref{equation428}) and Lemma \ref{lemma415}, there are $\rho_4 >0$, $S_4 \geq S_3(\rho_4) > 0$ and $c_4 >0$ such that if $\big\| e({\bm u}) \big\|_{L^\infty(\Theta_+)} \leq (\varepsilon_3(\rho_4 ))^2$, then for $s \geq S_4$, we have
\begin{align}\label{equation429}
v''(s) \geq  \big\|{\mc L}_\upkappa^{(\delta)}  \eta \big\|^2 - c_4 e^{- \tau_0 s}\geq (\tau_1)^2 \big\| \eta \big\|^2 - c_4 e^{- \tau_0 s}.
\end{align}
This implies that the function
\begin{align*}
e^{- \tau_1 s} \Big( v'(s) + \tau_1 v(s) - {  c_4 \over \tau_0 + \tau_1}  e^{- \tau_0 s} \Big)
\end{align*}
is non-decreasing on $[S_4, +\infty)$. Then by the fact that $\lim_{s \to \infty} v(s) = 0$, we see for $s \geq S_4$, 
\begin{align*}
v'(s) + \tau_1 v(s) - { c_4 \over \tau_0 +\tau_1} e^{-\tau_0 s}\leq 0.
\end{align*}
We can assume that $\tau_0 \neq \tau_1$; otherwise we can slightly improve (\ref{equation418}) or (\ref{equation428}) so that the $\tau_0$ and $\tau_1$ appeared there are different. Therefore
\begin{align*}
{d\over ds} \Big( e^{\tau_1 s} \big( v(s) +    {  c_4 \over (\tau_0 - \tau_1 )(\tau_0 + \tau_1)} e^{ - \tau_0 s} \big) \Big) \leq 0.
\end{align*}
Therefore we see there is a constant $c_5 >0$ such that for $s \geq S_4$, 
\begin{multline}\label{equation430}
v(s) \leq  { c_4 \over (\tau_1 - \tau_0) (\tau_1 + \tau_0)} e^{-\tau_0 s} + e^{- \tau_1 s} \Big( v(S_4) + {c_4 \over ( \tau_1 - \tau_0 ) (\tau_1 + \tau_0 )} e^{-\tau_0  S_4 } \Big) \\
\leq c_5 e^{- \min\{\tau_0, \tau_1\} s}.
\end{multline}

To derive pointwise estimate we can use the similar method as in did in (\ref{equation422})--(\ref{equation424}). Indeed, apply $\partial_s - J^T(\upkappa) \partial_t$ to (\ref{equation427}), we obtain
\begin{align*}
\Delta \eta = (\partial_s - J^T(\upkappa)\partial_t ) R^T - ( \partial_s - J^T(\upkappa)\partial_t ) {\sf A}_{\upkappa}^{(\delta)} \eta.
\end{align*}
Therefore by Lemma \ref{lemma414} and Lemma \ref{lemma415}, there is a constant $C$ such that
\begin{align*}
{1\over 2} \Delta |\eta|^2 = \langle \Delta \eta, \eta \rangle + |  d\eta |^2 \geq -C ( e^{-\tau_0 s} + |\eta|^2 ).
\end{align*}
Denoting $\tau = \min\{ \tau_0, \tau_1\}$, then there is another constant $C'>0$ such that 
\begin{align*}
\Delta |\eta|^2 \geq - C' ( e^{-\tau s} + e^{\tau s} |\eta|^4 ).
\end{align*}
This allows us to derive a similar mean value estimate as did in (\ref{equation422})--(\ref{equation424}) and therefore Proposition \ref{prop49} is proven.

\subsubsection*{The narrow case}

Now we assume $\upgamma = \exp(2\pi \lambda)$ is narrow. 

For the compact set $\wt{X}_\upgamma \cap \wt{K}$, there is a constant $D_2 > 0$ satisfying the following condition. For any smooth loop $x: S^1 \to \wt{X}_{\upgamma}\cap \wt{K}$, if ${\rm diam} (  x(S^1) ) \leq D_2$, then we can define the {\bf center of mass}, which is a unique point $\alpha \in \wt{X}_{\upgamma}$ such that there is a function $\eta: S^1 \to T_\alpha \wt{X}_\upgamma$ such that 
\begin{align*}
x(t) = \exp_\alpha \eta(t), \ \int_{S^1} \eta(t) dt = 0.
\end{align*}
Therefore, it is easy to see that there is a constant $\varepsilon_4>0$ such that if $\big\| e({\bm u}) \big\|_{L^\infty} \leq (\varepsilon_4)^2$, then $| \partial_t \ov{u} + {\mc X}_\lambda(\ov{u}) |$ is small enough and hence the diameter of the loop $\ov{v}(s, \cdot)$ is smaller than $D_2$. Then the {\bf center of mass} of $\ov{v}(s, \cdot)$ is smooth curve $\alpha: [0, +\infty) \to  \wt{X}_\upgamma$. We regard $\alpha$ as a map $\alpha: \Theta_+ \to \wt{X}_\upgamma$ which is independent of the $t$-variable. Then there is a section $\eta \in \Gamma ( \Theta_+, \alpha^* T \wt{X}_\upgamma )$ so that
\begin{align*}
\ov{v}(s, t) = \exp_{\alpha(s)} \eta(s, t),\ \int_{S^1} \eta(s, t) dt = 0.
\end{align*}

Let $E_1, E_2$ be the components of the derivative of the exponential map of $\wt{X}_{\upgamma}$, i.e., 
\begin{align*}
d \exp_x V = E_1(x, V) dx + E_2(x, V) \nabla V,\ x \in \wt{X}_\upgamma,\ V \in T_x \wt{X}_\upgamma.
\end{align*}
Then using the center of mass, we rewrite (\ref{equation426}) as 
\begin{align}\label{equation431}
E_1( \alpha, \eta) \alpha'(s) + E_2( \alpha, \eta) \nabla_s \eta + J^T(\ov{v}) \left( E_2(\alpha, \eta) \partial_t \eta  \right) = e^{-\lambda t} R_0^T.
\end{align}
Moreover, there exists a linear map $R_J^T(s, t) : T_{\alpha(s)}  \wt{X}_\upgamma \to T_{\ov{v}(s, t)} \wt{X}_\upgamma$ such that
\begin{align}\label{equation432}
E_2(\alpha, \eta)^{-1} J^T( \ov{v}) E_2 (\alpha, \eta) - J^T (\alpha) = R_J^T,
\end{align}
We denote $R_1^T = R_J^T (\partial_t \eta)$, $R_2^T = E_2^{-1} E_1 \alpha'(s) - \alpha'(s)$ and $R^T = R_1^T + R_2^T$. Then (\ref{equation431}) can be rewritten as 
\begin{align}\label{equation433}
\alpha'(s) +  \nabla_s \eta + J^T(\alpha) \partial_t \eta = E_2^{-1} \big( e^{-\lambda t} R_0^T \big) + R^T.
\end{align}

\begin{lemma}\label{lemma416}
There exists $c_6 >0$ and for any $\rho>0$, there are constants $\varepsilon_6 = \varepsilon_6(\rho)>0$ and  $S_6 = S_6 (\rho)>0$ satisfying the following conditions. If $\big\| e({\bm u}) \big\|_{L^\infty}  \leq (\varepsilon_6 )^2$, then for $s \geq S_6$, we have
\begin{align*}
\big| R^T(s, t) \big| \leq \rho \big| \eta \big| ,\ \big| \nabla_s R^T (s, t) \big| \leq \rho^2 \big| \eta \big| + \rho \big| \nabla_s \eta \big|;
\end{align*}
\begin{align*}
\big| \nabla_t R^T \big| \leq c_6 \big(  \big| \eta \big|  + \big| \nabla_t \eta \big| \big).
\end{align*}
\end{lemma}
It can be proved in a similar way as Lemma \ref{lemma412}. The proof is left to the reader.

\begin{lemma}\label{lemma417}
There exist $c_7>0$ and $\varepsilon_7>0$ depending only on $\wt{K}$ that satisfy the following condition. If $\big\| e({\bm u}) \big\|_{L^\infty} \leq (\varepsilon_7)^2$, then 
\begin{align}\label{equation434}
\big\| \eta(s, \cdot) \big\|_{L^2(S^1)} \leq c_7 e^{-{1\over 2} \tau_0 s}.
\end{align}
\end{lemma}

\begin{proof}
Denote by $R(s, t)$ the right hand side of (\ref{equation433}) and let $\ov{H}_s \subset L^2(T_{\alpha(s)} T \wt{X}_{\upgamma})$ be the subspace of functions with zero average on $S^1$. Then $\nabla_s$ preserves this subspace. Projecting (\ref{equation433}) onto $\ov{H}_s$, $\alpha'(s)$ is killed and we have
\begin{align*}
\nabla_s \eta + J^T \partial_t \eta = \ov{R}(s, t)
\end{align*}
where $\ov{R}(s, \cdot) \in \ov{H}_s$ is the image of $R(s, \cdot)$ under the projection. 

The operator $J \partial_t$ is coercive on $\ov{H}_s$, satisfying 
\begin{align*}
\| J \partial_t \eta \|_{L^2}^2 \geq \|\eta\|_{L^2}^2.
\end{align*}
Notice that $\tau_0 < 1$. Then (\ref{equation434}) can be derived in the same way as deriving (\ref{equation430}).
\end{proof}

Applying the mean value estimate as did in (\ref{equation422})--(\ref{equation424}), we can prove that 
\begin{align*}
| \eta(s, t) | \leq c_7 e^{-{1\over 2} \tau_0 s}.
\end{align*}
Then it implies that $\left| R(s, t) \right| \leq c_7 e^{- {1\over 2}\tau_0 s}$ with $c_7$ abusively used. Taking $L^2$-paring of (\ref{equation433}) with $\alpha'(s)$, one has
\begin{align*}
\left| \alpha'(s) \right| \leq c_7 e^{- {1\over 2} \tau_0 s}.
\end{align*}
Then it implies that there exists $\upkappa \in \wt{X}_\upgamma$ such that 
\begin{align*}
\lim_{s \to \infty} \alpha(s) = \lim_{s \to \infty} \ov{v}(s, t) = \upkappa. 
\end{align*}
Therefore Proposition \ref{prop410} is proven.

\subsection{Proof of Theorem \ref{thm44}}\label{subsection44}

Let $(A, u)$ be a bounded solution to the perturbed gauged Witten equation over $\vec{\mc C}$. By Theorem \ref{thm42} and \ref{thm43}, $u$ extends to a continuous orbifold section ${\mc U}: {\mc C} \to {\mc Y}$. 

Now we compute the energy of $(A, u)$. For the fibration $Y \to \Sigma$ and any $G$-connection $A$ on $P$, we have the minimal coupling form $\omega_A \in \Omega^2(Y)$. For any smooth section $u: \Sigma^* \to Y$, the following identity is well-known (see for example, the proof of \cite[Proposition 2.2]{Cieliebak_Gaio_Mundet_Salamon_2002}).
\begin{align*}
{1\over 2} \big\| d_A u \big\|^2 \nu = u^* \omega_A + \big\| \ov\partial_A u \big\|^2 \nu + \mu (u) \cdot F_A.
\end{align*}
Then by the definition of the kinetic energy we have:
\begin{align}\label{equation435}
E_K (A, u)= \big\| \ov\partial_A u \big\|_{L^2(\Sigma^*)}^2 + \int_\Sigma u^*\omega_A + {1\over 2} \big\| * F_A  + \mu^* (u)  \big\|_{L^2(\Sigma^*)}^2 - {1\over 2}\big\| \mu(u) \big\|_{L^2(\Sigma^*) }^2.
\end{align}
Since $u$ extends to ${\mc U}$, we have
\begin{align}\label{equation436}
\int_{\Sigma^*} u^* \omega_A = \big\langle \big[ {\mc U} \big], \big[ \omega - \mu \big] \big\rangle = \big\langle \big[ A, u\big], \big[\omega - \mu\big] \big\rangle.
\end{align}
On the other hand, let $\langle \cdot, \cdot \rangle$ be the real part of the Hermitian pairing on $T\wt{X}$. Then
\begin{align}\label{equation437}
\begin{split}
 & \big| \ov\partial_A u \big|^2 + \big| \nabla \wt{\mc W}_A (u) \big|^2 - \big| \ov\partial_A u + \nabla \wt{\mc W}_A (u) \big|^2 \\
 =& - 2 \big\langle \ov\partial_A u , \nabla \wt{\mc W}_A (u) \big\rangle = 2 {\rm Im} * \big(  d \wt{\mc W}_A (u) \wedge \ov\partial_A u \big) \\
=& - 2 {\rm Im} * \big( \ov\partial \wt{\mc W}_A(u) \big) +2 {\rm Im} * \big( \ov\partial \beta \wedge {\mc W}_A'(u) \big).
\end{split}
\end{align}
Now we identify each cylindrical end $U_j(2)$ with $\Theta_+$ and identify the restriction of $(A, u)$ to $U_j(2)$ with a solution ${\bm u}_j = (u_j, h_j)$ to a $(\lambda_j, \delta_j)$-cylindrical model, where $\delta_j = \delta_{j, A}$. Then by Stokes formula, Theorem \ref{thm42} and \ref{thm43}, we have
\begin{align}\label{equation438}
\begin{split}
- 2{\rm Im} \int_{\Sigma^*} \ov\partial \wt{\mc W}_A(u) = &\ - 2 \int_{\Sigma^*} d \big( {\rm Im} \wt{\mc W}_A(u) \big) \\
                  = &\ -2 {\rm Re} \Big[ \sum_{j=1}^k \lim_{s \to +\infty}  \int_{\{s \} \times S^1} \wt{W}_{h_j, \lambda_j}^{(\delta_j)}(u_j (s, \cdot)) dt \Big] \\
								=	&\ - 4\pi {\rm Re} \Big[ \sum_{j=0}^k \wt{W}_{\upgamma_j}^{(\delta_j)} (\upkappa_j) \Big]\\
						    = &\ -4 \pi {\rm Re} \Big[ \sum_{j=0}^k {\rm Res}_j(A, u)\Big].
\end{split}
\end{align}
Then by (\ref{equation435})-(\ref{equation438}),
\begin{align}\label{equation439}
E(A, u) = \big\langle \big[A, u\big] , \big[ \omega, - \mu\big] \big\rangle - {\rm Re} \Big[ 4\pi \sum_{j=1}^k {\rm Res}_j(A, u) + \int_{\Sigma^*} {\bm i} \ov\partial \beta \wedge {\mc W}_A'(u) \Big].
\end{align}
Therefore the first part of Theorem \ref{thm44} is proved. 

To prove the second part, i.e., the uniform bound on the energy, we have to control the residues ${\rm Res}_j(A, u)$ and the integral of $\ov\partial\beta \wedge {\mc W}_A'(u)$. Indeed, if $z_j$ is narrow, ${\rm Res}_j(A, u) = 0$; if $z_j$ is broad, the limit $\upkappa_j = (x_j, p_j) \in X \times {\mb C}$ is a critical point of the function 
\begin{align*}
\wt{W}^{(\delta_j)}_{\upgamma_j} = \sum_{l=0}^s F_{\upgamma_j; l}^{(\delta_j)}.
\end{align*}
By (\ref{equation26}), $( \delta_j^{-1} x_j, p_j)$ is a critical point of the function $\wt{W}_{\upgamma_j}|_{\wt{X}_{\upgamma_j}}$, which is independent of $(A, u)$. Therefore
\begin{align}\label{equation440}
\wt{W}^{(\delta_j)}_{\upgamma_j} ( \upkappa_j) = \sum_{l=0}^s F_{\upgamma_j;l}^{(\delta_j)} (x_j, p_j) = \sum_{l=1}^{s-1} (\delta_j)^{r- l} F_{\upgamma_j; l} ( x_j ) = (\delta_j)^r F_{\upgamma_j; l}( \delta_j^{-1} x_j) .
\end{align}
So the sum of the residues in (\ref{equation439}) is uniformly bounded. On the other hand, for each broad $z_j$, denote $C_j:= U_j\setminus U_j(2)$, which contains the support of $d\beta_j$. We have
\begin{align}\label{equation441}
\begin{split}
\Big| \int_{\Sigma^*} {\bm i} \ov\partial \beta \wedge {\mc W}_A'(u) \Big| \leq &\ \sum_{z_j\ {\rm broad}} \Big| \int_{C_j}  \ov\partial \beta_j \wedge \sum_{l =1}^s e^{\rho_l (h + \lambda_j t) } F_{\upgamma_j; l}^{(\delta_j)} (u_j) dz \Big| \\
 \leq &\ \sum_{z_j\ {\rm broad}} \delta_j \int_{C_j} \Big( \sum_{l =1}^{s} \big| e^{\rho_l (h)} \big|  \big| F_{\upgamma_j; l} (u_j) \big| \Big) ds dt \\
\leq &\ c^{(0)}  \sum_{z_j\ {\rm broad}} \delta_j \int_{N_j} \Big( \sum_{l =1}^{s} \big| e^{\rho_l (h)} \big| \Big) \sqrt{  1 + \big| \mu (u_j) \big|}  ds dt\\
\leq &\ c^{(0)} \sum_{z_j\ {\rm broad}} \delta_j \Big( \sum_{l =1}^{s} \big\| e^{\rho_l (h)} \big\|_{L^2(N_j)} \Big) \big\| \sqrt{ 1 + \big| \mu (u_j) \big|} \big\|_{L^2(N_j) } \\
 \leq & \ c \Big( 1 + \big\| \sqrt{\sigma} \mu (u) \big\|_{L^2(\Sigma^*)} \Big) \\[0.2cm]
\leq &\ c \Big( 1 +  \sqrt{ E (A, u) }  \Big).
\end{split}
\end{align}
Here the second line follows from the fact that each $F_{\upgamma_j;l}^{(\delta_j)}$ has at least one $\delta_j$ factor; the third line uses ({\bf P3}) of Hypothesis \ref{hyp28}; the fifth line uses the definition of $\delta_j$; $c>0$ is a constant depending on $c^{(0)}$. It follows from (\ref{equation439})--(\ref{equation441}) that for some constant $c'$ independent of $(A,u)$, we have
\begin{align*}
E (A, u) \leq \big\langle \big[ A, u \big], \big[ \omega - \mu \big] \big\rangle + c' + c'\sqrt{ E (A, u)}.
\end{align*}
It implies a uniform bound on $E(A, u)$ in terms of $\big[ A, u \big]$, so Theorem \ref{thm44} is proved.

\begin{rem}
In the narrow case, i.e., when all punctures are narrow, we don't have to perturb the equation. Then by (\ref{equation437}), the $L^2$-norm of $\ov\partial_A u$ is zero because all the residues are zero. So any solution of the gauged Witten equation is also a solution to the symplectic vortex equation, whose image is contained in ${\rm Crit} W$. So analysis in the narrow case are much easier than the broad case (when there is at least one broad puncture). 
\end{rem}

\section{Linear Fredholm Theory}\label{section4}

In this section we consider the linearized operator of the perturbed gauged Witten equation modulo gauge transformations. This section is more or less independent of the other sections of this paper and can be treated under a much more general set-up, for example, the Lie group $G$ could be nonabelian, and the superpotential $\wt{\mc W}_A$ need not be holomorphic. The condition on narrowness and broadness of punctures can also be more flexible. However for simplicity we still use the set-up of given in Section \ref{section2}.

\subsection{Banach manifolds, Banach bundles and sections}

Let $\vec{\mc C}$ be a rigidified $r$-spin curve with underlying Riemann surface $\Sigma$. Let ${\bm z} =\{ z_1, \ldots, z_k\}$ be the set of punctures and the monodromy of $\vec{\mc C}$ at $z_j$ is $\upgamma_j$. The corresponding punctured Riemann surface $\Sigma^*$ is equipped with the cylindrical metric, which is used to define the weighted Sobolev spaces. Let $\tau>0$. We use $W_\tau^{k, p}(\Sigma^*, E)$ to denote the space of sections of a vector bundle $E$ over $\Sigma^*$, of class $W_\tau^{k, p}$, with respect to some fixed choice of connection on $E$. We will omit the domain $\Sigma^*$ in this section and abbreviate the space by $W_\tau^{k, p}(E)$.

For each $\upgamma \in {\mb Z}_r$, we have the function $\wt{W}_\upgamma: \wt{X} \to {\mb C}$ as introduced in (\ref{equation25}) such that if $\upgamma$ is broad, then $\wt{W}_\upgamma |_{\wt{X}_\upgamma}$ is a holomorphic Morse function with finitely many critical points $\upkappa_\upgamma^{(\iota)},\ \iota = 1,\ldots, m_\upgamma$. For $\delta \in (0, 1]$, $\upkappa_{\upgamma; \delta}^{(\iota)} := \delta \upkappa_\upgamma^{(\iota)}$ is a critical point of $\wt{W}_\upgamma^{(\delta)}|_{\wt{X}_\upgamma}$.

We abbreviate $\wt{X}_j = \wt{X}_{\upgamma_j}$. Now for each $A\in {\mz A}$, we have defined $\delta_{j,A}$ in (\ref{equation220}). Then for each $A \in {\mz A}$, denote $\upkappa_{j, A}^{(\iota)} = \upkappa_{\upgamma_j; \delta_{j,A}}^{(\iota)}$.

For any $x_j \in \wt{X}_j$, define $\wt{x}_j: S^1 \to \wt{X}_j$ by $\wt{x}_j(t) = e^{- \lambda_j t} \upkappa_j$. The pull-back of $\wt{x}_j$ to any cylinder $[a, b]\times S^1$ via the projection $[a, b] \times S^1 \to S^1$ is still denoted by $\wt{x}_j$. When $z_j$ is broad, denote $\wt\upkappa^{(\iota)}_{j,A}= e^{-\lambda_j t} \upkappa^{(\iota)}_{j,A}$ for all $A \in {\mz A}$.

\subsubsection*{The Banach manifold}

Choose $\ud\delta \in (0, 1]$ and $\tau \in (0, \tau(\ud\delta)/2 )$ where $\tau(\ud\delta)>0$ is the one of Theorem \ref{thm43}. Choose $B \in H_2^G\big( \wt{X}; {\mb Z}[r^{-1}]\big)$. For each broad puncture $z_j$, choose $\iota_j \in \{1, \ldots, m_{\upgamma_j}\}$ and denote
\begin{align*}
\vec\upkappa = \big( \upkappa_{\upgamma_j}^{(\iota_j)}\big)_{z_j\ {\rm broad}}.
\end{align*}

Consider the space
\begin{align*}
{\mz B} = {\mz B}_{\tau, \ud\delta}\big( B, \vec\upkappa\big) \subset {\mz A}_\tau^{1, p} \times  W_{loc}^{1, p} \big( \Sigma^*, Y  \big),
\end{align*}
consisting of pairs $(A, u)\in {\mz A}_\tau^{1, p}\times W_{loc}^{1, p}(\Sigma^*, Y)$ such that $\delta_{j,A} \in (\ud\delta, 1]$ and such that there is an $S>0$ satisfying the following conditions.

(I) For each broad $z_j$, there is a section $\wt\eta_j \in W_\tau^{1, p} \big( \Theta_+(S), \big( \wt\upkappa_{j,A}^{(\iota_j)} \big)^* T\wt{X} \big)$ such that
\begin{align}\label{equation51}
u\circ \phi_j|_{\Theta_+(S)} = \exp_{\wt\upkappa_{j, A}^{(\iota_j)}} \wt\eta_j.
\end{align}

(II) For each narrow $z_j$, there is $(\star, p_j) \in \wt{X}_j$ and $\wt\eta_j \in  W_\tau^{1, p} \big( \Theta_+(S), T_{p_j} \wt{X} \big)$ such that
\begin{align}\label{equation52}
u\circ \phi_j |_{\Theta_+(S)} = \exp_{p_j} \wt\eta_j.
\end{align}

(III) The above two conditions implies that $u$ extends to an orbifold section over $\vec{\mc C}$ and we require that $\big[ A, u\big] = B$.

From now on within this section, $\vec\upkappa$ is fixed and we abbreviate $\upkappa_j = \upkappa_{\upgamma_j}^{(\iota_j)}$, $\upkappa_{j, A} = \upkappa_{j, A}^{(\iota_j)}$. 

The space of connections ${\mz A}_\tau^{1, p}$ doesn't contain all $W_\tau^{1, p}$-connections on $P$ but it gives a constrain on the holomorphic structure. By the definition of ${\mz A}_\tau^{1, p}$, ${\mz A}_\tau^{1, p}$ is an affine space modeled on the vector space
\begin{align*}
T{\mz A}_\tau^{1, p} \simeq \Big\{   \alpha = ( \alpha_0, \alpha_1) \in W_\tau^{1, p} \big({\mf g}_0 \oplus {\mf g}_1 \big) \ |\ \alpha_0 = * d h + d f,\ f, h \in W_\tau^{2, p}({\mf g}_0)  \Big\}.
\end{align*}
Here $* d h$ is the infinitesimal change of $A$ with respect to the infinitesimal change of the Hermitian metric, and $d f$ is the infinitesimal gauge transformation. Then we have

\begin{lemma}
${\mz B}$ carries a Banach manifold structure, whose tangent space at ${\mz X}= (A, u) \in {\mz B}$ is isomorphic to 
\begin{align}\label{equation53}
T_{\mz X} {\mz B}\simeq T{\mz A}_\tau^{1, p} \oplus W_\tau^{1, p} \big( u^* T^\bot Y \big) \oplus {\mb C}^{k_n}.
\end{align}
Here $k_n$ is equal to the number of narrow punctures. 

Moreover, the group ${\mz G}_\tau = {\mz G} \cap {\mz G}_\tau^{2, p}$ acts smoothly on ${\mz B}$ such that the isomorphism (\ref{equation53}) is equivariant in a natural way.
\end{lemma}

\begin{proof}
We define an exponential map for ${\mz X} = (A, u) \in {\mz B}$ and 
\begin{align*}
\big( \alpha, \xi, {\bm \zeta} \big) \in T{\mz A}_\tau^{1, p} \oplus W_\tau^{1, p}\big( u^* T^\bot Y \big) \oplus {\mb C}^{k_n}
\end{align*}
with sufficiently small norm. This will give a local chart of the Banach manifold structure. Let $S$ and $\wt\eta_j$ be the same as in (\ref{equation51}) and (\ref{equation52}).

For a broad puncture $z_j$, since the map $A \mapsto \upkappa_{j,A}$ is smooth, for $\|\alpha\|_{W^{1, p}_\tau}$ sufficiently small, there is a unique $\xi_j(\alpha) \in T_{\upkappa_{j, A}} \wt{X}_j$ such that
\begin{align*}
\upkappa_{j, A + \alpha} = \exp_{\upkappa_{j,A}} \xi_j(\alpha).
\end{align*}
Denote $\wt\xi_j(\alpha):= e^{-\lambda t} \xi_j(\alpha)$, which is along the map $\wt\upkappa_{j,A}$. Then we can extend $\wt\xi_j$ to a vector field along $( u \circ \phi_0 )|_{\Theta_+(S)}$, by the parallel transport of $\wt\xi_j$ along the family of geodesics
\begin{align*}
x_{s, t}(\epsilon) = \exp_{\wt\upkappa_{j, A}} \epsilon \wt\eta_j (s, t),\ (s, t) \in \Theta_+,\ \epsilon \in [0,1].
\end{align*}
Denote the vector field still by $\wt\xi_j$. Then choose a cut-off function $\beta_j': \Sigma^* \to [0,1]$ vanishing outside $\Theta_+(S+1)$ and being identically $1$ on $\Theta_+(S+2)$. Then $\beta_j' \wt\xi_j$ extends to a vertical tangent vector field along $u$, denoted by the same symbol.

For a narrow puncture $z_j$, any $\zeta_j \in {\mb C} \simeq T_{(\star, p_j)} \wt{X}_j$ can be viewed as a vector field $u\circ \phi_j$. We choose a cut-off function $\beta_j'$ similarly as the above, and denote $\wt\zeta_j = \beta_j' \zeta_j \in \in W_{loc}^{1, p} \big( \Sigma^*, u^* T^\bot Y \big)$. 

We define
\begin{align*}
V(\alpha, \xi, {\bm \zeta}):= \xi + \sum_{z_j\ {\rm broad}} \beta_j' \wt\xi_j + \sum_{z_j\ {\rm narrow}} \beta_j' \wt\zeta_j.
\end{align*}
Then the section $\exp_u V(\alpha, \xi, {\bm \zeta})$ represent the same homology class $\big[ A, u\big]$. Moreover, if $\big\|\alpha\big\|_{W_\tau^{1, p}}$ is small enough, then $\delta_{j; A + \alpha} > \ud\delta$. Therefore the object $(A', u') = ( A + \alpha, u') \in {\mz B}$. This gives local charts of ${\mz B}$. The assertion about the ${\mz G}_\tau$-action is easy to check.
\end{proof}

\begin{prop}
There exist $\tau>0$ and $\ud\delta\in (0, 1]$ such that for any bounded solution $(A, u)$ to the perturbed gauged Witten equation over $\vec{\mc C}$, if $\big[ A, u \big] = B$ and for each broad puncture $z_j$, $ev_j(A, u) = \upkappa_{j, A}$, then there is a gauge transformation $g\in {\mz G}$ such that $g^*(A, u) \in {\mz B}_{\tau, \ud\delta}(B, \vec{\upkappa})$.
\end{prop}
\begin{proof}
By Theorem \ref{thm44}, such solutions have a uniform energy bound $E(B)$. Therefore there exists $\ud\delta\in (0, 1]$ such that for any such solution $(A, u)$ and any broad puncture $z_j$, $\delta_{j, A}\geq \ud\delta$. There exists $g\in {\mz G}$ such that on each cylindrical end $g^*A$ is in temporal gauge. Then by Theorem \ref{thm42}, the vortex equation, the decay of the area form $\nu$, $A\in {\mz A}_\tau^{1, p}$ for any $\tau < 2$. On the other hand, Theorem \ref{thm43} implies that $u$ converges exponentially fast to its limits at punctures. Therefore, there exists $\tau>0$ such that $g^* (A, u) \in {\mz B}_{\tau, \ud\delta}(B, \vec{\upkappa})$. 
\end{proof}

\subsubsection*{The Banach bundle and the section}

Now we consider the ${\mz G}_\tau$-equivariant Banach vector bundle ${\mz E}\to {\mz B}$ whose fibre over $(A, u) \in {\mz B}$ is
\begin{align}\label{equation54}
{\mz E}|_{(A, u)} =  L_\tau^p \big( \Lambda^{0,1}T^* \Sigma^*\otimes u^* T^\bot Y \big) \oplus  L_\tau^p ({\mf g} ). 
\end{align}
The fact that ${\mz E}$ carries a smooth Banach bundle structure over ${\mz B}$ is the same as many classical cases (for example, Gromov-Witten theory, see \cite[Section 3]{McDuff_Salamon_2004}); a local trivialization of ${\mz E}$ can be obtained by using parallel transport. The ${\mz G}_\tau$-action also lifts naturally to a linear action on ${\mz E}$, making ${\mz E}$ a ${\mz G}_\tau$-equivariant bundle.

The perturbed gauged Witten equation gives a smooth section of ${\mz E}\to {\mz B}$. More precisely, for $(A, u) \in {\mz B}$, the left-hand-side of (\ref{equation221}) defines 
\begin{align*}
{\mz W}(A, u) \in {\mz E}|_{(A, u)}.
\end{align*}
The smoothness of ${\mz W}$ relies on the smoothness of the function $A \mapsto \delta_{j,A}$ (see Definition \ref{defn215}). It is also an equivariant section by the gauge invariance property of the perturbed gauged Witten equation. Moreover, for every ${\mz X} \in {\mz B}$, $g\in {\mz G}_\tau$ and ${\mz X}'= g^* {\mz X}$, $g$ induces an isomorphism
\begin{align*}
g^*: \left( T_{\mz X} {\mz B}, E|_{{\mz X}} \right) \to \left( T_{{\mz X}'} {\mz B}, E|_{{\mz X}'} \right).
\end{align*}
This makes the isomorphism (\ref{equation53}) and (\ref{equation54}) both transform naturally.

\subsubsection*{The deformation complex and the index formula}

The linearization of ${\mz W}$ at ${\mz X}\in {\mz B}$ is a bounded linear map
\begin{align*}
d{\mz W}_{\mz X}: T_{\mz X} {\mz B} \to {\mz E}|_{{\mz X}}.
\end{align*}
With respect to (\ref{equation53}) and (\ref{equation54}), it reads, 
\begin{align}\label{equation55}
d{\mz W}_{\mz X}(\alpha, \xi, {\bm \zeta}) = \big( {\mz D}_{\mz X} V(\alpha, \xi, {\bm \zeta}) + \delta_\alpha \nabla \wt{\mc W}_A(u),\ *_c d\alpha + \sigma d\mu^* (u) V(\alpha, \xi, {\bm \zeta} )   \big).
\end{align}
Here ${\mz D}_{\mz X}: W_{loc}^{1, p} \big( u^* T^\bot Y \big) \to L_{loc}^p \big( \Lambda^{0,1} T^* \Sigma^* \otimes u^* T^\bot Y \big)$ is the linearization of $\ov\partial_A u + \nabla \wt{\mc W}_A (u)$ in the direction of $\xi$, which reads
\begin{align*}
{\mz D}_{\mz X} (V) = \ov\partial_A V + \nabla_V \nabla \wt{\mc W}_A (u);
\end{align*}
and
\begin{align*}
\delta_\alpha \nabla \wt{\mc W}_A(u) = {d\over dt}|_{t=0} \nabla {\mc W}_{A + t\alpha}(u).
\end{align*}
The second component of (\ref{equation55}) is the linearization of the left-hand-side of the vortex equation $*_c ( F_A + \mu^* (u) \nu)$, where $\nu$ is the smooth area form and $*_c$ is the Hodge-star of of the cylindrical metric; we use this modification because otherwise the $* d$ is not uniformly elliptic with respect to the cylindrical coordinates.

On the other hand, the linearization of infinitesimal gauge transformation at ${\mz X}\in {\mz B}$ is a linear operator
\begin{align*}
\begin{array}{cccc}
d{\mz G}_{\mz X}: & {\rm Lie} {\mz G}_\tau:=  W_\tau^{2, p} ({\mf g}) & \to & \big( W_\tau^{1, p} ( T^* \Sigma \otimes {\mf g} ) \big) \oplus W_\tau^{1, p} \big( u^* T^{\bot}Y  \big)\\
 &  \xi & \mapsto & \big( d \xi,  - {\mc X}_\xi \big).
\end{array}
\end{align*}
Then the {\bf deformation complex} at ${\mz X}$ is the following complex of Banach spaces
\begin{align}\label{equation56}
{\mz C}_{\mz X}: \begin{CD}
{\rm Lie} {\mz G}_\tau  @> d{\mz G}_{\mz X} > > T_{\mz X}{\mz B} @> d{\mz W}_{\mz X} >> {\mz E}|_{\mz X}.
\end{CD}
\end{align}
For any ${\mz X}\in {\mz B}$, we abbreviate $A_1 = {\rm Lie} {\mz G}_\tau $, $A_2 = T_{\mz X} {\mz B}$ and $A_3 = {\mz E}|_{{\mz X}}$. 

To state the index formula, we need to introduced some notations. For each broad puncture $z_j$, we define $b_j(\vec{\mc C}) = {\rm dim}_{\mb C} \wt{X}_j$. We define
\begin{align*}
b(\vec{\mc C}) = \sum_{z_j \ {\rm broad}} b_j(\vec{\mc C}).
\end{align*}
On the other hand, for each $j$, the normal bundle $\wt{N}_j:= \wt{N}_j \to \wt{X}_j$ splits as 
\begin{align}
\wt{N}_j = \bigoplus_{i=1}^{{\rm codim} \wt{X}_j} \wt{N}_j^{(i)}
\end{align}
where each line bundles $\wt{N}^{(i)}_j$ has an associated weight $\nu^{(i)}_j \in {\mb Z}$ such that $(\upgamma_j)^{\nu^{(i)}_j} \neq 1$. We define
\begin{align*}
n_j(\vec{\mc C}) =  - {\bm i} \sum_i \big( \nu^{(i)}_j \lambda_j - \big\lfloor \nu^{(i)}_j \lambda_j \big\rfloor \big) \in {\mb Q}_{\geq 0},\ n(\vec{\mc C}) = \sum_{j=1}^k n_j(\vec{\mc C}).
\end{align*}
Here $\lfloor a\rfloor\in {\mb Z}$ is the greatest integer which is no greater than $a\in {\mb R}$. 

Our main theorem of this section is the following.
\begin{thm}\label{thmx52}
For any bounded solution ${\mz X} \in {\mz B}$ to the perturbed gauged Witten equation, the deformation complex (\ref{equation56}) is Fredholm. That means, the image of $d{\mz G}_{\mz X}$ is a closed subspace of $T_{\mz X}{\mz B}$ and has finite codimension in ${\rm ker} \left( d{\mz W}_{\mz X} \right)$. Moreover, in this case, the Euler characteristic of ${\mz C}_{\mz X}$ is given by the formula
\begin{align}\label{equationx58}
\chi ({\mz C}_{\mz X} ) := (2-2g) {\rm dim}_{\mb C} X + 2 c_1^G \cdot \big[ {\mz X} \big] - b(\vec{\mc C}) - 2 n(\vec{\mc C}).
\end{align}
Here $c_1^G$ is the equivariant first Chern class of $T\wt{X}$, and $[{\mz X}]\in H_2^G \big( X; {\mb Z}[r^{-1}] \big)$ is the homology class of ${\mz X}$.
\end{thm}

The theorem follows immediately from the following two propositions. 
\begin{prop}\label{prop54}
For any ${\mz X} \in {\mz B}$, the complex ${\mz C}_{\mz X}$ is Fredholm if and only if the operator ${\mz D}_{\mz X}: W_\tau^{1, p}\big( u^* T^\bot Y\big) \to L_\tau^p \big( \Lambda^{0,1} T^* \Sigma \otimes u^* T^\bot Y \big)$ is Fredholm. In that case, 
\begin{align*}
\chi \big( {\mz C}_{\mz X} \big) = {\rm ind} \big( {\mz D}_{\mz X} \big) - 2 (1-g).
\end{align*}
\end{prop}

\begin{prop}\label{prop55}
For $\tau\in (0, 1)$ sufficiently small, the operator ${\mz D}_X: W_\tau^{1, p}\big( u^* T^\bot Y\big) \to L_\tau^p \big( \Lambda^{0,1} T^* \Sigma \otimes u^* T^\bot Y \big)$ is Fredholm and 
\begin{align*}
{\rm ind} \big( D_{\mz X} \big) = (2- 2g) {\rm dim}_{\mb C} \wt{X} + 2 c_1^G \cdot \big[ {\mz X} \big] - b(\vec{\mc C}) - 2n(\vec{\mc C}) - 2k_n.
\end{align*}
\end{prop}

\subsection{Proof of Proposition \ref{prop54}}

The following results about Fredholm property of complexes of Banach spaces are standard.
\begin{lemma}\label{lemma56}
Suppose that $\begin{CD}A_1 @> d_1 >> A_2 @ > d_2 >> A_3\end{CD}$ is a complex of Banach spaces, and assume that there exists another Banach space $B$ and an operator $\delta_1: A_2 \to B$ such that the operator $\delta_1 d_1: A_1 \to B$ and $F= (\delta_1, d_2): A_2 \to B \oplus A_3$ are both Fredholm. Then the cohomology of the original complex is finite dimensional, and its Euler characteristic is
\begin{align*}
\chi = {\rm ind} (F) - {\rm ind} (\delta_1 d_1).
\end{align*}
\end{lemma}
\begin{lemma}\label{lemma57}
Suppose $D: A_1 \oplus A_2 \to B_1 \oplus B_2$ is a bounded operator, which is written in the matrix form as
\begin{align*}
D = \left( \begin{array}{cc} D_1 & \alpha_1 \\
                              0 & D_2 \end{array}\right)。
\end{align*}
If $D_1: A_1 \to B_1$, $D_2: A_2 \to B_2$ are both Fredholm, then $D$ is Fredholm and
\begin{align*}
{\rm ind} (D) = {\rm ind} (D_1) + {\rm ind} (D_2).
\end{align*}
\end{lemma}

\begin{proof}[Proof of Proposition \ref{prop54}]
The operator $(\alpha, \xi, {\bm \zeta})\mapsto \sigma d\mu^* (u) V(\alpha, \xi, {\bm \zeta})$ inside (\ref{equation55}) is a compact operator because it is of zero-th order and $\sigma$ converges to zero at punctures. Therefore it can be ignored when considering Fredholm properties. So we denote by
\begin{align*}
\wt{\mz D}_{\mz X}: W_\tau^{1, p} \big( T^* \Sigma \otimes {\mf g} \big) \oplus W_\tau^{1, p} \big( u^* T^\bot Y \big) \to L_\tau^p \big( \Lambda^{0,1}T^* \Sigma^* \otimes u^* T^\bot Y \big) \oplus L_\tau^p ( {\mf g})
\end{align*}
the operator defined by
\begin{align*}
\wt{\mz D}_{\mz X}\big( \alpha, \xi, {\bm \zeta} \big)  = \big( {\mz D}_{\mz X} (V(\alpha, \xi, {\bm \zeta})) + \delta_\alpha \nabla \wt{\mc W}_A(u) ,\ *_c d \alpha  \big).
\end{align*}
Then since $G$ is abelian, the modified sequence
\begin{align*}
{\mz C}_{\mz X}': \begin{CD} A_1 @> d{\mz G}_{\mz X} >> A_2 @> \wt{\mz D}_{\mz X} >> A_3\end{CD} 
\end{align*}
is still a chain complex. It has the same Euler characteristic as ${\mz C}_{\mz X}$ when one of them is Fredholm.

Now we define
\begin{align*}
\begin{array}{cccc}
\delta_1 : & A_2  &\to & L_\tau^p({\mf g})\\
& \big( \alpha , \xi, {\bm \zeta}\big) & \mapsto &  - *_c d *_c \alpha.
\end{array}
\end{align*}
Then for $h \in A_1$, we have $\delta_1 d{\mz G}_{\mz X}(h) = \Delta h$, where $\Delta: A_1 \to L_\tau^p( {\mf g})$ is the positive-definite Laplacian with respect to the cylindrical metric. Then by Lemma \ref{lemma56}, we see that ${\mz C}_{\mz X}'$ is Fredholm if and only if both $\Delta$ and ${\mz I} = (\delta_1, d{\mz W}_{\mz X}): A_2\to  A_3 \oplus L_\tau^p( {\mf g})$ are Fredholm operators. Indeed, since $\tau \in (0, 1)$, $\Delta$ is Fredholm and
\begin{align*}
{\rm ind} \Delta = -k {\rm dim}G = - \# \{{\rm punctures}\}\cdot {\rm dim} G = -2k.
\end{align*}
Therefore, if $\delta  \in (0, 1)$ and ${\mz I}$ is Fredholm, then by Lemma \ref{lemma56}, we have
\begin{align}\label{equation59}
\chi\big( {\mz C}_{\mz X}' \big) = {\rm ind} \big({\mz I} \big) - {\rm ind} \Delta = {\rm ind} \big( {\mz I} \big) + 2k.
\end{align}

Now we look at the operator ${\mz I}$, which is 
\begin{align*}
{\mz I} \left( \begin{array}{c}
\alpha  \\ \xi \\ {\bm \zeta}
\end{array}\right) \mapsto \left( \begin{array}{c}  \wt{\mz D}_{\mz X} (\alpha, \xi, {\bm \zeta}) \\ - *_c d *_c a
\end{array}        \right)= \left( \begin{array}{c} {\mz D}_{\mz X} (V(\alpha, \xi, {\bm \zeta})) + \delta_\alpha \nabla \wt{\mc W}_A(u) \\ *_c d \alpha \\  - *_c d *_c  \alpha
\end{array} \right).
\end{align*}
We claim that when $\tau \in (0,1)$, the operator $\alpha \mapsto ( *_c d\alpha, \ - *_c d *_c \alpha)$ is Fredholm and has index equal to $-2k - 2(1-g)$. By Lemma \ref{lemma57} and (\ref{equation59}), the proposition follows from this claim. 

To prove the claim, notice that the ${\mf g}_0$-component of $\alpha$, denoted by $\alpha_0$, is mapped by 
\begin{align*}
\alpha_0 = *_c d h + d f \mapsto (  *_c d *_c dh,\ - *_c d *_c df  ) = ( - \Delta h, \Delta f).
\end{align*}
It is Fredholm and has index $-2k$. On the other hand, for the ${\mf g}_1$ component of $\alpha$, denoted by $\alpha_1$, we define ${\mb R}$-linear isomorphisms $\iota_1: \Lambda^{0,1} T^* \Sigma^* \otimes_{{\mb C}} {\mf g}_1^{\mb C} \to T^*\Sigma^* \otimes_{\mb R} {\mf g}_1$ by $b\mapsto  (b + \ov{b})$ and $\iota_2: {\mf g}_1 \oplus {\mf g}_1 \to {\mf g}_1^{\mb C}$ by $\iota(a_1, a_2) = a_1 + {\bm i} a_2$. Then we have
\begin{align*}
\begin{split}
\iota_2 \big( *_c d,  - *_c d *_c \big) \iota_1 \theta = & *_c d( \theta + \ov\theta) - {\bm i}  *_c d *_c ( \theta + \ov\theta )\\
 = & *_c (\partial \theta + \ov\partial \ov\theta ) + {\bm i} ( \ov\partial^* \theta + \partial^* \ov\theta )\\
 = &\ {\bm i} \ov\partial^* \theta - {\bm i} \partial^* \ov\theta + {\bm i} \ov\partial^* \theta + {\bm i} \partial^* \ov\theta \\
 = &\ 2 {\bm i} \ov\partial^* \theta.
\end{split}
\end{align*}
Here $\partial^*$ and $\ov\partial^*$ are the adjoint of $\partial$ and $\ov\partial^*$ with respect to the cylindrical metric, respectively; the third equality follows from the K\"ahler identities on $\Sigma^*$. Therefore we see that the operator $\alpha_1 \mapsto (*_c d\alpha_1, -*_c d *_c \alpha_1)$ is Fredholm if and only if the operator
\begin{align*}
\ov\partial^*: W_\tau^{1, p} \big( \Lambda^{0,1} T^* \Sigma^* \big) \to L_\tau^p\otimes {\mb C}
\end{align*}
is Fredholm. When $\tau \in (0, 1)$, it is the case and 
\begin{align*}
{\rm ind}_{\mb R} \big( \ov\partial^* \big) = -2 (1-g).
\end{align*}
\end{proof}

\subsection{Proof of Proposition \ref{prop55}}

The proof of Proposition \ref{prop55} is a generalization of the computation of Fredholm indices in \cite{Mundet_Tian_Draft} and \cite[Section 5.1]{FJR3}.

\subsubsection*{Riemann-Roch for orbifold line bundles}

We consider a smooth Hermitian line bundle $L \to \Sigma^*$ together with a meromorphic unitary connection $A$. Suppose for each marked point $z_j$, over the cylindrical ends $U_j\simeq \Theta_+$, we choose a unitary trivialization $\xi_j: U_j \times {\mb C} \to  L|_{U_j}$ so that the connection form is
\begin{align*}
A = d + \alpha + \lambda_j dt
\end{align*}
where $\alpha \in \Omega^1(\Theta_+, {\bm i} {\mb R})$ extends to a continuous 1-form over the marked point and $\lambda_j \in {\bm i} {\mb R}$ (the residue) is a constant. $\lambda_j$ only depends on the homotopy class of the local trivialization $\xi_j$, and for different trivializations, the residues differ by an integer multiple of ${\bm i}$. $\exp (2\pi \lambda_j) \in U(1)$ is called the monodromy of the connection.

We assume that for every $z_j$, $\lambda_j \in {\bm i} {\mb Z}/ r$. Then we can define an ``orbifold completion'' ${\mc L} \to {\mc C}$ of $L \to \Sigma^*$, where ${\mc C}$ is an orbicurve obtained by adding orbifold charts near $z_j$ to $\Sigma^*$, and ${\mc L}$ is an orbifold line bundle. The orbifold degree of ${\mc L}$ is defined as follows. The trivializations ${\bm \xi}= (\xi_j)_{j=1}^k$ defines a smooth line bundle $L({\bm \xi}) \to \Sigma$. We define
\begin{align*}
{\rm deg}^{orb} {\mc L} = {\rm deg} L({\bm \xi})  - {\bm i} \sum_{j=1}^k \lambda_j \in {\mb Z}/r.
\end{align*}
We also define
\begin{align*}
\lfloor {\mc L} \rfloor = {\rm deg} L({\bm \xi}) + \sum_{j=1}^k \lfloor -{\bm i} \lambda_j \rfloor  \in {\mb Z}.
\end{align*}
Both ${\rm deg}^{orb} {\mc L}$ and $\lfloor {\mc L} \rfloor$ are independent of the choice of ${\bm \xi}$. 

Consider a class of real linear Cauchy-Riemann operators 
\begin{align*}
D: \Omega^0 (L) \to \Omega^{0,1}(L).
\end{align*}
Their Fredholm properties essentially only depends on their behavior near the punctures. 

\begin{defn}\label{defn58}
Let $L \to \Theta_+$ be a Hermitian line bundle and $D: \Omega^0(\Theta_+, L) \to \Omega^{0,1}(\Theta_+, L)$ is a real linear, first-order differential operator. $D$ is called {\bf admissible} if the following conditions are satisfied
\begin{enumerate}
\item $D - \ov\partial_A$ is a zero-th order operator for some meromorphic unitary connection $A$ on $L$.

\item If the monodromy of $A$ at the infinity of $\Theta_+$ is not 1, then $D = \ov\partial_A$. In this case we say that $D$ is of type I (at the puncture at infinity).

\item If the monodromy of $A$ at the infinity of $\Theta_+$ is 1, then there exists a trivialization $\xi: \Theta_+ \times {\mb C} \to L$ such that with respect to this trivialization, either $Df = \ov\partial f +  \tau \ov{f}$ for some $\tau >0$, or $Df = \ov\partial f$. In the first case we say that $D$ is of type ${\rm II}_1$ and in the second case we say that $D$ is of type ${\rm II}_2$.
\end{enumerate}

If $L\to \Sigma^*$ is a Hermitian line bundle and $D: \Omega^0(\Sigma^*, L) \to \Omega^{0,1}(\Sigma^*, L)$ is a real linear first-order differential operator, then we say that $D$ is admissible if its restriction to each cylindrical end $U_j \simeq \Theta_+$ is admissible in the above sense. If the restriction of $D$ to $U_j$ is of one of the three types defined above, we say that $z_j$ is a puncture of that type. We define $b({\mc L}, D) \in {\mb Z}$ be the number of type ${\rm II}_1$ punctures plus twice of the number of type ${\rm II}_2$ punctures. 
\end{defn}

We have the following index formula
\begin{prop}\label{prop57}
Suppose $D: \Omega^0(\Sigma^*, L) \to \Omega^{0,1}(\Sigma^*, L)$ is admissible. Then there exists $\tau_0>0$ such that for $\tau \in (0, \tau_0)$, the operator $D$ defines a Fredholm operator
\begin{align*}
D: W_\tau^{1, p} (L ) \to L_\tau^p (\Lambda^{0,1} T^* \Sigma^* \otimes L ).
\end{align*}
Moreover, its (real) index is given by
\begin{align*}
{\rm ind} ( D ) = 2-2 g- b({\mc L}, D)  + 2 \lfloor {\mc L} \rfloor. 
\end{align*}
\end{prop}

\begin{proof}
We can use the index gluing formula (about Cauchy-Riemann operators with totally real boundary conditions, see \cite[Appendix C]{McDuff_Salamon_2004}) to reduce the proof to a simple case. More precisely, we can cut the Riemann surface $\Sigma$ into the union of pair-of-pants, disks and cylinders, glued along common boundaries. Then the index of $D$ is the sum of the indices of Cauchy-Riemann operators $D_i$ on the $i$-th component, with totally real boundary conditions. If the component doesn't contains an original puncture, then its index formula is known. The only unknown case can be deduced from the case of an operator $D_0$ on the trivial line bundle on the sphere with only one puncture, where the puncture is either of type ${\rm II}_1$ or ${\rm II}_2$ (type I case is well-known).

In such a case $\lfloor {\mc L} \rfloor = 0$. If the puncture is of type ${\rm II}_1$, then using the cylindrical coordinates near the puncture, $D_0$ can be written as (up to a compact operator)
\begin{align*}
D_0 = {1\over 2} {\partial \over \partial s} +  {1\over 2} {\bm i}  {\partial\over \partial t} + \left( \begin{array}{cc} \tau & 0 \\ 0 & -\tau \end{array} \right).
\end{align*}
If we denote $S = \left( \begin{array}{cc} \tau & 0 \\ 0 & -\tau \end{array}\right)$, then the symplectic path $\left\{ e^{{\bm i} S t}\right\}_{t\geq 0}$ has eigenvalues $e^{\tau t}$ and $e^{-\tau t}$ which are not on the unit circle for $t >0$. Therefore, the Conley-Zehnder index of this path is zero. By the index formula for Cauchy-Riemann operators of this type, for $\delta_0>0$ small enough, $D_0$ is Fredholm and
\begin{align*}
{\rm ind} (D_0) = 1 = 2 - 2g(S^2) - 1 = 2 - 2g - b({\mc L}, D_0).
\end{align*}

If the puncture is of type ${\rm II}_2$, then $D_0$ is the same as a complex Cauchy-Riemann operator (up to a compact operator) with one point constrain. Therefore
\begin{align*}
{\rm ind} (D_0) = 2 ( 1- g(S^2) ) - 2 = 2 - 2 g(S^2) - b({\mc L}, D_0).
\end{align*}
\end{proof}

\subsubsection*{A splitting of $u^* T^\bot Y$}

For the fixed solution ${\mz X}= (A, u)\in {\mz B}$, denote $E := u^* T^\bot Y \to \Sigma^*$. Remember that the principal $G$-bundle extends to an orbifold $G$-bundle ${\mc P} \to {\mc C}$. Moreover, the section $u$ extends to an orbifold section ${\mc U}: {\mc C} \to {\mc Y}$. Similarly, we can show that $E$ extends to an orbifold vector bundle ${\mc E}\to {\mc C}$.

Now we consider the linearization ${\mz D}_{\mz X}$. The idea of computing ${\rm ind}\left( {\mz D}_{\mz X}\right)$ is that near each puncture, we can split $E$ as direct sums of line bundles, and, up to compact operators, the restriction of ${\mz D}_{\mz X}$ to each cylindrical end is the direct sums of admissible operators. Moreover, we can extend the splittings over $\Sigma^*$, i.e., we have a decomposition
\begin{align*}
E = \bigoplus_{i=1}^n L^{(i)}.
\end{align*}
Then we can show that, on each $L^{(i)}$, there is an operator $D^{(i)}$ which is an admissible Cauchy-Riemann operator on $L^{(i)}$ such that ${\mz D}_{\mz X}- \displaystyle \oplus_{i=1}^n D^{(i)}$ is compact. We carry out this idea in the following steps. Similar procedures are used in \cite{Mundet_Tian_Draft}.

{\bf Step 1.} First we examine the operator ${\mz D}_{\mz X}$ around each puncture $z_j$, with monodromy $\upgamma_j \in {\mb Z}_r$. With respect to the trivialization $\phi_j$, $u$ is identified with a map $u_j: \Theta_+ \to X$ and the connection is identified with a 1-form $\phi ds + \psi dt + \lambda_j dt$ for $\phi, \psi: \Theta_+ \to \mf g$. Since $\displaystyle \lim_{s \to +\infty} u_j(s, t) = v_j(t):=e^{ - \lambda_j t} \upkappa_j$. For the purpose of studying Fredholm properties of ${\mz D}_{\mz X}$, we can deform ${\mz X} = (A, u)$ such that over $\Theta_+$, $u_j(s, t) = e^{ - \lambda_j t} \upkappa_j$, and $A = d + \lambda_j dt$. Then, after this modification, we have
\begin{align*}
{\mz D}_{\mz X} \xi = \ov\partial \xi + {1\over 2} \nabla_\xi {\mc X}_{\lambda_j}(u_j) +  \sum_{l=0}^s e^{- \rho_l(\lambda_j t)}\nabla_\xi \nabla F_l^{(\delta_{j, A})}(u_j).
\end{align*}
Denote $W_j = \sum_{l=0}^s F_{j; l}^{(\delta_{j, A})}$.

{\bf Step 2.} Now we see that on $U_j$ we have an $S^1$-equivariant splitting $v_j^* T \wt{X} \simeq v_j^* T \wt{X}_j \oplus v_j^* \wt{N}_j$. Moreover, since $dF_l$ vanishes along the normal bundle $\wt{N}_j$, the operator ${\mz D}_{\mz X}$ splits over $U_j$ as the direct sum of two operators
\begin{align*}
\begin{array}{cccc}
{\mz D}_j^T: & \Gamma \big( \Theta_+, v_j^* T\wt{X}_j  \big) &\to & \Omega^{0,1} \big( \Theta_+, v_j^* T\wt{X}_j \big),\\
{\mz D}_j^N: & \Gamma \big( \Theta_+, v_j^* \wt{N}_j \big)  &  \to & \Omega^{0,1} \big( \Theta_+, v_j^* \wt{N}_j \big).
\end{array}
\end{align*}

{\bf Step 3.} We consider the tangential part ${\mz D}_j^T$. If $z_j$ is narrow, then $W_j|_{\wt{X}_j} \equiv 0$. In this case ${\mz D}_j^T$ is the same as a usual homogeneous Cauchy-Riemann operator. We trivialize $v_j^* T\wt{X}_j$ over $\Theta_+$ so that we can write
\begin{align*}
v_j^* T\wt{X}_j \simeq \bigoplus_{\nu=1}^{b_j} L^{(\nu)}
\end{align*}
and the restriction of ${\mz D}_j^T$ to $U_j$ is the direct sum of $D_j^{(\nu)}: \Omega^0( U_j, L^{(\nu)}) \to \Omega^{0,1}(U_j, L^{(\nu)})$. Here each $D_j^{(\nu)}$ is of type ${\rm II}_2$ in the sense of Definition \ref{defn58}.

If $z_j$ is broad, then $W_j|_{\wt{X}_j}$ is a holomorphic Morse function. The Hessian of $\wt{W}_j$ at $\upkappa_{j, A}$ is a real quadratic form $H_j$ on $T_{\upkappa_j} \wt{X}_j$ satisfying $H_j( \cdot, \cdot) = -H_j(J \cdot, J \cdot)$. Then we have decomposition of $T_{\upkappa_j} \wt{X}_j$ into complex lines
\begin{align*}
T_{\upkappa_j} \wt{X}_j \simeq \bigoplus_{\nu =1}^{b_j} Z^{(\nu)}
\end{align*}
with respect to which the Hessian is diagonalized. On each $Z^{\nu}$, $H_j$ has eigenvalues $\pm b_\nu$ for some $b_\nu > 0$. The path of diffeomorphisms $e^{\lambda_j t}$ induces a trivialization of $v_j^* T\wt{X}_j$ along $S^1$. Therefore we have a trivialization $U_j \times T_{\upkappa_j} \wt{X}_j  \to v_j^* T\wt{X}_j$, which is well-defined since $\wt{X}_j$ is fixed by $\upgamma_j$. With respect to this trivialization, ${\mz D}_j^T$ splits as the direct sum of operators
\begin{align*}
D_j^{(\nu)}: \Omega^0 \big( \Theta_+, L_j^{(\nu)} \big)  \to \Omega^{0,1} \big( \Theta_+, L_j^{(\nu)} \big),\ \nu =1, \ldots, b_j.
\end{align*}
Each $D_j^{(\nu)}$ is of type ${\rm II}_1$ in the sense of Definition \ref{defn58}.

{\bf Step 4.} Now we consider the normal component ${\mz D}_j^N$. By ({\bf P2}) of Hypothesis \ref{hyp25} and ({\bf Q2}) of Hypothesis \ref{hyp28}, the Hessian of $W_j$ vanishes in the normal direction. Therefore, 
\begin{align*}
{\mz D}_j^N \xi = \ov\partial \xi + {1\over 2} \nabla_\xi {\mc X}_{\lambda_j}(u_j).
\end{align*}
On the other hand, we have the splitting of normal bundles 
\begin{align*}
\wt{N}_j \simeq \bigoplus_{i=b_j+1}^n \wt{N}^{(i)}_j,
\end{align*}
where each $\wt{N}^{(i)}$ is an $S^1$-equivariant line bundle over $\wt{X}_j$. If we denote $L^{(i)}_j= v_j^* \wt{N}^{(i)}_j$, then ${\mz D}_j^N$ splits as the direct sum of Cauchy-Riemann operators $D_j^{(i)}: \Omega^0(\Theta_+, L^{(i)}_j)\to \Omega^{0,1}(\Theta_+, L^{(i)}_j)$. Each $D_j^{(i)}$ is of type I in the sense of Definition \ref{defn58}.

{\bf Step 5.} So far, for each cylindrical end, we have constructed a splitting
\begin{align}\label{equation510}
E|_{U_j} = v_j^* T \wt{X} \simeq \bigoplus_{i=1}^n L^{(i)}_j
\end{align}
and differential operators
\begin{align*}
{\mz D}^{(i)}_j : \Omega^0 \big( \Theta_+, L^{(i)}_j \big) \to \Omega^{0,1} \big( \Theta_+,  L^{(i)}_j \big)
\end{align*}
such that ${\mz D}_{\mz X} - \bigoplus_{i=1}^n {\mz D}^{(i)}_j$ is a compact operator. We claim that the union of the splittings over $\cup_{j=1}^k U_j$ can be extended to whole $\Sigma^*$. 

Indeed, over $\Sigma^* \setminus \cup_{j=1}^k U_j$ the bundle $E$ is trivial. Choosing a trivialization, the splitting (\ref{equation510}) induces a smooth map from $\partial \left( \Sigma^* \setminus \cup_{j=1}^l U_j \right)$ to the flag manifold ${\rm Flag}({\mb C}^n)$. Since ${\rm Flag}({\mb C}^n)$ is simply-connected, this map can be smoothly extended to $\Sigma^* \setminus \cup_{j=1}^l U_j$, which means we extend the splitting (\ref{equation510}) to the interior. 

Then we obtained a splitting of $E$ as direct sum of line bundles $L^{(i)} \to \Sigma^*$ for $i = 1, \ldots, n$. The differential operators ${\mz D}^{(i)}_j$ on $L^{(i)}|_{U_j} = L_j^{(i)}$ can be extended smoothly to ${\mz D}^{(i)}: \Omega^0 (L^{(i)}) \to \Omega^{0,1}(L^{(i)})$, while the ambiguities of the extensions are compact operators. By our construction in previous steps, ${\mz D}^{(i)}$ is admissible in the sense of Definition \ref{defn58}. Apply Proposition \ref{prop57} to each ${\mz D}^{(i)}$, we see that there exists $\delta_0>0$ such that for all $\delta \in (0, \delta)$, each ${\mz D}^{(i)}$ induces a Fredholm operator
\begin{align*}
{\mz D}^{(i)}: W_\tau^{1, p} \big( L^{(i)} \big) \to L_\tau^p \big( \Lambda^{0,1}\otimes L^{(i)} \big).
\end{align*}
Moreover, each $L^{(i)}$ extends to an orbi-bundle ${\mc L}^{(i)} \to {\mc C}$ and 
\begin{align}\label{equation511}
\begin{split}
{\rm ind} \big( {\mz D}_{\mz X} \big) = &\ \sum_{i=1}^n {\rm ind} \big( {\mz D}^{(i)} \big) \\
= &\ \sum_{i=1}^n \big( 2- 2g - b ( {\mc L}^{(i)}, {\mz D}^{(i)} )  + 2 \big\lfloor {\mc L}^{(i)} \big\rfloor \big) \\
= &\ (2- 2g) {\rm dim}_{\mb C} \wt{X} - b(\vec{\mc C}) + 2 \sum_{i=1}^n \big\lfloor  {\mc L}^{(i)} \big\rfloor.
\end{split}
\end{align}
Proposition \ref{prop55} follows by noticing that the sum of all $\lfloor {\mc L}^{(i)} \rfloor$ is equal to 
\begin{align*}
c_1^G \cdot \big[ {\mz X} \big] - \sum_{z_j\ {\rm narrow}} \big( n_j(\vec{\mc C}) + 1 \big).
\end{align*}

\section{Stable solutions and the compactness theorem}\label{section6}

From this section on we start to consider the compactification of the moduli space of the perturbed gauged Witten equation. 

\subsection{Solitons}

\begin{defn}
Let $\delta \in (0, 1]$, $\lambda \in {\bm i}[0, 1) \cap ({\bm i} {\mb Z}/r)$ and $\upgamma = \exp (2\pi \lambda)$. We have the function $\wt{W}_\upgamma^{(\delta)}: \wt{X} \to {\mb C}$ introduced in (\ref{equation25}) and $\wt{W}_{\lambda}^{(\delta)}: \Theta \times \wt{X} \to {\mb C}$ given by
\begin{align}
 \wt{W}_{\lambda}^{(\delta)} (s, t, x) =  \sum_{l=0}^s e^{\rho_l(\lambda t)} F_{\upgamma; l}^{(\delta)} (x).
\end{align}
Consider the equation for a map $u: \Theta \to \wt{X}$
\begin{align}\label{equation62}
\partial_s u + J \big( \partial_t u + {\mc X}_\lambda(u) \big) + 2 \nabla \wt{W}_{\lambda}^{(\delta)} (u) = 0.
\end{align}
The energy of a solution $u$ is defined as 
\begin{align*}
E(u) = {1\over 2} \big\| \partial_s u \big\|_{L^2(\Theta)}^2 + {1\over 2} \big\| \partial_t u  + {\mc X}_{\lambda}(u) \big\|_{L^2(\Theta)}^2 + \big\| \nabla \wt{W}_\lambda^{(\delta)} (u) \big\|_{L^2(\Theta)}^2.
\end{align*}
A solution $u$ to (\ref{equation62}) whose energy is finite and whose image has compact closure is called a $(\lambda, \delta)$-soliton, or simply a soliton. A soliton having nonzero energy is called nontrivial, otherwise it is called trivial. 
\end{defn}

We see that the data $(\lambda, \delta)$ naturally gives a $\pm\lambda$-cylindrical model of the perturbed gauged Witten equation with parameters $(\sigma = 0, \delta)$ (on the positive part $\Theta_+$ and the negative part $\Theta_-$, respectively). If $u$ is a soliton, then the restriction of $(u, 0)$ to $\Theta_\pm$ is a bounded solution to the corresponding cylindrical model. Then by Theorem \ref{thm42}, for any $(\lambda, \delta)$-soliton, there exist $\upkappa_\pm \in \wt{X}_\upgamma$ such that
\begin{align*}
\lim_{s\to \pm \infty} e^{\lambda t} u(s, t) = \upkappa_\pm.
\end{align*}
We define the {\bf evaluation} of the soliton by $u_\pm = \upkappa_\pm$.

\begin{lemma}\label{lemma62}
If $\upgamma = \exp(2\pi \lambda)$ is narrow, then every $(\lambda, \delta)$-soliton is trivial.
\end{lemma}

\begin{proof}
Abbreviate $\wt{W}_\lambda = \wt{W}_{\lambda}^{(\delta)}$ and $\wt{W} = \wt{W}_\upgamma^{(\delta)}$. Define $v: \Theta \to X$ by $v(s, t) = e^{r \lambda t} u( rs, rt)$. Then
\begin{align*}
\begin{split}
\partial_s v(s, t) + J \partial_t v (s, t) = &\ m ( e^{r \lambda t} )_* \big( \partial_s u (rs, rt) + J (\partial_t u (rs, rt) + {\mc X}_\lambda(u(rs, rt)) ) \big) \\
= &\ -2  m (e^{r\lambda t})_* \nabla \wt{W}_\lambda(rs, rt, u(rs, rt)) = -2m \nabla \wt{W}( v(s, t)) .
\end{split}
\end{align*}
Moreover, $v(s, \cdot)$ converges uniformly to $\upkappa_\pm$ as $s\to \pm\infty$. Therefore
\begin{align}\label{equation63}
\begin{split}
2 m \big\| \nabla \wt{W} (v) \big\|_{L^2(\Theta)}^2 = &\ - \int_\Theta \big\langle \partial_s v + J \big( \partial_t v \big) , \nabla \wt{W} (v) \big\rangle dsdt\\
= &\ -  \int_\Theta d \wt{W} \cdot \big( \partial_s v + J \big( \partial_t v  \big) \big) ds dt \\
= &\ -2  \int_\Theta {\partial \wt{W} (v)  \over \partial \ov{z}} ds dt \\
= &\  2\pi \big( \wt{W}(\upkappa_-) - \wt{W} (\upkappa_+) \big).
\end{split}
\end{align}
Since $\upgamma$ is narrow, $\wt{W} (\upkappa_-) =  \wt{W}(\upkappa_+) = 0$. Then $v$ is holomorphic. Since $v$ is a multiple cover of $u$, $v$ has finite energy. By removal of singularity, $v$ extends to a holomorphic sphere. By ({\bf X1}) of Hypothesis \ref{hyp21}, $(\wt{X}, \omega)$ is aspherical, $v \equiv \upkappa_\pm$ is a constant. So $u$ has zero energy. \end{proof}

On the other hand, if $\upgamma$ is broad, then similar to (\ref{equation63}),
\begin{align*}
\big\| \nabla \wt{W}_{\lambda}^{(\delta)} (u) \big\|_{L^2(\Theta)}^2 = 2\pi \big( \wt{W}_\upgamma^{(\delta)}(\upkappa_-) - \wt{W}_\upgamma^{(\delta)} (\upkappa_+) \big).
\end{align*}
In particular, this implies that ${\rm Im} \wt{W}_\upgamma^{(\delta)} (\upkappa_-) = {\rm Im}\wt{W}_\upgamma^{(\delta)} (\upkappa_+)$.

A stable $(\lambda, \delta)$-soliton is a finite sequence
\begin{align*}
{\bm u} = ( u_1, \ldots, u_\nu )
\end{align*}
where for each $\alpha = 1, \ldots, \nu$, $u_\alpha$ is a {\it nontrivial} $(\lambda, \delta)$-soliton such that 
\begin{align*}
( u_\alpha )_+ = ( u_{\alpha+1})_- \in \wt{X}_{\upgamma},\  \alpha = 1, \ldots, \nu-1.
\end{align*}

\subsection{Stable solutions and convergence}

Let $\vec{\mc C}$ be a rigidified $r$-spin curve with punctures $z_1, \ldots, z_k$.

\begin{defn}
A {\bf stable solution} to the perturbed gauged Witten equation over $\vec{\mc C}$ is a triple
\begin{align*}
\big( A, u, \{ {\bm u}_{j}\}_{z_j\ {\rm broad}} \big)
\end{align*}
where
\begin{enumerate}
\item $(A, u)$ is a bounded solution to the perturbed gauged Witten equation on  $\vec{{\mc C}}$.

\item For each broad puncture $z_j$, ${\bm u}_j = ( u_{j;1}, \ldots, u_{j; \nu_j})$ is a stable $(\lambda_j, \delta_j)$-soliton, where $\lambda_j$ is the residue of the $r$-spin structure at $z_j$ and $\delta_j = \delta_{j, A} \in (0, 1]$ introduced in (\ref{equation220}).

\item If $\nu_j \geq 1$, then  
\begin{align*}
ev_j (A, u) = ( u_{j; 1} )_- \in {\rm Crit} \Big( \wt{W}^{(\delta_j)}_{\upgamma_j} |_{\wt{X}_{\upgamma_j}} \Big).
\end{align*}
\end{enumerate}
\end{defn}

Now we can define the topology in the space of stable solutions. At the ``tails'', the convergence of the stable solitons are just an $A$-parametrized version of convergence of stable solutions to the corresponding Floer type equation (\ref{equation62}), because in $\nabla \wt{W}_{\lambda_i}^{(\delta_j)}$, the parameter $\delta_j$ depends on $A$. Therefore it suffices to define the convergence of a sequence of usual solutions over $\vec{\mc C}$ to a stable solution. The definition in the rest of cases can be easily written down and we omit it.
\begin{defn}\label{defn64}
Suppose $( A^{(i)}, u^{(i)} )$ is a sequence of solutions to the gauged Witten equation over a fixed rigidified $r$-spin curve $\vec{{\mc C}}$ with underlying punctured Riemann surface $\Sigma^*$. We say that the sequence converges to a stable solution $( (A, u), \{{\bm u}_j\}_{z_j\ {\rm broad}} )$ if the following conditions are satisfied. 

\begin{enumerate}
\item $( A^{(i)}, u^{(i)} )$ converges to $(A, u)$ in $W^{1, p}_{loc}$-topology.

\item For each broad puncture $z_j$, If ${\bm u}_j = (u_{j;1}, \ldots, u_{j; \nu_j} )$ and $\nu_j \geq 1$, then the following conditions are satisfied.
\begin{itemize}
\item For $\alpha = 1, \ldots, \nu_j$, there are sequences $s_{\alpha}^{(i)}>0$ such that  
\begin{align*}
\lim_{i \to +\infty} s_{\alpha}^{(i)} = +\infty,\ \alpha> \alpha' \Longrightarrow \lim_{i \to +\infty} s_{\alpha}^{(i)} - s_{\alpha'}^{(i)}= +\infty.
\end{align*}
\item Let $(u^{(i)}, h^{(i)})$ be the sequence of solutions to the cylindrical model obtained by restricting $(A^{(i)}, u^{(i)})$ to $U_j(2)$. Then for each $\alpha$, the sequence $u^{(i)} ( s_{\alpha}^{(i)} + \cdot, \cdot )$ converges to $u_{j;\alpha}$ uniformly on any compact subset of $\Theta$.

\item We have
\begin{align*}
\lim_{s \to +\infty} \limsup_{i \to \infty} E \big( A^{(i)}, u^{(i)}; U_j (s_{\nu_j}^{(i)} + s) \big) = 0.
\end{align*}
\end{itemize}

\item For each broad puncture $z_j$, if $\nu_j = 0$, then 
\begin{align*}
\lim_{s \to +\infty} \limsup_{i \to \infty} E \big( A^{(i)}, u^{(i)}; U_j (s) \big) = 0.
\end{align*}
\end{enumerate}
\end{defn}

Now we state the compactness theorem.
\begin{thm}\label{thm65}
If $(A^{(i)}, u^{(i)} ) \in {\mz A} \times \Gamma(Y)$ is a sequence of smooth bounded solutions to the gauged Witten equation (\ref{equation221}) with
\begin{align*}
\sup_i E ( A^{(i)}, u^{(i)} ) < \infty,
\end{align*}
then there is a subsequence (still indexed by $i$), a stable solution $( (A, u), \{{\bm u}_j\}_{z_j\ {\rm broad}})$, and a sequence of smooth gauge transformations $g^{(i)}\in {\mz G}$ such that $( g^{(i)} )^* ( A^{(i)}, u^{(i)} )$ converges to $( (A, u), \{{\bm u}_j\}_{z_j\ {\rm broad}} )$ in the sense of Definition \ref{defn64}.
\end{thm}

\section{Energy quantization in blowing up}\label{section7}

Now we start to prove the compactness theorem of the moduli space of gauged Witten equation. The first main concern is about the uniform $C^0$-bound on the solutions, which is the prerequisite of applying all the bubbling analysis. 

We first summarize the methods of achieving $C^0$-bound in relevant situations when the target space is noncompact. In gauged Gromov-Witten theory, one can achieve the $C^0$-bound by imposing the equivariant convexity assumption on the target space (see \cite[Page 555]{Cieliebak_Gaio_Mundet_Salamon_2002}). This is an assumption generalizing the convexity condition in \cite{Eliashberg_Gromov}. In Fan-Jarvis-Ruan's LG A-model theory, for a quasi-homogeneous superpotential $f$ on ${\mb C}^N$, the crucial condition for $C^0$-control is a growth estimate of $df$ (\cite[Theorem 5.8]{FJR1}), deduced from the nondegeneracy of the singularity.

In this paper we also imposed a convexity condition at infinity (({\bf X4}) of Hypothesis \ref{hyp21}), a more concrete form of the assumption used in \cite{Cieliebak_Gaio_Mundet_Salamon_2002}. However, since we have to perturb the equation, the solutions are no longer holomorphic and there are error terms in the estimate (see Proposition \ref{prop82}). So {\it a priori}, there could exist a sequence of bounded solutions which escape to infinity in the limit near broad punctures. One way to overcome this trouble is to establish an energy quantization property for the {\it a priori} blow-up of $C^0$-norm (Theorem \ref{thm71}). Then a $C^0$-bound follows from a local maximal principle argument.

One of the difficulty in establishing the energy quantization comes from the fact that the inhomogeneous term of the Witten equation is not bounded and not proper. The $\epsilon$-regularity argument only applies in a scale comparable to $\big| \nabla \wt{\mc W}_A \big|^{-1}$. Moreover, ${\rm Crit}W$ is the union of two parts, 
\begin{align*}
\wt{X}_B:= \big\{ (x, p)\ |\ Q(x) = 0 \big\},\ \wt{X}_S : = \big\{ (\star, p) \ |\  p \in {\mb C} \big\}.
\end{align*}
If the blow-up happens in the region away from $\wt{X}_B$ and $\wt{X}_S$, then it is easy to establish the energy quantization (Proposition \ref{prop75}); if the blow-up happens near $\wt{X}_B$ and $\wt{X}_S$, then the magnitude of $\nabla \wt{\mc W}_A$ can change dramatically and we have to use different arguments (Proposition \ref{prop74} and Proposition \ref{prop76}). 

We remark that one should be able to generalize the results of this section to the case of complete intersections, i.e., the superpotential is of the form $p_1 Q_1 + p_2 Q_2 + \ldots + p_k Q_k$ on a manifold $X \times {\mb C}^k$, where $Q_i: X \to {\mb C}$ are homogeneous functions and $p_1, \ldots, p_k$ are the complex variables of the ${\mb C}^k$-factor.

The main technical result of this section is stated in terms of local models.
\begin{thm}\label{thm71}
For each $H>0$, there exists $\epsilon_0 = \epsilon_0(H) >0$ satisfying the following condition. Suppose we have a sequence $(\beta_i, \sigma_i, \delta_i)$ of parameters of local models over $B_{r^*}$ and a corresponding sequence of solutions $(u_i, h_i)$. Suppose 
\begin{align}\label{equation71}
\lim_{i \to \infty} \big| \mu(u_i(0)) \big| = +\infty,\ \big\|h_i \big\|_{L^\infty(B_{r^*})} \leq H.
\end{align}
Then there exists a subsequence (still indexed by $i$) such that one of the following conditions holds.
\begin{enumerate}
\item We have $\displaystyle \lim_{r \to 0}\lim_{i \to \infty} E (u_i, h_i; B_r) \geq \epsilon_0$.

\item We have $\displaystyle \lim_{i \to \infty} \sigma_i = 0$ (uniformly on $B_{r^*}$) and there exists $r_0 >0$ (depending on the subsequence) such that
\begin{align}\label{equation72}
\lim_{i \to \infty} \inf_{B_{r_0} } \big| \mu(u_i) \big| = +\infty.
\end{align}
\end{enumerate}
\end{thm}
The proof is given in Subsection \ref{subsection71}

\begin{cor}\label{cor72}
For every $E>0$, there exists $\epsilon_E >0$ satisfying the following conditions. Suppose $(A_i, u_i)$ is a sequence of solutions to the perturbed gauged Witten equation over $\vec{\mc C}$ such that $E (A_i, u_i ) \leq E$. Then there exist a subsequence (still indexed by $i$), and sequences of points
\begin{align}
\{z_i^\alpha\}_{1 \leq \alpha \leq l},\ \{ z_i^{j, \beta}\}_{z_j\ {\rm broad},\ 1\leq \beta \leq l_j}
\end{align}
satisfying the following conditions (here $d$ (resp. $\wt{d}$) is the distance of the cylindrical metric on $\Sigma^*$ (resp. smooth metric on $\Sigma$)).
\begin{enumerate}
\item For each $\alpha = 1, \ldots, l$, $\lim_{i \to \infty} z_i^\alpha = z^\alpha \in \Sigma^*$ and all $z^\alpha$'s are distinct.

\item For each $\alpha = 1, \ldots, l$, we have
\begin{align*}
\lim_{i \to \infty} \big| \mu(u_i(z_i^\alpha)) \big| = +\infty,\ \lim_{r \to 0} \lim_{i \to \infty} E ( A_i, u_i; B_r(z_i^\alpha)) \geq \epsilon_E .
\end{align*}

\item For each broad puncture $z_j$ and $\beta = 1, \ldots, l_j$, $z_i^{j, \beta} \in U_j$ and 
\begin{itemize}
\item For each $\beta$, $\displaystyle \lim_{i \to \infty} \wt{d} \big( z_i^{j, \beta}, z_j\big) = 0$;

\item For any $\beta_1 \neq \beta_2$, we have $\displaystyle \liminf_{i \to \infty} d \big( z_i^{j,\beta} , z_i^{j,\beta'} \big) >0$.

\item For any $\beta$, $\displaystyle \lim_{i \to \infty} \big| \mu \big(u_i \big( z_i^{j,\beta} \big) \big) \big| = \infty$ and $\displaystyle \lim_{r \to \infty} \lim_{i \to \infty} E \big( A_i, u_i ; B_r(z_i^{j,\beta}) \big) \geq \epsilon_E$.
\end{itemize}

\item For any sequence $z_i$ of points in $\Sigma^*$, if $\displaystyle \liminf_{i \to \infty} d(z_i, z_i^\alpha) > 0$ for all $\alpha = 1, \ldots, l$, $\displaystyle \liminf_{i \to \infty} d( z_i, z_i^{j, \beta})>0$ for all broad punctures $z_j$ and $\beta = 1, \ldots, l_j$, $\displaystyle \liminf_{i \to \infty} \wt{d}(z_i, z_j)>0$ for all narrow punctures $z_j$, then 
\begin{align*}
\limsup_{i \to \infty} \big| \mu(u_i(z_i)) \big| < \infty.
\end{align*}
\end{enumerate}
\end{cor}

\begin{proof} 
For each $q \in \Sigma^*$, we can restrict $(A_i, u_i)$ to $B_{r^*} (q)$ to obtain a solution $(u_i, h_i)$ to a local model. By the equation $\Delta h_i ds dt  = F_{A_i}$, the elliptic estimate and Sobolev inequality, we see that there exists $H(E)>0$ such that for all $i$ and all $q\in \Sigma^*$,
\begin{align}\label{equation73}
\sup_q \sup_i \big\| h_i \big\|_{B_{r^*}(q)} \leq H(E).
\end{align}
Abbreviate $\epsilon_0 = \epsilon_0(H(E))$.

We construct the sequences $z_i^\alpha$ by an induction argument. We take an exhausting sequence of compact subsets of $\Sigma^*$, denoted by $K^{(l)}$, $l = 1, 2, \ldots$. We consider
\begin{align*}
 \limsup_{i \to \infty} \big\| \mu (u_i) \big\|_{L^\infty(K^{(l)})}.
\end{align*}
If it is finite, then we move on to $K^{(l+1)}$. If it is infinite, then there exist a subsequence (still indexed by $i$) and a sequence of points $q_i \in K^{(l)}$ which converges to some $q \in K^{(l)}$, such that $\displaystyle \lim_{i \to \infty} \big| \mu (u_i(q_i)) \big| = + \infty$. Then we can apply Theorem \ref{thm71} to the sequence $(u_i, h_i)$ which is the solution to the local model obtain by restricting $(A_i, u_i)$ to $B_{r^*}(q_i)$. Then there is a subsequence (still indexed by $i$) such that
\begin{align*}
\lim_{r \to 0} \lim_{i \to \infty} E ( A_i,  u_i; B_r(q_i) ) \geq \epsilon_0.
\end{align*}
(Here the second case of Theorem \ref{thm71} doesn't happen because the area form is uniformly bounded from below near $q$.)

Now we replace $\Sigma^*$ by $\Sigma^* \setminus \{q\}$, and retake an exhausting sequence of compact subsets $\{K^{(l)} \}$ of $\Sigma^* \setminus \{q\}$. We restart the induction process. It is easy to see that the induction process stops until we find a finite subset $Z = \{ z_1, \ldots, z_l\}$, a subsequence (still indexed by $i$) and sequences $z_\alpha^i$ for which (1) and (2) are satisfied for $\epsilon_E = \epsilon_0$, because the total energy of $(A_i, u_i)$ is uniformly bounded. 

Now we consider the possible blowing up at a broad puncture $z_j$. Take $S>0$ sufficiently large so that $U_j(S)\cap Z = \emptyset$. Therefore for each $K>0$, we have
\begin{align*}
\limsup_{i \to \infty} \big\| \mu (u_i) \big\|_{L^\infty([S, S+K]\times S^1)} < \infty.
\end{align*}
Now suppose there exist a subsequence (still indexed by $i$) and a sequence of points $z_i= (s_i, t_i)\in U_j(S)$ such that with 
\begin{align*}
\lim_{i \to \infty} s_i = +\infty,\ \lim_{i \to \infty} \big| \mu (u_i(z_i)) \big| = +\infty.
\end{align*}
We claim that 
\begin{align}\label{equation74}
\lim_{r \to 0} \limsup_{i \to \infty} E \left( A_i, u_i; B_r(z_i) \right) \geq \epsilon_0.
\end{align}
Suppose it is not true, then consider the subset
\begin{align*}
 \Theta^*:= \Big\{ z\in \Theta = {\mb R} \times S^1 \ |\  \limsup_{i \to \infty} \left| \mu ( z_i + z ) \right| = \infty \Big\}.
\end{align*}
Theorem \ref{thm71} implies that $\Theta^*$ has nonzero measure. There are two possibilities.

{\bf (I)} Suppose the boundary of $\Theta^*$ is a finite set, then $\Theta^*$ has infinite area. On the other hand, let $h_i = h_{A_i}: U_j \to {\mf g}^{\mb C}$ be the function defined by (\ref{equation216}). In the same way as deriving (\ref{equation73}), we may assume that 
\begin{align*}
\sup_i \big\| h_i \big\|_{L^\infty(U_j)} \leq H(E). 
\end{align*}
By the definition of $\delta_{j, i} = \delta_{j, A_i}$ (see (\ref{equation220})), we have $\inf_i \delta_{j,i} =: \underline\delta > 0$. Then take
\begin{align}
M(E):= \sup_{|h|\leq H(E),\ \ud\delta \leq \delta \leq 1} |h - \log \delta |_{\wt{X}},
\end{align}
which, by Lemma \ref{lemma23}, is finite. Let $L$ of $\Theta_+$ be a compact subset. Then for $i$ sufficiently large, for any $z \in z_i +  L$, $e^{h_i(z)} \delta_{j, i}^{-1} u_i(z) \notin \wt{K}_{\upgamma_j}$ where $\wt{K}_{\upgamma_j} \subset \wt{X}$ is the compact subset in ({\bf P5}) of Hypothesis \ref{hyp28}. Then by the definition of $|h|_{\wt{X}}$, we have
\begin{align*}
\begin{split}
\Big| \nabla \wt{\mc W}_{A_i} (u_i (z)) \Big| = &\ \Big| \sum_{l=0}^s e^{\ov{\rho_l(h_i(z))}} \nabla F_l^{(\delta_{j,i})}(u_i(z)) \Big|\\
 = &\ \Big| ( e^{h_i(z)})^* \Big( d \wt{W}_{\upgamma_j}^{(\delta_{j, i})} ( e^{h_i(z)} u_i(z))\Big) \Big|\\
= &\ \big(\delta_{j, i}\big)^r \Big| ( e^{h_i(z)} \delta_{j, i}^{-1}  )^*  \Big( d \wt{W}_{\upgamma_j} ( e^{h_i(z)} \delta_{j, i}^{-1} u_i(z) ) \Big) \Big| \\
\geq &\ \ud\delta^r M(E)^{-1} c_{\upgamma_j}.
\end{split}
\end{align*}
Here $c_{\upgamma_j} > 0$ is the one in ({\bf P5}) of Hypothesis \ref{hyp28}. Therefore  
\begin{align*}
E( A_i, u_i ) \geq E ( A_i, u_i;  z_i + L )  \geq \int_{z_i + L} \Big| \nabla \wt{\mc W}_{A_i}(u_i) \Big|^2 ds dt  \geq \ud\delta^r M(E)^{-1} c_{\upgamma_j} {\rm Area}(L).
\end{align*}
This contradicts with the energy bound because ${\rm Area} (L)$ can be arbitrarily large.

{\bf (II)} Suppose the boundary of $\Theta^*$ is an infinite set, then choose an integer $m > E/ \epsilon_0$ and $m$ distinct points $w_1, \ldots, w_m$ of $z \in \partial \Theta^*$. Then by the definition of $\Theta^*$, there exists a subsequence (still indexed by $i$) and sequences of points $w_{l, i}$, $l = 1, \ldots, m$ such that 
\begin{align*}
\lim_{i \to \infty} w_{l, i} = w_l,\ \lim_{i \to \infty} \big| \mu(u_i( z_i + w_{l, i})) \big| = +\infty.
\end{align*}
On the other hand, use the trivialization of the $G$-bundle over $U_j$, the restriction of $(A_i, u_i)$ to the disk $B_{r^*}(z_i + w_{l, i})$ is a sequence of solutions to some sequence of local models over $B_{r^*}$, which satisfy the hypothesis of Theorem \ref{thm71}. However, the second implication of Theorem \ref{thm71} doesn't holds because $w_{l, i}$ converges to a point on the boundary of $\Theta^*$. Therefore
\begin{align*}
\lim_{r \to 0} \lim_{i \to \infty} E(A_i, u_i; B_r(z_i + w_{l, i})) \geq \epsilon_0.
\end{align*}
By the choice of $m$, this contradicts with the energy bound of $(A_i, u_i)$. 

Therefore (\ref{equation74}) is true. Moreover, the set $\Theta^*$ must be finite because for any $z\in \Theta^*$, we can prove (\ref{equation74}) is true with $z_i$ replaced by $z_i + z$. The energy bound implies that such points are only of finitely many. Then we do an induction to construct a subsequence (still indexed by $i$) and sequences $z_i^{j,\beta}= (s_i^{j,\beta}, t_i^{j,\beta})$, $\beta = 1,2, \ldots$ such that 
\begin{align*}
 \lim_{i \to \infty} s_i^{j,\beta} = +\infty,\ \forall \beta \neq \beta',\ \lim_{i \to \infty} d \big( z_i^{j,\beta} , z_i^{j,\beta'} \big) > 0
\end{align*}
and 
\begin{align*}
 \lim_{i \to \infty} \big| \mu(u_i (z_i^{j,\beta})) \big| = +\infty,\ \lim_{r \to \infty} \lim_{i \to \infty} E (  A_i, u_i; B_r (z_i^{j,\beta}) ) \geq \epsilon_0.
\end{align*}
Since the energy is uniformly bounded, the induction process stops at finite time. Therefore, whenever the induction stops, the sequences $z_i^{j,\beta}$ satisfy the conditions listed in (3) of this corollary \ref{cor72} for $\epsilon_E = \epsilon_0$. Item (4) of this corollary is obvious from our construction.
\end{proof}

\subsection{Proof of Theorem \ref{thm71}}\label{subsection71}

The proof of Theorem \ref{thm71} follows from Lemma \ref{lemma73}, Proposition \ref{prop74}--\ref{prop76} below. First we introduce some notations. For any solution ${\bm u} = (u, h)$ to a local model over $B_r$, parametrized by $(\sigma, \beta, \delta)$, we have the following density functions of $(u, h)$ which are comparable to the square root of the potential density functions.
\begin{align*}
\begin{array}{lcll}
{\mf p}'(z) & := & {\mf p}'({\bm u})(z) & =\ \big| \nabla \wt{W}_h^{(\delta)} (u(z)) \big|,\\
{\mf p}''(z) & := & {\mf p}''({\bm u})(z) & =\ \sqrt{\sigma(z)} \big|\mu(u(z)) \big|,\\
{\mf p}(z) & := & {\mf p}( {\bm u} )(z)  & =\ {\mf p}'(z) + {\mf p}''(z).
\end{array}
\end{align*}

\begin{lemma}\label{lemma73}
There exist $\epsilon_1>0$, $C_1>0$, $r_1 \in (0, r^*]$, $\lambda_1 \in (0,{1\over 2}]$ and for $p>2$, $c_{(p)}>0$, satisfying the following conditions. Suppose $r\in (0, r_1]$ and $( u, h)$ is a solution to the local model on $B_r$ parametrized by $(\beta, \sigma, \delta)$. If
\begin{align}\label{equation76}
r\sup_{z\in B_r} {\mf p}(z) \leq  1,\ E(u, h; B_r) \leq \epsilon_1,
\end{align}
then there exists a gauge transformation $g: B_r \to G$ such that if we denote by $(u', h') = g^* (u, h)$ and $\phi' + {\bm i} \psi' = 2( \partial h'/ \partial \ov{z})$, then for every $\lambda \in (0, \lambda_1]$, we have
\begin{align}\label{equation77}
\begin{split}
{\rm diam} \big( u' (B_{\lambda r}) \big) \leq &\  C_1 \Big( \big\| d_A u \big\|_{L^2(B_r)} + \sqrt{r} +  \lambda \Big), \\
\big\| d_A u \big\|_{L^p(B_{\lambda r})} \leq &\ c_{(p)} (\lambda r)^{{2\over p} - 1}  \Big( \big\| d_A u \big\|_{L^2(B_r)} + \sqrt{r} + \lambda \Big).
\end{split}
\end{align}
\end{lemma}

Let $K_Q \subset X$ be the subset in ({\bf Q1}) of Hypothesis \ref{hyp25} and $\wt{K}_0 = \big\{ (x, p) \in \wt{X} \ |\ x \in K_Q,\ |p| \leq 1 \big\}$. For any $D>0$, denote 
\begin{align*}
\wt{X}_B^D = \big\{ x \in \wt{X} \setminus \wt{K}_0 \ |\  d(x, \wt{X}_B) \leq D \big\},\ \wt{X}_S^D = \big\{ x \in \wt{X} \setminus \wt{K}_0\ |\ d(x, \wt{X}_S) \leq D \big\}.
\end{align*}

\begin{prop}\label{prop74}
For each $H>0$ there exist $\epsilon_2 = \epsilon_2(H)>0$, $M_2 = M_2(H)>0$ and $D = D(H)>0$ satisfying the following condition. 

Suppose $r \in (0, r^*]$ and $(u, h)$ is a solution to the local model with parameter $(\beta, \sigma, \delta)$ on $B_r$ with finite energy. If $u(0) \in \wt{X}_B^D$ and 
\begin{align}\label{equation79}
\big\| h \big\|_{L^\infty(B_r)} \leq H,\ {\mf p}(0) \geq {1\over 2} \sup_{B_r} {\mf p} \geq M_2,\ {\mf p}'(0) \geq {\mf p}''(0),\ r {\mf p}(0) = \lambda \in (0, 1],
\end{align}
then we have
\begin{align*}
E( u, h; B_r ) \geq \epsilon_2 \lambda^2.
\end{align*}
\end{prop}

The next proposition considers the case that the image of $u$ is away from $\wt{X}_S \cup \wt{X}_B$. 
\begin{prop}\label{prop75}
There exist $\epsilon_3>0$ and for each $H>0$, a constant $M_3 = M_3(H) > 0$ satisfying the following condition. Suppose $(u, h)$ is a solution the local model with parameter $(\beta, \sigma, \delta )$ on $B_r$ with $r \in (0, r^*]$, such that 
\begin{align}\label{equation710}
\big\| h\big\|_{L^\infty(B_r)} \leq H,\ {\mf p}(0) \geq {1\over 2} \sup_{ z\in B_r } {\mf p}(z)\geq M_3,\ r {\mf p}(0) = \lambda  \in (0, 1];
\end{align}
and such that either of the following two conditions are satisfied:
\begin{align}\label{equation711} 
{\rm (I)}\ u(0) \notin \wt{X}_B^{D} \cup \wt{X}_S^{D}, \hspace{1cm} {\rm (II)}\ {\mf p}''(0) \geq {\mf p}'(0).
\end{align}
Here $D = D(H) >0$ is the one of Proposition \ref{prop74}. Then
\begin{align*}
E ( u, h; B_r ) \geq \epsilon_3 \lambda^2. 
\end{align*}
\end{prop}

The remaining case is that the blow up happens near $\wt{X}_S= \{ (\star, p)\ |\ p \in {\mb C}\}$. The function $W$ is degenerate along the normal direction of $\wt{X}_S$ and we couldn't find a straightforward argument to deal with this situation. Instead we have the following proposition, from whose proof one can see that it is a corollary to the above two propositions.	

\begin{prop}\label{prop76}
For each $H>0$ there exists $\epsilon_4 = \epsilon_4(H) >0$ satisfying the following condition.

Suppose $(\beta_i, \sigma_i, \delta_i)$ is a sequence of parameters of local models of gauged Witten equation over $B_r$ with $r \in (0, r^*]$ and suppose $(u_i, h_i)$ are a sequence of corresponding solutions. Suppose
\begin{align*}
\lim_{i \to \infty} {\mf p}_i(0) = \infty,\ \big\| h_i \big\|_{L^\infty(B_r)} \leq H
\end{align*}
and 
\begin{align*}
{\mf p}_i'(0) \geq {\mf p}_i''(0),\ {\mf p}_i(0) \geq {1\over 2} \sup_{B_{r_i}} {\mf p}_i, u_i(0) \in \wt{X}_S^D.
\end{align*}
Here $r_i = {\mf p}_i(0)^{-1}$ and $D = D(H)$ is the one of Proposition \ref{prop74}. Then there exists a subsequence (still indexed by $i$) such that one of the following conditions holds
\begin{enumerate}
\item We have
\begin{align}\label{equation712}
\lim_{r \to \infty} \lim_{i \to \infty} E ( u_i, h_i; B_r ) \geq \epsilon_4.
\end{align}

\item $\displaystyle \lim_{i \to \infty} \sigma_i = 0$ uniformly on $B_{r^*}$ and there exists $\tau>0$ (which may depend on the subsequence) such that
\begin{align}\label{equation713}
\lim_{i \to \infty} \inf_{B_\tau} \big| \mu( u_i) \big| = + \infty.
\end{align}
\end{enumerate}
\end{prop}

\begin{proof}[Proof of Theorem \ref{thm71}]
First, if there exists a subsequence (still indexed by $i$) and a sequence $z_i \to 0$ such that $\lim_{i \to \infty} {\mf p}_i(z_i) = \infty$, then the conclusion holds according to Proposition \ref{prop74}, \ref{prop75} and \ref{prop76}. Indeed, let $r_i:= {\mf p}_i(0)^{-1}$ which converges to zero. Apply Hofer's lemma (Lemma \ref{lemmaa7}) to the function ${\mf p}_i: B_{r_i}(z_i) \to {\mb R}$. Then there exist a point $y_i \in B_{r_i}(z_i)$ and $\delta_i \in (0, r_i/ 2]$ such that
\begin{align*}
{\mf p}_i(y_i) \geq {1\over 2} \sup_{B_{\delta_i}(y_i)}{\mf p}_i,\ \delta_i {\mf p}_i(y_i) \geq {r_i \over 2} {\mf p}_i(z_i) = {1\over 2}.
\end{align*}
By taking a subsequence, we may assume that either ${\mf p}_i''(y_i) \geq  {\mf p}_i'(y_i)$, or ${\mf p}_i'(y_i) \geq {\mf p}_i''(y_i)$. In the former case, for $i$ large enough, (\ref{equation710}) and (II) of (\ref{equation711}) are satisfied by $(u_i, h_i)$ over $B_{\delta_i}(y_i)$. Then by Proposition \ref{prop75}, for $i$ large enough, we have
\begin{align*}
E(u_i, h_i; B_{\delta_i}(y_i)) \geq \epsilon_3 \big( \delta_i {\mf p}_i(y_i) \big)^2 \geq {\epsilon_3  \over 4}.
\end{align*}
In the latter case, by taking a subsequence, we have three distinct possibilities. 
\begin{enumerate}
\item For all $i$, $u_i(y_i) \in \wt{X}_B^D$. Then by Proposition \ref{prop74}, for $i$ large enough (so that (\ref{equation79}) is satisfied by $(u_i, h_i)$ over $B_{\delta_i}(y_i)$), we have
\begin{align*}
E(u_i, h_i; B_{\delta_i}(y_i)) \geq \epsilon_2 \big( \delta_i {\mf p}_i(y_i) \big) \geq {\epsilon_2 \over 4}.
\end{align*}

\item For all $i$, $u_i(y_i) \notin \wt{X}_B^D \cup \wt{X}_S^D$, Then by Proposition \ref{prop75}, for $i$ large enough (so that (\ref{equation710}) and (I) of (\ref{equation711}) are satisfied by $(u_i, h_i)$ over $B_{\delta_i}(y_i)$), we have
\begin{align*}
E(u_i, h_i; B_{\delta_i}(y_i)) \geq \epsilon_3 \big( \delta_i {\mf p}_i(y_i) \big)^2 \geq {\epsilon_3 \over 4}.
\end{align*}

\item For all $i$, $u_i(y_i) \in \wt{X}_S^D$. Then since $\lim_{i \to \infty} y_i = 0$, we can choose $r'$ small enough so that $B_{r'}(y_i) \subset B_r$. Then we can apply Proposition \ref{prop76} to $(u_i, h_i)$ restricted to $B_{r'}(y_i)$. It implies that by taking a further subsequence, we have either
\begin{align*}
\lim_{r \to 0}\lim_{i \to \infty} E(u_i, h_i; B_r(y_i)) \geq \epsilon_4,
\end{align*}
or $\displaystyle \lim_{i \to \infty} \sigma_i = 0$ uniformly and there is $\tau >0$ (depending on the subsequence) such that 
\begin{align*}
\lim_{i \to \infty} \inf_{B_{\tau} (y_i)} \big| \mu(u_i) \big| = \infty.
\end{align*}
\end{enumerate}
Since $y_i \to 0$, this implies the conclusion for $\epsilon_0 = \min\{ \epsilon_2/4, \epsilon_3/4, \epsilon_4 \}$ and $\tau_0 = \tau$. 

It remains to consider the case that ${\mf p}_i$ doesn't blow up at $0$. Then we can assume that there exist a subsequence (still indexed by $i$) and $\tau >0$ (which depends on the subsequence) such that
\begin{align*}
\limsup_{i \to \infty} \big\|{\mf p}_i \big\|_{L^\infty(B_\tau)} = M < \infty. 
\end{align*}
Then we can take $\tau$ smaller than both the $r_1$ of Lemma \ref{lemma73} and ${1\over M}$. By taking a subsequence, we can assume for all $i$, either $E(u_i, h_i; B_\tau ) > \epsilon_0$ or $E(u_i, h_i;  B_\tau ) \leq \epsilon_0$. In the former case the current lemma is proven; in the latter case, (\ref{equation76}) is satisfied and by Lemma \ref{lemma73} and (\ref{equation77}), $(u_i, h_i)$ is gauged equivalent to some $(u_i', h_i')$ such that 
\begin{align*}
{\rm diam} (u' (B_{\lambda_0 \tau })) \leq C,
\end{align*}
where $\lambda_0 \in (0, {1 \over 2}]$ is the one in Lemma \ref{lemma73} and $C$ is a constant, independent of the sequence. Thus this implies $u_i(B_{\lambda_0 \tau })$ escape to infinity uniformly. Thus (\ref{equation72}) is true and necessarily $\sigma_i$ should converges to zero uniformly.
\end{proof}

\subsection{Proof of Lemma \ref{lemma73}}

Let $\phi, \psi: B_r \to {\mf g}$ be the functions defined by $\phi + {\bm i} \psi = 2(\partial h / \partial \ov{z})$. We pull back $(u, h)$ and $\phi ds + \psi dt$ via the rescaling $B_1 \to B_r$ given by $w \mapsto z= rw$, which is denoted by $( u_r, h_r)$,and $\phi_r ds + \psi_r dt$.  Denote $A_r= d + \phi_r ds + \psi_r dt$. Then
\begin{align*}
 \partial \psi_r - \partial_t \phi_r +  r^2 \sigma \mu^* (u_r) = 0.
\end{align*}
Hence by (\ref{equation76}) we have
\begin{align*}
\big\| F_{A_r} 	\big\|_{L^\infty(B_1)} = r^2 \big\| \sigma \mu (u_r) \big\|_{L^\infty(B_r)} \leq r^2 \Big( \sup_{B_r} \sqrt{\sigma} \Big) \Big( \sup_{B_r} {\mf p}'' \Big)\leq r \sqrt{\sigma^+}.
\end{align*}
There exists $f: B_1 \to {\mf g}$ solving the Neumann boundary value problem
\begin{align*}
\Delta f = d^* (\phi_r ds + \psi_r dt),\ -{\partial f \over \partial {\bm n}}d \theta = (\phi dt - \psi ds)|_{\partial B_1},\ f(0) = 0.
\end{align*}
It means that the gauge transformation $g = e^f$ will turn $A_r$ into Coulomb gauge on $B_1$. Denote $g^* A_r = d + \phi_r' ds + \psi_r' dt$, then there is a universal constant $c >0$ such that 
\begin{align*}
\big\| \phi_r' \big\|_{L^\infty(B_1)} + \big\| \psi_r' \big\|_{L^\infty(B_1)} \leq c \big\| F_{A_r} \big\|_{L^\infty(B_1)} \leq c r \sqrt{\sigma^+}.
\end{align*}
%Therefore, by zooming out, we know that $(u, h)$ is gauge equivalent to $(u', h')$ such that if $\phi', \psi'$ is obtained by $\phi' + {\bm i} \psi' = 2(\partial h' / \partial \ov{z})$, then 
%\begin{align}
%\big\| \phi' \big\|_{L^\infty(B_r)} + \big\| \psi'\big\|_{L^\infty(B_r)} \leq c \sqrt{ \sigma^+}.
%\end{align}
Denote $u_r' = g^{-1} u_r$, $h_r' = h_r + f$, $u' (z) = u_r'(z/r)$. Then in this new gauge, 
\begin{align}\label{equation714}
\partial_s u_r'  + {\mc X}_{\phi_r'} (u_r') +  J ( \partial_t u_r'  +  {\mc X}_{\psi_r'} (u_r') ) + 2 r \nabla \wt{W}_{h_r'}^{(\delta)} (u_r')= 0.
\end{align}
By ({\bf X3}) of Hypothesis \ref{hyp21}, there exists a constant $C_1>0$ (which is abusively used in this proof) such that
\begin{align}\label{equation715}
\begin{split}
&\ \big\| {\mc X}_{\phi_r'}( u_r') \big\|_{L^2(B_1)} + \big\| {\mc X}_{\psi_r'}( u_r') \big\|_{L^2(B_1)} \\[0.1cm]
\leq &\ \sqrt{\pi} \Big( \big\| {\mc X}_{\phi_r'}(u_r') \big\|_{L^\infty(B_1)} + \big\| {\mc X}_{\psi_r'}(u_r') \big\|_{L^\infty(B_1)} \Big) \\
\leq &\ C_1 r \sqrt{\sigma^+}  \Big( 1  + \sup_{B_r} \sqrt{|\mu(u)|}  \Big)\\
\leq &\ C_1 r \sqrt{\sigma^+}  \Big( 1 +  (\sigma^-)^{-{1\over 4}} \sup_{B_r} \sqrt{{\mf p}} \Big)\\
\leq &\ C_1 (\sigma^+)^{1\over 4} \sqrt{r}.
\end{split}
\end{align}
Since $\sigma^+$ is bounded from above, we can take $r_1$ sufficiently small so that  
\begin{align}\label{equation716}
 r \leq r_1 \Longrightarrow  C_1 (\sigma^+)^{1\over 4} \sqrt{r} \leq {1\over 2} {\bm \epsilon_2}.
\end{align}
Here ${\bm \epsilon_2}$ is the one in Lemma \ref{lemmaa4}. We can also assume that $E ( u, h; B_r ) \leq \epsilon_1 \leq \big( {1\over 4} {\bm \epsilon_2} \big)^2$. Then by (\ref{equation715}) and (\ref{equation716}), we have
\begin{align*}
\begin{split}
\big\| d u_r' \big\|_{L^2(B_1)} \leq &\  \big\| \partial_s  u_r'  + {\mc X}_{\phi_r'}(u_r' ) \big\|_{L^2(B_1)} + \big\| \partial_t u_r' +  {\mc X}_{\psi_r'}  (u_r' )  \big\|_{L^2(B_1)} \\
&\ + \big\|  {\mc X}_{\phi_r' }( u_r' ) \big\|_{L^2(B_1)} + \big\| {\mc X}_{\psi_r'} ( u_r' ) \big\|_{L^2(B_1)} \\
\leq &\ 2 \sqrt{ E ( u, h; B_r ) } + C_1 (\sigma^+)^{1\over 4}  \sqrt{r} \leq {\bm \epsilon}_2.
\end{split}
\end{align*}
On the other hand, by (\ref{equation715}) and (\ref{equation716}), we have
\begin{multline*}
\Big\| {\mc X}_{\phi_r'} ( u_r' ) + J {\mc X}_{\psi_r' } ( u_r' ) + 2 r \nabla \wt{W}_{h_r'}^{(\delta)} (u_r') \Big\|_{L^\infty(B_1)} \\
 \leq \big\| {\mc X}_{\phi_r'}( u_r' ) \big\|_{L^\infty(B_1)} + \big\|  {\mc X}_{\psi_r'}( u_r' ) \big\|_{L^\infty(B_1)} + 2 r \big\|  \nabla \wt{W}_{h_r'}^{(\delta)} ( u_r') \big\|_{L^\infty(B_1)} \leq 3.
\end{multline*}

Take $\lambda_1 =  {1\over 6} {\bm \epsilon_2}{\bm \epsilon_p}$. For $\lambda \in (0, \lambda_1]$, the restriction of $u_r'$ to $B_{2\lambda}$ satisfies the assumptions of Lemma \ref{lemmaa4}. Thus there exists $c_{(p)}>0$ (which is abusively used below) such that
\begin{align*}
\begin{split}
&\ \big\| d u' \big\|_{L^p(B_{\lambda r})} \\
= &\ r^{{2\over p} - 1} \big\| d u_r' \big\|_{L^p(B_\lambda)} \\
\leq &\  c_{(p)} (\lambda r)^{{2\over p }-1} \Big(  \big\| d u_r' \big\|_{L^2(B_{2\lambda_1 })} + 6 \lambda \Big) \\
\leq &\ c_{(p)} (\lambda r)^{{2\over p} - 1} \Big(  \big\| d_A u \big\|_{L^2(B_r)} + \big\| {\mc X}_{\phi_r'}(u_r') \big\|_{L^2(B_{2\lambda_1})} + \big\| {\mc X}_{\psi_r'}( u_r' ) \big\|_{L^2(B_{2\lambda_1})} + 6\lambda      \Big)\\
\leq &\ c_{(p)} (\lambda r)^{{2\over p}- 1}  \Big( \big\| d_A u \big\|_{L^2(B_r)} + \sqrt{r} + \lambda \Big).
\end{split}
\end{align*}
We used (\ref{equation710}) to derive the last inequality. Moreover, there exists $C_1 >0$ such that
\begin{align*}
{\rm diam} \big( u' (B_{\lambda r}) \big)= {\rm diam} \big( u_r' (B_{\lambda}) \big) \leq C_1 \Big( \big\| d_A u \big\|_{L^2(B_r)} + \sqrt{r} + \lambda \Big). 
\end{align*}
Lastly, we may assume that for the same $c_{(p)}>0$, we have
\begin{align*}
\begin{split}
\big\| d_A u \big\|_{L^p(B_{\lambda r})} \leq &\  \big\| d u' \big\|_{L^p(B_{\lambda r})} + \big\| {\mc X}_{\phi'}(u' ) \big\|_{L^p(B_{\lambda r})} + \big\| {\mc X}_{\psi'}(u' ) \big\|_{L^p(B_{\lambda r})} \\
\leq &\  \big\| d u' \big\|_{L^p(B_{\lambda r})} + r^{-1} \Big( \big\| {\mc X}_{\phi_r'}(u_r') \big\|_{L^\infty} + \big\| {\mc X}_{\psi_r'}(u_r' ) \big\|_{L^\infty} \Big)  \big( \pi \lambda^2 r^2 \big)^{1\over p}\\
\leq &\  c_{(p)} (\lambda r)^{{2\over p} -1}  \Big( \big\| d_A u \big\|_{L^2(B_r)} +  \sqrt{r}+ \lambda \Big).
\end{split}
\end{align*}
Thus Lemma \ref{lemma73} is proved.

\subsection{Proof of Proposition \ref{prop74}}

We assume first that $D \leq 1/2$. By Lemma \ref{lemma73}, we know that there exist $\wt\epsilon_1 >0$, $\wt{r}_1>0$ and $\wt\lambda_1 \in (0, {1\over 2}]$ such that if $E ( u, h; B_{\wt{r}} ) \leq \wt\epsilon_1$, $\wt{r} \leq \wt{r}_1$, then up to gauge transformation, ${\rm diam} \big( u(B_{ \wt\lambda_1 \wt{r}}) \big) \leq 1/2$. Then we can assume that $M_2$ is big enough such that $r \leq \lambda M_2^{-1} \leq \wt{r}_1$. Then $u\big( B_{\wt\lambda_1 r} \big) \subset \wt{X}_B^1$. It suffices to prove that there exists $\epsilon_2 > 0$ such that
\begin{align}\label{equation718}
E (u, h; B_{\wt\lambda_1 r} ) \geq \epsilon_2 \lambda^2.
\end{align}

Denote $\ov\partial_A u = {1\over 2} \big( \partial_s u + {\mc X}_\phi(u) + J \partial_t u + J {\mc X}_\psi(u) \big)$. Let $\pi_N: T\wt{X}_B^1 \to T\wt{X}_B^1$ be the orthogonal projection onto the distribution spanned over ${\mb C}$ by $\partial / \partial p$ and $\nabla Q$ and abbreviate $\pi_N(\ov\partial_A u(z)) = V_N(z)$. To prove (\ref{equation718}), we need the following estimate, which follows from straightforward but technical calculations and estimates.

\begin{lemma}\label{lemma77} For each $H>0$, there exist $c > 0$, $\wt\lambda_2 \in (0, {1\over 2}]$, $M_2 >0$, $\wt\epsilon_2 > 0$ and $D > 0$ such that if $(u, h)$ satisfies (\ref{equation79}) with this $H$, this $M_2$ and $u(0) \in \wt{X}_B^D$, and $E(u, h) \leq \wt\epsilon_2$, then $u \big( B_{\wt\lambda_2 r} \big) \subset \wt{X}_B^{2 D}$ and over $B_{\wt\lambda_2 r}$ we have
\begin{align}\label{equation719}
\Delta \big| V_N \big|^2  \geq - c \big| V_N(0) \big) \big|^2 \Big( 1+ \big| d_A u \big|^2 + \big| d h'' \big|^2 + \big| F_A \big|   \Big).
\end{align}
\end{lemma}
The proof is given in Subsection \ref{subsection76}. 

Now take any $\wt\lambda \in (0, \wt\lambda_2]$. We apply the local maximal principle \cite[Theorem 9.20]{Gilbarg_Trudinger} for $p =1$, $n=2$, $R = \wt\lambda r$, $u = \left| \ov\partial_A u \right|^2$ and $f$ equal to the right hand side of (\ref{equation719}), we see that there exists a constant $c' > 0$ (independent of $\wt\lambda$ and $r$) such that
\begin{multline*}
{1\over c'} \big| V_N(0) \big|^2 \leq  {1\over (\wt\lambda r)^2 } \int_{B_{\wt\lambda r}} \big| V_N \big|^2 dsdt \\
+ c \wt\lambda r \big| V_N(0) \big|^2 \Big( \wt\lambda r+  \big\| d_A u \big\|_{L^4(B_{\wt\lambda r})}^2 + \big\| dh''  \big\|_{L^4(B_{\wt\lambda r})}^2 + \big\| F_A \big\|_{L^2(B_{\wt\lambda r})}\Big).
\end{multline*}
Therefore, we see
\begin{multline}\label{equation720}
E ( u, h; B_{\wt\lambda r} ) \geq  {1\over 2} \int_{B_{\wt\lambda r}} \big| V_N \big|^2 ds dt \geq { (\wt\lambda r)^2 \over 2 c'} \big| V_N(0) \big|^2 \\
- {  ( \wt\lambda r )^3 c \over 2} \big| V_N(0) \big|^2  \Big(  \wt\lambda r + \big\| d_A u \big\|_{L^4(B_{\wt\lambda r})}^2 + \big\| dh'' \big\|_{L^4(B_{\wt\lambda r})}^2 + \big\| F_A \big\|_{L^2(B_{\wt\lambda r})} \Big).
\end{multline}
To proceed, we need
\begin{lemma}\label{lemma78}
There exist $\wt\epsilon_2' >0$, $\wt\lambda_2' \in (0, {1\over 2}\wt\lambda_2]$ such that if $E (u, h; B_{\wt\lambda_2 r} ) \leq \wt\epsilon_2'$, then 
\begin{align}\label{equation721}
\wt\lambda_2' r \Big( \wt\lambda_2' r + \big\| d_A u \big\|_{L^4(B_{\wt\lambda_2' r})}^2 + \big\| dh'' \big\|_{L^4(B_{\wt\lambda_2' r})}^2 + \big\| F_A \big\|_{L^2(B_{\wt\lambda_2' r})} \Big) \leq {1\over 2 c c'}.
\end{align}
\end{lemma}

\begin{proof}
We can ignore the terms $\wt\lambda_2' r$ and $\|F_A\|_{L^2(B_{\wt\lambda_2' r})}$ in (\ref{equation721}) because they are easily bounded by the radius and the energy on $B_{\wt\lambda_2' r}$. On the other hand, by Lemma \ref{lemma73} and (\ref{equation77}), if $E ( u,h; B_r )$ is sufficiently small, then for $p = 4$, we have
\begin{align*}
\wt\lambda_2' r \big\| d_A u \big\|_{L^4(B_{\wt\lambda_2' r})}^2 \leq c_{(4)}^2 \Big( \big\| d_A u \big\|_{L^2(B_r)} + \sqrt{r} + \wt\lambda_2' \Big)^2.
\end{align*}
On the other hand, define $h_r''(z) = h'' (rz)$ for $z \in B_1$. Then by the elliptic estimate over $B_1$, for some universal constant $c'' > 0$, 
\begin{multline*}
\wt\lambda_2' r \big\| dh'' \big\|_{L^4(B_{\wt\lambda_2' r})}^2 =  \wt\lambda_2' \big\| d h_r'' \big\|_{L^4(B_{\wt\lambda_2'})}^2  \leq \wt\lambda_2' c'' \big\| \Delta h_r'' \big\|_{L^2(B_1)}^2 \\
= c'' \wt\lambda_2' r^2 \big\| \Delta h'' \big\|_{L^2(B_r)}^2 = c'' \wt\lambda_2' r^2 \big\| F_A \big\|_{L^2(B_r)}^2.
\end{multline*}
Therefore it is easy to see Lemma \ref{lemma78} is true.
\end{proof}

Then by (\ref{equation720}) and Lemma \ref{lemma78}, we see that if $E( u, h; B_r ) \leq \wt\epsilon_2'$, then
\begin{align*}
E(u, h; B_r) \geq E ( u, h; B_{\wt\lambda_2' r} ) \geq { ( \wt\lambda_2' r )^2 \over 4 c_2} \big| V_N(0) \big|^2 \geq {(\wt\lambda_2' r)^2 \over 4c_2} {{\mf p}(0)^2 \over 16} = { (\wt\lambda_2')^2 \over 64 c_2} \lambda^2.
\end{align*}
Here the third inequality uses Lemma \ref{lemma712}. Therefore Proposition \ref{prop74} holds for  
\begin{align*}
\epsilon_2 = \min	\Big\{ \wt\epsilon_1, \wt\epsilon_2, \wt\epsilon_2', {( \wt\lambda_2')^2 \over 64 c_2 } \Big\}
\end{align*}
and the $M_2$ and $D$ which we already specified.

\subsection{Proof of Proposition \ref{prop75}}

Suppose $(u, h)$ satisfies (\ref{equation710}) with $M_3$ undetermined. By Lemma \ref{lemma73}, there exist $\wt\epsilon>0$, $\wt\lambda \in (0, 1/2]$ and $\wt{M}>0$ such that if $E( u, h; B_r ) \leq \wt\epsilon$ and ${\mf p} (0) \geq \wt{M}$ (so $r$ is small enough), then $( u,h)$ is gauge equivalent to a triple $( u', h')$ such that ${\rm diam} \big( u'(B_{\wt\lambda r}) \big) \leq 1/4$. Take $M_3 \geq \wt{M}$. So without loss of generality we may assume that ${\rm diam} \big( u ( B_{ \wt\lambda r} ) \big) \leq 1/4$. Moreover, by the equation $\Delta h'' + \sigma \mu^*(u) = 0$, we can take $\wt\epsilon>0$ small enough such that 
\begin{align}\label{equation722}
\sup_{B_{\wt\lambda r}} \big| \rho_0(h'') \big| - \inf_{B_{\wt\lambda r}} \big| \rho_0( h'') \big| \leq \log 2.
\end{align}

(I) Assume that ${\mf p}''(0) \geq {\mf p}'(0)$. Then 
\begin{align*}
\big| \mu(0) \big| \geq \big(\sqrt{ \sigma^+}\big)^{-1} {\mf p}''(0) \geq {M_3 \over 2 \sqrt{\sigma^+}}.
\end{align*}
By the properness of $\mu$ and ({\bf X3}) of Hypothesis \ref{hyp21}, we may take $M_3$ big enough so that for any $z \in B_{\wt\lambda r}$ ($u(z)$ is at most $1/4$ away from $u(0)$), we have
\begin{align*}
\big| \nabla \mu(u(z))) \big| \leq \big| \mu(u(z)) \big|.
\end{align*}
Therefore
\begin{align*}
\big| \mu(u(z)) - \mu(u(0)) \big| \leq  \sup_{t \in [0,1]} \big| \nabla \mu(u(tz))\big| \cdot d (u(z), u(0)) \leq {1\over 4} \sup_{B_{\wt\lambda r}} \big| \mu(u) \big|.
\end{align*}
Therefore by the triangle inequality, we have
\begin{align*}
\sup_{B_{\wt\lambda r}} \big| \mu(u) \big| \leq {4 \over 3}\big| \mu(u(0)) \big|.
\end{align*}
Then we have
\begin{align*}
\big| \mu(u(z)) \big| \geq \big| \mu(u(0)) \big| - \big| \mu(u(0)) - \mu(u(z)) \big| \geq {1\over 2} \big| \mu(u(0)) \big|.
\end{align*}
Therefore by the property of $\sigma$, we have
\begin{align*}
{\mf p}''(z)^2 = \sigma(z) \big| \mu (u(z)) \big|^2 \geq {1\over 8} \sigma(0) \big| \mu (u(0)) \big|^2 = {1\over 8} {\mf p}''(0)^2.
\end{align*}
Therefore
\begin{align*} 
E (u,h; B_r ) \geq  \int_{B_{ \wt\lambda r}} {\mf p}''(z)^2 \geq  {\pi \over 8} \big( \wt\lambda r  {\mf p}''(0) \big)^2 \geq \Big( {\pi \over 32} (\wt\lambda)^2 \Big) \lambda^2.
\end{align*}

(II) Assume that ${\mf p}'(0) \geq {\mf p}''(0)$ and $u(0) \notin  \wt{X}_B^D \cup \wt{X}_S^D$. For convenience we introduce 
\begin{align*}
{\mf p}_0'(z) = \Big| e^{\ov{\rho_0(h)}} \nabla F_0(u(z)) \Big|,\ {\mf p}_+'(z) = \Big| \beta(z) \sum_{l=1}^s e^{\ov{\rho_l(h)}} \nabla F_l^{(\delta)}(u(z)) \Big|.
\end{align*}

Let $K_Q\subset X$ be the compact subset of ({\bf Q1)} of Hypothesis \ref{hyp25} and $\wt{K} = K_Q \times B_1\subset \wt{X}$. We claim that there exists a constant $\alpha= \alpha(D) \in (0, 1)$ such that if 
\begin{align}\label{equation723}
(x, p) \notin \wt{K} \cup \wt{X}_B^{D\over 2} \cup \wt{X}_S^{D\over 2} \Longrightarrow \alpha \big| \nabla^2 W(x, p) \big| \leq \big| \nabla W(x, p) \big|.
\end{align}
To justify our claim, with respect to the product structure $\wt{X} = X \times {\mb C}$, we write  
\begin{align*}
\nabla W(x, p) = \big( \ov{p} \nabla Q(x),\ \ov{Q(x)} \big),\ \nabla^2 W(x, p) = \left( \begin{array}{cc} \ov{p} \nabla^2 Q(x) & \nabla Q(x)\\
\langle \cdot, \nabla Q(x)\rangle & 0
\end{array} \right).
\end{align*}
It suffices to bound each component of $\nabla^2 W$ by components of $\nabla W$. Indeed, if $(x, p) \notin \wt{K} \cup \wt{X}_B^{D\over 2} \cup \wt{X}_S^{D\over 2}$, then either $|p| \geq 1$ and $d(x, \star) \geq D/2$ or $|p| \leq 1$, $x \notin K_Q$ and $d(x, X_Q) \geq D/2$. In the former case, by ({\bf Q1}) of Hypothesis \ref{hyp25}, there exists $\alpha_1(D)>0$ such that $\alpha_1(D) |\nabla^2 Q(x)| \leq |\nabla Q(x)|$; so $\ov{p}\nabla^2 Q$ and $\nabla Q$ can be controlled by $\ov{p}\nabla Q$. In the latter case, there exists $\alpha_2(D)>0$ such that $\alpha_2(D) |\nabla Q(x)| \leq |Q(x)|$. Then $\nabla Q$ and $\ov{p} \nabla^2 Q$ can be controlled by $Q$ (by ({\bf Q1}) of Hypothesis \ref{hyp25}). The claim is proved. 

On the other hand, by Lemma \ref{lemma710} below, we have ${\mf p}_+'(z) \leq c_P(H)$. Therefore
\begin{align*}
M_3 \leq {\mf p}(0) \leq 2 {\mf p}'(0) \leq 2 {\mf p}_0'(0) + 2 {\mf p}_+'(0) \leq 2 {\mf p}_0'(0) + 2 c_P (H).
\end{align*}
Taking $M_3 = M_3(H) \geq 4 c_P (H)$, we have 
\begin{align}\label{equation724}
{\mf p}_0'(0) \geq M_3/4,\ {\mf p}_0'(0) \geq {\mf p}_+'(0).
\end{align}
By the expression of ${\mf p}_0'$ and the fact that $|h| \leq H$, we may take $M_3$ big enough such that 
\begin{align*}
u(B_{\wt\lambda r}) \cap \wt{K} = \emptyset. 
\end{align*}
On the other hand, by Lemma \ref{lemma73} we may take $M_3$ big enough and $\wt\lambda$ small enough such that
\begin{align*}
{\rm diam}\big( u( B_{\wt\lambda r})\big) \leq \min \big\{ \alpha/4, D/2\big\}.
\end{align*}
Then $u(B_{\wt\lambda r}) \subset \wt{X} \setminus ( \wt{K} \cup \wt{X}_B^{D \over 2} \cup \wt{X}_S^{D\over 2})$. Then by (\ref{equation723}), for any $z \in B_{\wt\lambda r}$, we have
\begin{align*}
\Big| \big| \nabla W(u(z)) \big| - \big| \nabla W(u(0))\big|\Big| \leq \sup_{t \in [0,1]} \big| \nabla^2 W(u(tz)) \big| {\alpha \over 4} \leq {1 \over 4} \sup_{B_{\wt\lambda r}} \big| \nabla W(u)\big|. 
\end{align*}
Then by triangle inequality,
\begin{align*}
\sup_{B_{\wt\lambda r}} \big|\nabla W(u) \big| \leq {4 \over 3} \big| \nabla W(u(0)) \big|.
\end{align*}
Therefore
\begin{align*}
\big| \nabla W(u(z)) \big| = \big| \nabla W(u(0)) \big| - \Big( \big| \nabla W(u(0)) \big| - \big| \nabla W(u(z)) \big| \Big) \geq {1\over 2} \big| \nabla W(u(0)) \big|.
\end{align*}
Then by (\ref{equation722}), we have ${\mf p}_0'(z) \geq {1\over 4} {\mf p}_0'(0)$; then by (\ref{equation724}) we have ${\mf p}'(0) \geq {\mf p}'(0)/8$ and 
\begin{align*}
E ( u, h; B_r ) \geq \int_{B_{ \wt\lambda r}} {\mf p}'(z)^2 \geq  { \pi \over 64} \big( \wt\lambda r {\mf p}'(0)\big)^2 \geq {\pi (\wt\lambda)^2 \over 256}\lambda^2.
\end{align*}

Therefore we see that $\epsilon_3 = \min \big\{ \wt\epsilon,  2^{-8} \pi (\wt\lambda)^2 \big\}$ and $M_3$ big enough satisfy the condition stated in this proposition.

\subsection{Proof of Proposition \ref{prop76}}

Suppose that the sequence $( u_i, h_i)$ satisfies the hypothesis of Proposition \ref{prop76}. Since $\| h_i\|_{L^\infty(B_r)}$ is uniformly bounded, ${\mf p}_i(0) \to +\infty$ implies that $\big| \mu (u_i(0)) \big|\to +\infty$. We write the map as $u_i (z) = ( \ov{u}_i(z), p_i(z))$ with respect to the decomposition $\wt{X} = X \times {\mb C}$. Then the condition $u_i(0) \in \wt{X}_S^D$ implies that 
\begin{align*}
\lim_{i \to \infty} \big| p_i(0) \big| =  +\infty.
\end{align*}

Projecting the Witten equation onto the ${\mb C}$-factor, we have
\begin{align}\label{equation725}
{\partial \over \partial \ov{z}} \left( e^{\rho_s (h_i)} p_i \right) = - e^{\rho_s(h_i)} \left( e^{\ov{\rho_0(h_i)}} \ov{Q}(\underline{u}_i) + \beta e^{\ov{\rho_s(h_i)}}\delta_i^r \ov{a}   \right).
\end{align}
Denote ${\mf s}_i(z)$ the norm of the inhomogeneous term of (\ref{equation725}) at $z$. By taking a subsequence, there are two possibilities.

(I) There exist $\rho>0$ and $M>0$ (which may depend on the subsequence) such that for every $i$,
\begin{align*}
\sup_{z\in B_\rho} {\mf s}_i(z) \leq M.
\end{align*}
Then by the standard diameter estimate for Cauchy-Riemann equations (with target ${\mb C}$), we see that $p_i$ diverges to infinity uniformly on $B_{\rho \over 2}$. Therefore (\ref{equation713}) holds and $\displaystyle \lim_{i \to \infty} \sigma_i = 0$. Otherwise
\begin{align*}
E ( u_i, h_i ; B_r ) \geq \int_{B_{\rho \over 2}} \sigma_i \big| \mu (u_i)  \big|^2 dsdt \to \infty.
\end{align*}
This contradicts with the uniform bound on energies of $(u_i, h_i)$. 

(II) There exists a sequence $y_i \in B_{r^*}$, $y_i \to 0$ such that
\begin{align}\label{equation726}
\lim_{i \to \infty} {\mf s}_i(y_i) = \infty. 
\end{align}
Denote $\tau_i = {\mf s}_i(y_i)^{-1}$. Applying Hofer's lemma (Lemma \ref{lemmaa7}) to the function ${\mf s}_i$ on $B_{{\tau_i \over 2}}(y_i)$, we see that there exist $z_i \in B_{\tau_i \over 2}(y_i)$ and $\rho_i \in (0, {\tau_i \over 4}]$ such that
\begin{align}\label{equation727}
{\mf s}_i(z_i) \geq {1\over 2} \sup_{B_{\rho_i}(z_i)} {\mf s}_i,\  {\mf s}_i(z_i)  \rho_i = {1\over 4}.
\end{align}
Now, equation (\ref{equation725}) and Lemma \ref{lemmaa4} imply that for some $\lambda_0$ small enough up to gauge transformation,
\begin{align}\label{equation728}
{\rm diam}  \left( p_i \left(　 u_i (B_{\lambda_0 \rho_i}(z_i)) \right) \right) \leq 1.
\end{align}
(Here the target is ${\mb C}$, so we don't need the condition that the energy of $u_i|_{B_{\rho_i}(z_i)}$ is small.) Then applying Hofer's lemma to the function ${\mf p}_i$ on $B_{\lambda_0 \rho_i}(z_i)$, we obtain $w_i \in B_{{1\over 2} \lambda_0 \rho_i}(z_i)$, $\kappa_i \in (0, {1\over 4} \lambda_0 \rho_i]$ such that
\begin{align}\label{equation729}
{\mf p}_i(w_i) \geq {1\over 2} \sup_{B_{\kappa_i}(w_i)} {\mf p}_i,\ \kappa_i {\mf p}_i(w_i) \geq { \lambda_0 \rho_i {\mf p}_i (z_i) \over 4 } \geq  {\lambda_0 \over 16}.
\end{align}
Notices that 
\begin{align}\label{equation730}
{\mf p}_i(w_i) \geq {\mf p}_i(z_i) \geq {\mf s}_i(z_i) \geq {\mf s}_i(y_i) \to +\infty.
\end{align}

{\it Claim.} For sufficiently large $i$, we have that either $u_i(w_i) \notin \wt{X}_S^{D}$ or ${\mf p}_i''(w_i) \geq {\mf p}_i'(w_i)$.

\begin{proof}[Proof of the claim] Suppose that there is a subsequence (still indexed by $i$) such that $u_i(w_i) \in \wt{X}_S^D$ and ${\mf p}_i'(w_i) \geq {\mf p}_i''(w_i)$. We write $\nabla \wt{\mc W}_{h_i}^{(\delta_i)}(u_i)$ as 
\begin{align*}
\nabla \wt{\mc W}_{h_i}^{(\delta_i)} (u_i) = e^{\ov{\rho_0(h_i)}}  \left(\begin{array}{c} \ov{Q}(u_i) \\ \ov{p_i} \nabla Q(u_i) \end{array} \right) + \beta \nabla {\mc W}_{h_i}^{(\delta_i)'}(u_i).
\end{align*}

Then $u_i(w_i) \in \wt{X}_S^D$ implies that $\ov{Q}(u_i(w_i))$ and $\nabla Q(u_i(w_i))$ are both bounded. Then by (\ref{equation730}) and the uniformly bound on $h_i$, we have
\begin{align*}
{\mf p}_i'(w_i) \geq {\mf p}_i''(w_i) \Longrightarrow {\mf p}_i'(w_i) \to +\infty \Longrightarrow \big| p_i(w_i) \big| \to +\infty
\end{align*}
Then (\ref{equation728}) implies that $\big| p_i (u_i(z_i)) \big| \geq \big| p_i (u_i(w_i)) \big|/2$. However, (\ref{equation726}) and (\ref{equation727}) imply that $\big| Q(u_i(z_i)) \big| \to +\infty$ and hence $\big| \nabla Q (u_i(z_i)) \big| + \infty$. Therefore 
\begin{align*}
\Big| \ov{p}_i(z_i) \nabla Q(u_i(z_i)) \Big| >> \Big| \ov{p}_i(w_i) \nabla Q(u_i(w_i)) \Big|,
\end{align*}
which contradicts with (\ref{equation729}). Therefore the claim holds.
\end{proof}

Now applying Proposition \ref{prop74} (if $u_i(w_i) \in \wt{X}_B^D$ and ${\mf p}_i''(w_i) < {\mf p}_i'(w_i)$) or Proposition \ref{prop75} (if $u_i(w_i) \notin \wt{X}_B^D$ or ${\mf p}_i''(w_i) \geq {\mf p}_i'(w_i)$) to the disk $B_{\kappa_i}(w_i)$, with the condition (\ref{equation729}), we see that for any $r>0$ and sufficiently large $i$,
\begin{align*}
E ( u_i, h_i; B_r ) \geq E ( u_i, h_i; B_{\kappa_i} (w_i) ) \geq 2^{-8} \min\{\epsilon_2, \epsilon_3 \} \lambda_0^2.
\end{align*}
(\ref{equation712}) holds for $\epsilon_4 = 2^{-8} \min \{ \epsilon_2, \epsilon_3 \} \lambda_0^2$.

\subsection{Proof of Lemma \ref{lemma77}}\label{subsection76}

Let $NB\subset T\wt{X}_B^1$ be the distribution spanned over ${\mb C}$ by $\partial / \partial p$ and $\nabla Q$ and let $TB$ be its orthogonal complement. Let $\pi_T = {\rm Id} - \pi_N$. 
\begin{lemma}\label{lemma79}
$\pi_T$ and $\pi_N$ are $G$-invariant tensor fields and 
\begin{enumerate}
\item For any $Z \in T\wt{X}_B^1$, 
\begin{align}\label{equation731}
\nabla_{JZ} \pi_T = -J \nabla_Z \pi_T,\ \nabla_{JZ} \pi_N = - J \nabla_Z \pi_N.
\end{align}

\item There exists $c_Q>0$ (which we can assume to coincide with the one of ({\bf Q1}) of Hypothesis \ref{hyp28}) such that in $\wt{X}_B^1$, 
\begin{align}\label{equation732}
\big| \nabla \pi_T \big| \leq c_Q,\ \big| \nabla \pi_N \big| \leq c_Q,\ \big| \nabla^2 \pi_T \big| \leq c_Q,\ \big| \nabla^2 \pi_N \big| \leq c_Q.
\end{align}
\end{enumerate}
\end{lemma}

\begin{proof}
The distribution $NB$ and the metric are both $G$-invariant so $\pi_N$ and $\pi_T$ are $G$-invariant. 

It is easy to see that with respect to the decomposition $T\wt{X}_B^1 \simeq TB \oplus NB$, for any tangent vector $Z$, we can write
It is easy to see that with respect to the decomposition $T\wt{X}_B^1 \simeq TB \oplus NB$, for any tangent vector $Z$, we can write
\begin{align*}
\nabla_Z \pi_T = - \nabla_Z \pi_N  = \left( \begin{array}{cc} 0 & F_Z \\
F^*_Z & 0
\end{array} \right),\ F_Z: NB \to TB.
\end{align*}
Moreover, the restriction of $F_Z$ to the ${\partial \over \partial p}$-direction is zero. Now we have
\begin{align*}
F_{JZ} \left(\nabla Q\right) = - \pi_T \nabla_{JZ} \nabla Q = \pi_T \left( J \nabla_Z \nabla Q \right) = J \pi_T \nabla_Z \nabla Q = - J F_Z (\nabla Q).
\end{align*}
Here the second equality follows from Lemma \ref{lemma34}. Since $T\wt{X}_B^1 \simeq TB \oplus TN$ is $J$-linear, we see $F^*_{JZ} = - J F^*_Z$. Therefore (\ref{equation731}) is proven.

To estimate $\nabla \pi_T$, we see that by ({\bf Q1)} of Hypothesis \ref{hyp25}, for any $Z$, we have
\begin{align}\label{equation733}
\left| F_Z (\nabla Q ) \right|  = \left| \pi_T \nabla_Z \nabla Q \right| \leq c_Q \left| \nabla Q\right||Z|.
\end{align}

Now we consider the second derivative of $\pi_T$. In $\wt{X}_B^1$, we can write the Levi-Civita connection as
\begin{align*}
\nabla = \left( \begin{array}{cc} \nabla^T & - F \\
F^* & \nabla^N
\end{array}  \right).
\end{align*}
Then
\begin{align}\label{equation734}
\nabla^2 \pi_T = \left[ \nabla, \nabla \pi_T \right] = \left( \begin{array}{cc} - 2 FF^* & \nabla^T F - F \nabla^N\\
                                                                              \nabla^N F^* - F^* \nabla^T  & 2 F^* F 
\end{array} \right).
\end{align}
Therefore it suffices to consider the two off-diagonal terms, which are adjoint to each other. Consider the upper-right one. Take tangent vectors $Z_1, Z_2$ with $\nabla_{Z_1} Z_2$ vanishes at a point. Then at that point, using ({\bf Q1}) of Hypothesis \ref{hyp25}, we see
\begin{align*}
\begin{split}
&\ \big| 	\big( \nabla^T_{Z_1} F_{Z_2} - F_{Z_2} \nabla_{Z_1}^N \big) \nabla Q  \big| \\
\leq &\ \big| \nabla_{Z_1}^T \pi_T \nabla_{Z_2}\nabla Q \big| + \big| F_{Z_2} \pi_N \nabla_{Z_1} \nabla Q\big| \\
\leq &\ \big| \nabla_{Z_1} \pi_T \nabla_{Z_2} \nabla Q \big| + \big| F_{Z_2} \pi_N \nabla_{Z_1} \nabla Q \big| \\
\leq &\ \big| F_{Z_1} \nabla_{Z_2} \nabla Q \big| + \big| \nabla_{Z_1} \nabla_{Z_2} \nabla Q \big| + \big| F_{Z_2} \pi_N \nabla_{Z_1} \nabla Q \big| \\
\leq &\ c_Q |Z_1||Z_2| \big|\nabla Q \big|.
\end{split}
\end{align*}
By (\ref{equation733}), (\ref{equation734}) and above we see (\ref{equation732}) holds.
\end{proof}

\begin{lemma}\label{lemma710}
For any $H>0$, there exists $c_P = c_P(H)>0$ such that if $|h| \leq H$ and $\delta \leq 1$, then for any $x \in \wt{X}$,
\begin{align}\label{equation735}
\sum_{l=1}^s \big| e^{\ov{\rho_l(h)}} \nabla^{(j)} F_l^{(\delta)} (x) \big| \leq c_P\ (j=1, 2, 3).
\end{align}
Moreover, if $(u, h)$ is a solution to a local model over $B_r$ such that $\big\| h\big\|_{L^\infty(B_r)} \leq H$ and $u(B_r) \subset \wt{X}_B^1$, then for any $z \in B_r$, 
\begin{align}\label{equation736}
\big| \pi_T( \ov\partial_A u (z)) \big| \leq c_P.
\end{align}
\end{lemma}
\begin{proof}
The first estimate follows from the hypothesis $|h(z)|\leq H$ and ({\bf P3}) of Hypothesis \ref{hyp28}. By the equation $\ov\partial_A u + \wt{W}(u) = 0$ we have
\begin{align*}
\pi_T (\ov\partial_A u) = - \pi_T( \nabla \wt{W} (u)) = - \pi_T ( \nabla W' (u)) = - \beta \pi_T \big( \sum_{l=1}^s e^{\ov{\rho_l(h)}} \nabla F_l^{(\delta)}(u) \big).
\end{align*}
Then (\ref{equation736}) follows from (\ref{equation735}) with $j=1$.
\end{proof}

Now we consider the Hessian of $\wt{W}_{h}$ in $\wt{X}_B^1$. With respect to the decomposition $NB \oplus TB$, we write
\begin{align*}
\nabla^2 W_h = \left( \begin{array}{cc} E_1 & E_2 \\ E_3 & E_4 \end{array} \right),\ \nabla^2 \wt{W}_h = \wt{E} =  \left( \begin{array}{cc} \wt{E}_1 & \wt{E}_2 \\ \wt{E}_3 & \wt{E}_4 \end{array} \right).
\end{align*}
On the other hand, with respect to the splitting $T \wt{X} = {\mb C} \oplus TX$, we have
\begin{align}\label{equation737}
\nabla^2 W_h = e^{\ov{\rho_0(h)}} \left( \begin{array}{cc} 0 & \left(\nabla Q \right)^*\\
\nabla Q & \ov{p} \nabla^2 Q
\end{array} \right).
\end{align}

\begin{lemma}\label{lemma711}For any $H>0$, there exist $c_H>0$, $D_1 = D_1(H)>0$ and a compact subset $\wt{K}_1:= \wt{K}_1(H) \subset \wt{X}$ such that if $|h|\leq H$ and $x \in \wt{X}_B^{D_1} \setminus \wt{K}_1$, then 
\begin{align*}
\big| \wt{E}_i(z, x) \big| \leq c_H \big( 1 + \big| \nabla \wt{W}_h (z, x) \big| \big),\ \big| \wt{E}_i(z, x) \big| \leq {1\over 6} \big| \wt{E}_1(z, x) \big|.
\end{align*}
In particular,
\begin{align*}
\big| \wt{E}_1(z, x) \big| \geq {1\over 2} \big| \wt{E}(z, x) \big|.
\end{align*}
\end{lemma}

\begin{proof}
By the definition of $\wt{E}_i$ and (\ref{equation737}), we see that for $i = 2, 3, 4$, 
\begin{align*}
\begin{split}
\big| \wt{E}_i \big| \leq &\ \big| e^{\ov{\rho_0(h)}} \ov{p} \nabla^2 Q \big| + \big| \nabla^2 W_h' \big| \\
\leq &\ c_Q \big| e^{\ov{\rho_0(h)}} \ov{p} \nabla Q \big| + c_P(H)\\
\leq &\ c_Q \big| \nabla \wt{W}_h \big| + c_Q c_P(H) + c_P(H).
\end{split}
\end{align*}
Therefore the first inequality holds by choosing $c_H$ properly. On the other hand, if $x \in \wt{X}_B^D$ and $D \leq  1/2 c_Q$, then by (\ref{equation737}),
\begin{align*}
\begin{split}
\big| \wt{E}_1 \big| \geq &\ \big| E_1 \big| - \big| \nabla^2 W_h' \big|\\
\geq &\ \big| e^{\ov{\rho_0(h)}} \nabla Q \big| - \big| e^{\ov{\rho_0(h)}} \ov{p} \nabla^2 Q \big| - c_P(H)\\
\geq &\ \big| e^{\ov{\rho_0(h)}} \nabla Q \big| - c_Q D \big| e^{\ov{\rho_0(h)}} \nabla Q \big| - c_P(H)\\
\geq &\ {1\over 2} \big| e^{\ov{\rho_0(h)}} \nabla Q \big| - c_P(H).
\end{split}
\end{align*}
Moreover, since $|h|\leq H$, we can take $\wt{K}(H)$ sufficiently big such that if $x \notin \wt{K}(H)$, then
\begin{align*}
\big| e^{\ov{\rho_0(h)}} \nabla Q(x) \big| \geq e^{-|\rho_0(H)|} \big| \nabla Q(x)\big| \geq 50 c_P(H).
\end{align*}
We take $D \leq 1/ 24c_Q$. Then for $i = 2, 3, 4$, we have
\begin{align*}
\begin{split}
\big| \wt{E}_i\big| \leq &\ c_Q \big| e^{\ov{\rho_0(h)}} \ov{p} \nabla Q \big| + c_P(H)\\
\leq &\ c_Q D \big| e^{\ov{\rho_0(h)}} \nabla Q \big| + c_P(H)\\
\leq &\ 2 c_Q D \big| \wt{E}_1 \big| + 2 c_Q c_P(H) D + c_P(H)\\
\leq &\ 2 c_Q D \big| \wt{E}_1 \big| + 2c_P(H)\\
\leq &\ {1\over 6} \big| \wt{E}_1 \big|.
\end{split}
\end{align*}
\end{proof}

\begin{lemma}\label{lemma712}
For any $H>0$, there exists $M_2(H)>0$ such that if $(u, h)$ is a solution to a local model over $B_r$ satisfying (\ref{equation79}) and $\big\| h\big\|_{L^\infty(B_r)} \leq H$, $u(B_\rho) \subset \wt{X}_B^1$. Then for any $z \in B_\rho$, we have
\begin{align}\label{equation738}
\big| \ov\partial_A u(z) \big| \leq  4 \big| \ov\partial_A u(0) \big| \leq 8 \big| \pi_N ( \ov\partial_A u (0) ) \big|
\end{align}
\end{lemma}
\begin{proof}
Indeed, by (\ref{equation79}), 
\begin{align*}
\big| \ov\partial_A u(z)\big| = {\mf p}'(z) \leq {\mf p}(z) \leq 2 {\mf p}(0) \leq 4 {\mf p}'(0) = 4 \big| \ov\partial_A u(0) \big|.
\end{align*}
On the other hand, by (\ref{equation79}) and (\ref{equation735}), we have
\begin{align*}
M_2 (H) /2 \leq \big| \ov\partial_A u(0) \big| \leq \big| \pi_N(\ov\partial_A u(0))\big| + c_P.
\end{align*}
So the second inequality of (\ref{equation738}) holds if $M_2 (H)$ is big enough. 
\end{proof}

\begin{proof}[Proof of Lemma \ref{lemma77}] Let $D = D_1 (H)/2$ where $D_1(H)$ comes from Lemma \ref{lemma711}. Then by Lemma \ref{lemma73}, for certain $M_2 > 0$ and $\wt{\lambda}_2>0$, if $(u, h)$ satisfies (\ref{equation79}) for this $M_2$ and $u(0) \in \wt{X}_B^D$, then ${\rm diam} \big( u(B_{\wt\lambda_2 r}) \big) \leq D$. Then we can take $M_2 = M_2(H)>0$ big enough such that $u(B_{\wt\lambda_2 r}) \subset \wt{X}_B^{2D} \setminus \wt{K}_1 (H)$ and (\ref{equation738}) is satisfied. Then we can apply the estimates obtained in Lemma \ref{lemma710}, \ref{lemma711}, \ref{lemma712}. 

Abbreviate $\ov\partial_A u = V$ and $\pi_N(\ov\partial_A u) = V_N$. Then we have
\begin{align}\label{equation739}
\begin{split}
{1\over 2} \Delta \big( \big| V_N \big|^2 \big) = &\  {1\over 2} \partial_s^2 \big\langle V_N, V_N \big\rangle + {1\over 2} \partial_t^2 \big\langle V_N , V_N \big\rangle \\
= &\ \partial_s \big\langle D_{A, s} V_N, V_N \big\rangle + \partial_t u \big\langle D_{A, t} V_N, V_N \big\rangle\\
= &\ \big| D_{A, s} V_N \big|^2 + \big| D_{A, t} V_N \big|^2 +  \big\langle \big( D_{A, s}^2 + D_{A, t}^2 \big) V_N, V_N\big\rangle.
\end{split}
\end{align}
Then we have
\begin{align}
\begin{split}
&\ \big( D_{A, s}^2 + D_{A,t}^2 \big) V_N \\
= &\ D_{A, s} \big( D_{A, s} - J D_{A, t} \big) V_N  + J D_{A, t} \big( D_{A, s} - J D_{A, t} \big) V_N + J \big[ D_{A, s}, D_{A, t} \big] V_N\\
= &\ 4 D_A^{0,1} D_A^{1, 0} V_N +  J \big[ D_{A, s}, D_{A, t} \big]  V_N \\
= &\ 4 D_A^{0,1} D_A^{1,0} V_N + J R(v_s, v_t)  V_N + J \nabla_{V_N} {\mc X}_{F_A}.
\end{split}
\end{align}
By ({\bf X2}) and ({\bf X4}) of Hypothesis \ref{hyp21}, there exist $c_R, c_\mu >0$ such that
\begin{align}
\Big| \big\langle J R(v_s, v_t) V_N, V_N \big\rangle \Big| \leq c_R \big| V_N\big|^2 \big|d_A u \big|^2.
\end{align}
\begin{align}
\Big| \big\langle J \nabla_{V_N} {\mc X}_{F_A}, V_N \big\rangle \Big| \leq \big| V_N \big|^2 \big| F_A \big| \big| \nabla^2 \mu \big| \leq c_\mu \big| F_A \big| \big| V_N \big|^2.
\end{align}

Abbreviate $\partial = \partial / \partial z$, $\ov\partial = \partial / \partial \ov{z}$, $\partial \ov\partial = \partial^2 / \partial z \partial \ov{z}$. Then We have
\begin{align}\label{equation743}
\begin{split}
D_A^{0,1} D_A^{1,0} V_N = &\ - D_A^{0,1} D_A^{1, 0} \pi_N \big( \nabla \wt{W}_h \big) \\
= &\ - D_A^{0,1} \Big( \big( \nabla_V \pi_N \big) \nabla \wt{W}_h +  \pi_N \big( \nabla_V \nabla \wt{W}_h + \partial \beta \nabla W_h' (u) \big)  \Big) \\
= &\   \big[ D_A^{0,1}, \nabla_{ \nabla \wt{W}_h} \pi_N  \big] \nabla \wt{W}_h  + \big( \nabla_{ \nabla \wt{W}_h} \pi_N \big) D_A^{0,1} \nabla \wt{W}_h \\
 &\ - \big( \nabla_{\partial_A u} \pi_N \big) \big( \nabla_V \nabla \wt{W}_h + \partial \beta \nabla W_h' (u) \big) \\
&\  + \pi_N \Big( D_A^{0,1} \nabla_{\nabla \wt{W}_h} \nabla \wt{W}_h -  \partial \ov\partial \beta  \nabla W_h'(u) - \partial \beta D_A^{0,1} \nabla W_h'(u) \Big). 
\end{split}
\end{align}
We estimate the above expression term by term. 

(I) By (1) of Lemma \ref{lemma79}, for any tangent vector field $Z_1$ and $Z_2$, we have
\begin{align}\label{equation744}
\begin{split}
\big[ D_{A, s}, \nabla_{Z_1} \pi_N \big] (Z_2) = &\ \big( \nabla^2_{v_s, Z_1} \pi_N \big)  (Z_2) + \big( \nabla_{D_{A, s} Z_1} \pi_N \big) (Z_2),\\
\big[ D_{A, t}, \nabla_{Z_1} \pi_N \big] (Z_2) = &\ \big( \nabla^2_{v_t, Z_1} \pi_N \big)  (Z_2) + \big( \nabla_{D_{A, t} Z_1} \pi_N \big) (Z_2).\\
\end{split}
\end{align}
Therefore, 
\begin{align*}
\begin{split}
&\ \Big| \big[ D_A^{0,1}, \nabla_{ \nabla \wt{W}_h} \pi_N  \big] \nabla \wt{W}_h \Big|\\
\leq &\ \Big( \big|  \nabla_{D_A^{1,0}  \nabla \wt{W}_h } \pi_N \big| + \big|\nabla^2_{v_s, \nabla \wt{W}_h} \pi_N \big| +  \big|\nabla^2_{v_t, \nabla \wt{W}_h} \pi_N \big| \Big) \big| 	V \big| \\
\leq &\ c_Q \Big( \big| D_A^{1,0} \nabla \wt{W}_h \big| + \big| d_A u \big| \big| \nabla \wt{W}_h \big|\Big) \big| V \big| \\
\leq &\  c_Q \Big( c_P  + \big| \nabla_V \nabla \wt{W}_h \big| + \big| \nabla \wt{W}_h \big|  \big|d_A u \big| \Big)  \big| V \big| \\
\leq &\  c_Q \Big( \big| \wt{E}_1 \big| \big| V_N \big| + \big| \wt{E}_2 \big| \big| \pi_T(V) \big| + c_P + \big| \nabla \wt{W}_h \big| \big|d_A u \big| \Big) \big| V \big|\\
\leq &\ c_Q \Big( \big| \wt{E} \big|  \big| V_N \big| + c_P \big| \wt{E} \big| + c_P + \big| d_A u \big|^2 \Big) \big| V\big|.
\end{split}
\end{align*}
We briefly explain how we obtain this estimate. To derive the first inequality, we used (\ref{equation744}) and (2) of Lemma \ref{lemma79}; to derive the second inequality we used (2) of Lemma \ref{lemma79}; to derive the third inequality we used the expression of $D_A^{1,0} \nabla \wt{W}_h$ in (\ref{equation317}). Then
\begin{align}
\begin{split}
&\ \big\langle \big[ D_A^{0,1}, \nabla_{ \nabla \wt{W}_h} \pi_N  \big] \nabla \wt{W}_h, V_N \big\rangle\\
\geq &\ - c_Q \big| \wt{E} \big| \big| V \big| \big| V_N \big|^2 - c_Q c_P \big| \wt{E} \big| \big| V \big| \big| V_N \big|  - c_Q \big( c_P + \big| d_A u \big|^2 \big) \big| V \big| \big| V_N \big|\\
\geq &\ - {1\over 64} \big| \wt{E} \big|^2 \big| V_N \big|^2 - (4 c_Q)^2 \big| \nabla \wt{W}_h \big|^2 \big| V_N \big|^2\\
&\ - {1\over 64} \big| \wt{E} \big|^2 \big| V_N \big|^2 - (4 c_Q c_P)^2 \big| V \big|^2 - c_Q \big( c_P + \big| d_A u\big|^2 \big) \big| V \big|^2\\
\geq &\ -{1\over 32} \big| \wt{E} \big|^2 \big| V_N \big|^2 - c_{P, Q}^{(1)} \big( 1 +  \big| d_A u \big|^2 \big) \big| V_N(0) \big|^2.
\end{split}
\end{align}
Here $c_{P, Q}^{(1)}>0$ depends on $c_P$ and $c_Q$ and the last inequality uses Lemma \ref{lemma712}.

(II) For the second summand of (\ref{equation743}), we have
\begin{align}
\begin{split}
&\ \big| \big( \nabla_{\nabla \wt{W}_h} \pi_N \big)\big( D_A^{0,1} \nabla \wt{W}_h \big) \big|\\
\leq &\ c_Q \big| V \big| \big| D_A^{0,1} \nabla \wt{W}_h \big|\\
\leq &\  c_Q \big| V \big| \Big( c_P + \big| \partial_A u \big| \big| \nabla^2 \wt{W}_h \big|  + \big| dh''\big|  \big( c_P + \big| \nabla W_h \big|  \big)  \Big)\\
\leq &\   c_Q \big| V \big| \big| \wt{E} \big| \big| \partial_A u \big| + c_Q \big| V \big| \big( c_P + 2 c_P \big| dh'' \big| + \big| dh'' \big| \big| V \big| \big)\\
\leq &\ c_Q \big| \wt{E} \big| \big| V \big|  \big| d_A u \big| + c_Q \big| V \big| \big( c_P + (c_P)^2 +  2 \big| dh'' \big|^2  +  \big| d_A u \big|^2 \big).
\end{split}
\end{align}
Here the first inequality uses (2) of Lemma \ref{lemma79}; the second one uses the expression of $D_A^{0,1} \nabla \wt{W}_h$ in (\ref{equation318}) and the bound on perturbation terms given by Lemma \ref{lemma710}. Then
\begin{align}
\begin{split}
&\ \big\langle \big( \nabla_{\nabla \wt{W}_h} \pi_N \big)\big( D_A^{0,1} \nabla \wt{W}_h \big), V_N \big\rangle\\
\geq & - c_Q \big| \wt{E} \big|  \big| V_N\big| \big| V\big|\big| \partial_A u\big| - c_Q \big( c_P + (c_P)^2 + 2 \big| dh''\big|^2 + \big| d_A u\big|^2 \big) \big| V \big|^2\\
\geq & - {1\over 64} \big| \wt{E} \big|^2 \big| V_N \big|^2  - (4 c_Q)^2 \big| d_A u \big|^2 \big| V\big|^2 - c_Q \big( c_P + (c_P)^2 + 2 \big| dh''\big|^2 + \big| d_A u\big|^2 \big) \big| V \big|^2\\
\geq & - {1\over 64} \big| \wt{E} \big|^2 \big| V_N \big|^2  - c_{P, Q}^{(2)} \big( 1 + \big| dh'' \big|^2 + \big| d_A u \big|^2 \big) \big| V_N(0)\big|^2.
\end{split}
\end{align}
Here $c_{P,Q}^{(2)}>0$ depends on $c_P$ and $c_Q$ and the last inequality uses Lemma \ref{lemma712}.

(III) For the third summand of (\ref{equation743}), we have 
\begin{align}
\begin{split}
&\ - \big\langle \big( \nabla_{\partial_A u } \pi_N \big) \big( \nabla_V \nabla \wt{W}_h  + (\partial \beta) \nabla W_h' \big), V_N \big\rangle\\
\geq &\ -c_Q \big| \partial_A u \big| \big| \nabla_V \nabla \wt{W}_h \big| \big| V_N\big|  - c_P c_Q \big| \partial_A u \big| \big| V_N\big|\\
\geq &\ - c_Q \big| \wt{E}\big| \big| V_N \big| \big| d_A u \big| \big| V \big| - c_P c_Q \big| d_A u \big| \big| V \big| \\
\geq &\ - {1 \over 64} \big| \wt{E} \big|^2 \big| V_N \big|^2 - ( 4 c_Q)^2 \big| d_A u \big|^2 \big| V \big|^2 - c_P c_Q \big| d_A u \big| \big| V \big|\\
\geq &\ - {1 \over 64} \big| \wt{E} \big|^2 \big| V_N \big|^2 - c_{P, Q}^{(3)} \big( 1 + \big| d_A u \big|^2 \big) \big| V_N(0) \big|^2.
\end{split}
\end{align}
Here the first inequality we used (2) of Lemma \ref{lemma79} and the bound on $\nabla W_h'$ given by Lemma \ref{lemma710}; $c_{P, Q}^{(3)}>0$ depends on $c_P$ and $c_Q$ and the last inequality uses Lemma \ref{lemma712}.

(IV) For the fourth summand of (\ref{equation743}), we have
\begin{align}
\begin{split}
 &\ - \big\langle \pi_N \big( (\partial \ov\partial \beta) \nabla W_h' + (\partial \beta) D_A^{0,1} \nabla W_h' \big), V_N \big\rangle\\
\geq &\  -c_P \big| V_N\big|  - \big| \partial_A u \big| \big| \nabla^2 W_h' \big|\big| V_N \big| - c_P \big| dh'' \big| \big| V_N \big|\\
\geq &\ - c_{P, Q}^{(4)} \big( 1 + \big| dh'' \big|^2 \big) \big| V_N(0) \big|^2.
\end{split}
\end{align}
Here we used the bounds on $\nabla W_h'$ and $\nabla^2 W_h'$ given by Lemma \ref{lemma710}, and an expression of $D_A^{0,1} \nabla W_h'$ similar to (\ref{equation318}). The last inequality uses Lemma \ref{lemma712}.

(V) Lastly, the dominating part of (\ref{equation743}) is estimated as follows.  
\begin{lemma}\label{lemma713}
There exist a constant $c_H^{(5)} > 0$, a compact subset $\wt{K}_2 =\wt{K}_2(H)\subset \wt{X}$ which depends on $H>0$ such that if a solution $(u, h)$ to a local model over $B_\rho$ (for some $\rho>0$) satisfies (\ref{equation738}) and
\begin{align*}
\big\| h \big\|_{L^\infty(B_\rho)} \leq H,\ u(B_\rho) \subset \wt{X}_B^{D_1} \setminus \wt{K}_2,
\end{align*}
then 
\begin{align}\label{equation750}
\big\langle  \pi_N \big( D_A^{0,1} \nabla_{\nabla \wt{W}_h} \nabla \wt{W}_h \big), V_N  \big\rangle \geq {1\over 16} \big| \wt{E} \big|^2 \big| V_N \big|^2 - c_H^{(5)} \big( 1 + \big| d_A u \big|^2 + \big| dh'' \big|^2 \big)\big| V_N(0)\big|^2.
\end{align}
\end{lemma}
This lemma is proved at the very end of this section. 

Without loss of generality, we may take $M_2(H)$ big enough so that $u(B_{\wt\lambda_2 r}) \subset \wt{X}_B^{2D} \setminus \wt{K}_2$. Then by (\ref{equation739})--(\ref{equation750}), we see that there is a constant $c(H)>0$
\begin{align*}
\begin{split}
{1\over 2} \Delta \big| V_N \big|^2 \geq &\ - \big| V_N \big|^2 \big( c_R \big| d_A u \big|^2 + c_\mu \big| F_A \big| \big) \\
  - & {1\over 32} \big| \wt{E} \big|^2 \big| V_N\big|^2 - c_{P, Q}^{(1)} \big| V_N(0) \big|^2 \big( 1+ \big| d_A u \big|^2 \big)\\
  - & {1 \over 64} \big| \wt{E} \big|^2 \big| V_N \big|^2 - c_{P, Q}^{(2)} \big| V_N(0) \big|^2  \big( 1 + \big| dh'' \big|^2 + \big| d_A u \big|^2 \big)\\
  - & {1\over 64} \big| \wt{E} \big|^2 \big| V_N \big|^2 - c_{P, Q}^{(3)} \big| V_N(0) \big|^2 \big( 1+  \big| d_A u \big| \big)^2 \\
  - & c_{P, Q}^{(4)} \big| V_N(0) \big|^2 \big( 1 + \big| dh'' \big|^2 \big)\\
  + & {1\over 16} \big| \wt{E}\big|^2 \big| V_N \big|^2 - c_{P, Q}^{(5)} \big| V_N(0) \big|^2 \big( 1 + \big| dh'' \big|^2 + \big| d_A u \big|^2 \big)\\
\geq &\ - c(H) \big| V_N(0) \big|^2 \big( 1 + \big| dh'' \big|^2 + \big| d_A u \big|^2 + \big| F_A \big| \big).
\end{split}
\end{align*}
\end{proof}

\subsubsection*{Proof of Lemma \ref{lemma713}}

We need the following estimate on the tensor field $\wt{H}_A$ defined locally by (\ref{equation319})--(\ref{equation320}).
\begin{lemma}\label{lemma714}
For each $H>0$, there exist $c_H>0$ and a compact subset $\wt{K}_2(H)\subset \wt{X}$ satisfying the following conditions. 

Let $(\beta, \sigma, \delta)$ be a parameter of a local model on $B_\rho$ and $(u, h) \in C^\infty( B_\rho, \wt{X}\times {\mf g} \times {\mf g})$. Suppose 
\begin{align*}
\big\| h \big\|_{L^\infty(B_r)} \leq H,\ u(B_\rho) \subset \wt{X} \setminus \wt{K}_2(H).
\end{align*}
Then for any smooth vector field $Z$ along $u$, we have
\begin{align*}
\big| \wt{H}_A (u, d_A u, Z) \big| \leq c_H \big( \big| d_A u \big| + \big| d h'' \big| \big) \big| \nabla^2 \wt{W}_h \big| \big| Z \big|.
\end{align*}
\end{lemma}

\begin{proof}
By (\ref{equation319}) and Lemma \ref{lemma710}, it is easy to see that  
\begin{align}\label{equation751}
\big| \beta \sum_{l =1}^s H_A^{(l)}(u, d_A u, Z) \big| \leq c_P  \big( \big| d_A u \big| + \big| dh'' \big| \big) \big| Z \big|.
\end{align}
Therefore we only have to consider $H_A^{(0)}$. By the expression of $H_A^{(0)}$, we see
\begin{align}
\begin{split}
\big| e^{\ov{\rho_0(h)}} \rho_0( {\bm i }\partial_s h'' - \partial_t h'' ) \nabla_Z \nabla F_0 \big| \leq &\ r e^{rH} \big|d h''\big| \big| \nabla^2 W_h \big| \big| Z \big|;\\
\big| e^{\ov{\rho_0(h)}} \rho_0(\partial_s h'' + {\bm i} \partial_t h'' ) \nabla_Z \nabla F_0 \big| \leq &\ r e^{rH} \big|d h'' \big| \big| \nabla^2 W_h \big| \big| Z \big|.
\end{split}
\end{align}
On the other hand, the term $e^{\ov{\rho_0(h)}} G_{F_0}(v_s, Z)$ (resp. $e^{\ov{\rho_0 (h)}} G_{F_0} (v_t, Z)$) is bounded by $|v_s|$ (resp. $|v_t|$) times the third order derivative of $W_h$ in the vertical direction. By ({\bf Q1}) in Hypothesis \ref{hyp25}, we have (up to certain universal or dimensional constants)
\begin{align*}
\big| \nabla^3 F_0 \big| \leq \big| p \nabla^3 Q \big| + \big| \nabla^2 Q \big| \leq c_Q \big| p \nabla^2 Q \big| + c_Q \big| \nabla Q \big| \leq c_Q \big| \nabla^2 F_0 \big|.
\end{align*}
Then we have
\begin{align}\label{equation753}
\big| e^{\ov{\rho_0 (h)}} G_{F_0} (v_s, Z) \big| \leq c_Q \big| \nabla^2 W_h \big| \big| v_s \big| \big| Z \big|,\ \big| e^{\ov{\rho_0 (h)}} G_{F_0} (v_t, Z) \big| \leq c_Q \big| \nabla^2 W_h \big| \big| v_t \big| \big| Z\big|.
\end{align}
Since $\big| \nabla^2 W_h \big| \geq \big| e^{\ov{\rho_0(h)}} \nabla Q \big|$ and $\big| \nabla^2 W_h' \big| \leq c_P(H)$, there exists $\wt{K}_2(H)$ such that outside $\wt{K}_2(H)$, $1 \leq \big| \nabla^2 W_h \big| \leq  2 \big| \nabla^2 \wt{W}_h \big|$. Then the lemma follows from (\ref{equation751})--(\ref{equation753}).
\end{proof}

Now we can prove Lemma \ref{lemma713}. By the definition of $\wt{H}_A$ and (\ref{equation321}) we have
\begin{align}\label{equation754}
\begin{split}
&\ D_A^{0,1} \nabla_{\nabla \wt{W}_h} \nabla \wt{W}_h \\
= &\ \nabla_{D_A^{1,0} \nabla \wt{W}_h} \nabla \wt{\mc  W}_A  +  \ov\partial \beta \nabla_{\nabla \wt{W}_h} \nabla W_h' + \wt{H}_A^{0,1} \big( u, d_A u, \nabla \wt{W}_h \big)\\
= &\ \nabla_{\nabla_{\ov\partial_A u} \nabla \wt{W}_h} \nabla \wt{W}_h + \ov\partial \beta \nabla_{\nabla \wt{W}_h} \nabla W_h'+ \nabla_{ \partial \beta \nabla {\mc W}'_A } \nabla \wt{W}_h + \wt{H}^{0,1}_A \big( u, d_A u, \nabla \wt{W}_h \big).
\end{split}
\end{align}
The summands of (\ref{equation754}) can be estimated as follows.

(I) We have
\begin{align}
\begin{split}
&\ \big\langle \nabla_{\nabla_{\ov\partial_A u} \nabla \wt{W}_h} \nabla \wt{W}_h, V_N \big\rangle \\
= &\ \big\langle \nabla_{V_N} \nabla \wt{W}_h, \nabla_{\ov\partial_A u} \nabla \wt{W}_h \big\rangle  \\
  = &\  \big| \wt{E}_1 (V_N) \big|^2 + \big| \wt{E}_2 (V_N)\big|^2  + \big\langle \wt{E}_1 (V_N), \wt{E}_3 (\pi_T (V)) \big\rangle +  \big\langle \wt{E}_2 (V_N), \wt{E}_4 (\pi_T (V)) \big\rangle\\
	\geq &\ {3\over 4} \big| \wt{E}_1 (V_N) \big|^2 - \big| \wt{E}_3 (\pi_T (V)) \big|^2 - {1\over 4} \big| \wt{E}_4 ( \pi_T (V)) \big|^2\\
	\geq &\ {3\over 16} \big| \wt{E} \big| \big|V_N \big|^2 - \big(  c_P c_H \big( 1 + \big| V \big| \big)\big)^2.
	\end{split}
\end{align}
To derive the last inequality we used Lemma \ref{lemma711}.

(II) The terms in (\ref{equation754}) containing the cut-off function $\beta$ can be controlled as follows. 
\begin{align}
\begin{split}
\big\langle \ov\partial \beta \nabla_{\nabla \wt{W}_h} \nabla W_h', V_N \big\rangle \geq &\ - c_P \big| \nabla \wt{W}_h \big| \big| V_N \big| \geq - c_P \big| V_N \big| \big| d_A u \big|;\\
\big\langle \nabla_{\partial \beta \nabla W_h'} \nabla \wt{W}_h, V_N \big\rangle = &\ \big\langle \nabla_{V_N} \nabla \wt{W}_h, (\partial \beta) \nabla W_h' \big\rangle\\
\geq &\ - c_P \big| \wt{E}_1 (V_N) \big|  - c_P \big| \wt{E}_2 (V_N) \big| \\
\geq &\ - 2 c_P \big| \wt{E} \big| \big| V_N \big|\\
\geq &\ - {1\over 16} \big| \wt{E} \big|^2 \big| V_N \big|^2 - (4c_P)^2.
\end{split}
\end{align}
Here the first estimate follows from the uniform bound on $\nabla^2 W_h'$ and the second estimate follows from Lemma \ref{lemma710} and \ref{lemma711}.

(III) By Lemma \ref{lemma714}, we have 
\begin{align}\label{equation757}
\begin{split}
&\  \big\langle  \wt{H}_A^{0,1} \big( u, d_A u, \nabla \wt{W}_h \big) , V_N \big\rangle \\
\geq &\ - c_{P, Q}^{(6)} \big( \big|d_A u \big| + \big| d h'' \big| \big) \big| \wt{E} \big| \big| \nabla \wt{W}_h \big| \big| V_N \big|\\
\geq &\ -{1 \over 16} \big| \wt{E} \big|^2 \big| V_N \big|^2 - (4 c_{P, Q}^{(6)} )^2  \big| V \big|^2 \big( \big| d_A u \big|^2 + \big|dh''\big|^2 \big). 
\end{split}
\end{align}

Then by (\ref{equation754})--(\ref{equation757}) and redefining $c_{P, Q}^{(6)} > 0$ properly, we have
\begin{align*}
\big\langle  D_A^{0,1} \big( \nabla_{\nabla \wt{W}_h} \nabla \wt{W}_h \big), V_N \big\rangle \geq {1\over 16} \big| \wt{E} \big|^2 \big| V_N \big|^2 - c_{P, Q}^{(6)} \big| V_N(0) \big|^2 \big(  1 + \big| d_A u \big|^2 + \big| dh'' \big|^2 \big).
\end{align*}
So Lemma \ref{lemma713} is proved.

\section{Proof of the compactness theorem}\label{section8}

\subsection{The uniform $C^0$-bound}

In this subsection, we show that the ``bubbling at infinity'' won't happen and solutions to the gauged Witten equation are uniformly bounded everywhere. The argument is based on a maximal principle near the point where the bubbling may happen {\it a priori}. Similar estimates appear \cite[Page 859]{Cieliebak_Gaio_Salamon_2000}, \cite[556]{Cieliebak_Gaio_Mundet_Salamon_2002} and \cite[Page 780]{FJR1}.

Let ${\mc F}= {\mc F}_{b_0}: \wt{X} \to {\mb R}$ be the $G$-invariant function in Lemma \ref{lemma24}.
\begin{lemma}\label{lemma81}
 Let ${\mc H}$ be the Hessian of ${\mc F}$. Then we have
\begin{align}\label{equation81}
{\mc H}(J\cdot, J\cdot) = {\mc H}(\cdot, \cdot).
\end{align}
Moreover, as a quadratic form on $TX$, we have
\begin{align}\label{equation82}
0 \leq {\mc H} \leq 1.
\end{align}
\end{lemma}
\begin{proof}
Since $J$ is integrable, for any tangent vector $V$, we have
\begin{multline*}
{\mc H} (JV, JV) = JV (JV {\mc F}) - \left(\nabla_{JV} JV\right) {\mc F} = JV \omega( {\mc X}_\upxi, JV) - \omega({\mc X}_\xi, J \nabla_{JV} V ) \\
=JV \langle {\mc X}_\upxi, V \rangle - \langle {\mc X}_\upxi, \nabla_{JV} V \rangle = \langle \nabla_{JV} {\mc X}_\upxi, V \rangle.
\end{multline*}
Replacing $JV$ by $V$, we have
\begin{align*}
{\mc H} (V, V) = \left\langle \nabla_{ J(-JV)} {\mc X}_\upxi, -JV\right\rangle = \left\langle \nabla_V {\mc X}_\upxi, -JV \right\rangle = \left\langle \nabla_{JV} {\mc X}_\upxi, V \right\rangle.
\end{align*}
The last equality is true because ${\mc X}_\upxi$ is Killing. Therefore (\ref{equation81}) holds. On the other hand, (\ref{equation82}) follows from ({\bf X4}) of Hypothesis \ref{hyp21} and the definition of ${\mc F}$.
\end{proof}

Since ${\mc F}$ is $G$-invariant, it lifts to a function ${\mc F}: Y \to {\mb R}$. We have
\begin{prop}\label{prop82}
For each $E>0$, there exist $c(E)>0$ such that for any solution $(A, u)$ to the perturbed gauged Witten equation with $E(A, u)\leq E$, we have
\begin{align*}
\Delta_c {\mc F}(u) \geq  {\sigma \over 2c_0} \big| \mu(u)\big|^2 - c(E).
\end{align*}
Here $\sigma: \Sigma^* \to {\mb R}_+$ is the ratio of the smooth metric over the cylindrical metric of $\Sigma^*$ and $c_0>0$ is the constant in Lemma \ref{lemma24}.
\end{prop}

\begin{proof}
Near any $q \in \Sigma^*$, we use the local model of the perturbed gauged Witten equation so $(A, u)$ gives a solution $(u, h): B_{r^*} \to \wt{X} \times {\mf g} \times {\mf g}$ to a local model parametrized by $(\beta, \sigma,\delta)$. Moreover, there exists $H = H(E)$ such that $\big\| h \big\|_{L^\infty(B_{r^*})} \leq H(E)$. 

Over $B_{r^*}(q)$, the cylindrical area form can be expressed as $\tau ds dt$ where $\tau: B_{r^*} \to {\mb R}$, which is uniformly bounded from above and uniformly bounded away from zero. Then by the equation, we have
\begin{align}\label{equation83}
\begin{split}
\tau \Delta_c {\mc F}(u)  = &\ \partial_s \big\langle \nabla {\mc F}, v_s  \big\rangle + \partial_t \big\langle \nabla{\mc F} ,  v_t \big\rangle \\
= &\ \partial_s \big\langle \nabla {\mc F}, - J v_t -2 \nabla \wt{\mc W}_A(u) \big\rangle + \partial_t \big\langle \nabla {\mc F}, J v_s+ 2 J \nabla \wt{\mc W}_A (u) \big\rangle\\
= &\ - 2 \partial_s \big\langle \nabla {\mc F},  \nabla \wt{\mc W}_A(u) \big\rangle +  2 \partial_t \big\langle \nabla {\mc F}, J \nabla \wt{\mc W}_A(u) \big\rangle\\
 &\ + \partial_s \big\langle \nabla {\mc F}, - J v_t \big\rangle + \partial_t \big\langle \nabla F , J v_s \big\rangle.
\end{split}
\end{align}
Now by Lemma \ref{lemma81} and (\ref{equation311}) we have
\begin{align}\label{equation84}
\begin{split}
 &\ \partial_s \big\langle \nabla {\mc F}, - J v_t \big\rangle + \partial_t \big\langle \nabla{\mc F}, J v_s  \big\rangle \\
= &\ {\mc H} \big(  v_s , - J v_t \big) + {\mc H} \big( v_t , J v_s \big) + \big\langle \nabla {\mc F}, - J D_{A, s} v_t + J D_{A, t} v_s \big\rangle\\
= &\ 2 {\mc H} \big( \partial_A u, \partial_A u \big) - 2 {\mc H} \big( J \ov\partial_A u, J  \ov\partial_A u \big) + \big\langle \nabla {\mc F}, - J {\mc X}_{F_A} \big\rangle\\ 
= &\ 2 {\mc H} \big( \partial_A u , \partial_A u \big) - 2 {\mc H} \big(   \ov\partial_A u, \ov\partial_A u \big) +	\sigma \big\langle \nabla {\mc F}, J {\mc X}_{\mu^*} \big\rangle.
\end{split}
\end{align}
In the second identity we used (\ref{equation311}) and in the third equality we used the vortex equation $F_A + \sigma \mu^*(u) ds dt = 0$.

On the other hand, denote ${\mc W}_{A, b}' = d {\mc W}_A' \cdot \nabla {\mc F} $ where $d$ denotes the differential in the vertical direction. Then
\begin{align}\label{equation85}
\begin{split}
&\ \partial_s \big\langle \nabla {\mc F}, -2\nabla \wt{\mc W}_A \big\rangle + \partial_t \big\langle \nabla {\mc F}, + 2 J \nabla \wt{\mc W}_A \big\rangle \\
= &\ -4 {\rm Re} \Big[ {\partial \over \partial \ov{z}} \big\langle\big\langle \nabla {\mc F},  \nabla \wt{\mc W}_A(u) \big\rangle\big\rangle \Big] \\
= &\ -4 {\rm Re} \Big[ {\partial \over \partial \ov{z}} {\mc W}_A (u) + {\partial \over \partial \ov{z}} (\beta {\mc W}_{A, b}'(u)) \Big]  \\
= &\ -4 {\rm Re} \Big[ d {\mc W}_{A} \cdot \ov\partial_A u + \beta d {\mc  W}_{A,b}' \cdot \ov\partial_A u + {\partial \beta \over \partial \ov{z}} {\mc W}_{A,b}'(u) \Big] \\
= &\ - 4 {\rm Re} \Big[ d \wt{\mc W}_A \cdot \ov\partial_A u + \beta d \big( {\mc W}_{A, b}' - {\mc W}_A' \big) \cdot \ov\partial_A u + {\partial \beta\over \partial \ov{z}} {\mc W}_{A, b}'(u)  \Big] \\
= &\ 4 \big| \ov\partial_A u \big|^2 - 4 {\rm Re} \Big[ \beta d \big( {\mc W}_{A, b}' - {\mc W}_A' \big) \cdot \ov\partial_A u + {\partial \beta \over \partial \ov{z}} {\mc W}_{A,b}'(u) \Big].
\end{split}
\end{align}
Then (\ref{equation83})--(\ref{equation85}) imply that 
\begin{align}\label{equation86}
\begin{split}
\tau \Delta_c {\mc F}(u) = &\ 2{\mc H} \big( \partial_A u, \partial_A u \big) - 2{\mc H}\big( \ov\partial_A u , \ov\partial_A u \big) + \sigma \big\langle \nabla {\mc F}, J{\mc X}_{\mu^*} \big\rangle + 4 \big| \ov\partial_A u \big|^2\\
&\  - 4 {\rm Re} \Big[ \beta d \big( {\mc W}_{A, b}' - {\mc W}_A' \big) \cdot \ov\partial_A u + {\partial \beta \over \partial \ov{z}} {\mc W}_{A, b}'(u) \Big]\\
 \geq &\  2 {\mc H} \big( \partial_A u, \partial_A u \big) - 2{\mc H} \big( \ov\partial_A u , \ov\partial_A u \big) + \sigma \big\langle \nabla {\mc F}, J {\mc X}_{\mu^*} \big\rangle + 2 \left| \ov\partial_A u \right|^2  \\
&\ - 2 |\beta|^2 \big| d {\mc W}_A'(u) - d {\mc W}_{A, b}'(u) \big|^2 - 4 {\rm Re} \Big[ { \partial \beta \over \partial \ov{z}} {\mc W}_{A, b}'(u) \Big]\\
\geq &\ - 2 |\beta|^2 \big| d {\mc W}_A'(u) - d {\mc W}_{A, b}'(u) \big|^2+ \sigma \big\langle \nabla {\mc F}, J {\mc X}_{\mu^*} \big\rangle - 4 {\rm Re} \Big[ {\partial \beta \over \partial \ov{z}} W_{A, b}'(u) \Big].
\end{split}
\end{align}
Here the second inequality follows from (\ref{equation82}). Moreover, by Lemma \ref{lemma710} and ({\bf P3}) of Hypothesis \ref{hyp28} there exists $c(E)>0$ depending on $E$ such that 
\begin{align*}
\big| d {\mc W}_A'(u) - d {\mc W}_{A, b}'(u) \big|^2 \leq c(E),\ \big| {\mc W}_{A, b}'(u) \big| \leq c(E) \sqrt{ 1+ \big|\mu(u) \big|}.
\end{align*}
Then with the symbol $c(E)$ abusively used, we have
\begin{align}
\begin{split}
\tau \Delta_c {\mc F}(u) \geq &\ \sigma \big\langle \nabla {\mc F}, J{\mc X}_{\mu^*} \big\rangle - c(E) |\beta|^2 - c(E) \big| d\beta \big| \sqrt{ 1+ \big| \mu(u) \big| }\\
\geq &\ {\sigma \over c_0} \big| \mu(u) \big|^2 - c(E)|\beta|^2 - c(E) \big|d\beta\big| \big| \mu(u) \big|\\
\geq &\ {\sigma \over 2c_0} \big| \mu(u) \big|^2 - c(E)|\beta|^2 - { c(E) |d\beta|^2 \over 2 \sigma}\\
\geq &\ {\sigma \over 2c_0} \big| \mu(u) \big|^2 - c(E).
\end{split}
\end{align}
Here the last inequality follows from the fact that $d\beta$ is controlled by $\sigma$. 
\end{proof}

Now we can prove the uniform bound on the section.
\begin{thm}\label{thm83}
For every $E>0$, there exists $K(E) > 0$ such that for every solution $(A, u)$ to the perturbed gauged Witten equation with $E(A, u) \leq E$, we have
\begin{align*}
\big\| {\mc F}(u) \big\|_{L^\infty(\Sigma^*)} \leq K(E).
\end{align*}
\end{thm}

\begin{proof}
For any bounded solution $(A, u)$ to the perturbed gauged Witten equation over $\vec{\mc C}$, the function ${\mc F}(u)$ extends continuously to $\Sigma$, thanks to Theorem \ref{thm42}. Moreover, the value of ${\mc F}(u)$ at every broad punctures is uniformly bounded because the limit at each broad puncture $z_j$ lies in a uniformly bounded subset of $\wt{X}_{\upgamma_j}$.

Suppose the statement is not true, then there exists a sequence $(A^{(i)}, u^{(i)})$ of solutions to the perturbed gauged Witten equation over $\vec{\mc C}$ with $E(A^{(i)}, u^{(i)})\leq E$ such that 
\begin{align}\label{equation88}
\lim_{i \to \infty} \big\| {\mc F}(u^{(i)}) \big\|_{L^\infty(\Sigma)} = +\infty.
\end{align}
By Corollary \ref{cor72}, there is a subsequene of the sequence $(A^{(i)}, u^{(i)})$ (still indexed by $i$), and sequences of points $\{z^{(i)}_\beta\}_{1\leq \beta\leq m}$ contained in $\Sigma^*$ which satisfy (3) of Corollary \ref{cor72}. In particular, for each $\beta$ and each small $r>0$, the restriction of ${\mc F}(u^{(i)})$ to $\partial B_r(z_i)$ is uniformly bounded. Then, apply the mean value estimate to ${\mc F}(u^{(i)} )$ restricted to $B_r(z_i)$, with the first differential inequality in Proposition \ref{prop82}, we see that
\begin{align*}
{\mc F}\big( u^{(i)}(z_\beta^{(i)}) \big) \leq  \max_{ \partial B_r(z_\beta^{(i)} )} {\mc F}\big( u^{(i)} \big) + { C(E) \over 8 \pi}r.
\end{align*}
This contradicts with the divergence of ${\mc F}(u^{(i)} (z_\beta^{(i)}))$. Therefore, $m=0$; by Corollary \ref{cor72}, it means on any compact subset of the complement of narrow punctures, ${\mc F}(u^{(i)} )$ is uniformly bounded. 

On the other hand, for any bounded solution $(A, u)$, any $r>0$ sufficiently small and any narrow puncture $z_j$, we define
\begin{align*}
K_j(A, u) = \sup_{\partial \wt{B}_r(z_j)} {\mc F}(u).
\end{align*}
Here $\wt{B}_r(z_j)$ is the radius $r$ disk around $z_j$ with respect to the smooth metric. We also take
\begin{align*}
K' = \sup \Big\{ {\mc F}(x)\ |\ x\in \wt{X},\ \big| \mu(x) \big| \leq c_0 \Big\}
\end{align*}
where $c_0$ is the one in (\ref{equation24}). Since $\mu$ is proper, $K'(E)$ is finite. We claim that for each narrow puncture $z_j$,
\begin{align}\label{equation89}
\sup_{B_r(z_j)} {\mc F}(u) \leq \max \big\{ K_j(A, u), K' \big\}. 
\end{align}
Then it leads to a contradiction with (\ref{equation88}).

Indeed, take $w \in B_r(z_j)$ with ${\mc F}(u(w)) = \sup_{B_r(z_j)} {\mc F}(u) > \max \{ K_j(A, u), K'\}$. If $w\neq z_j$, then near $w$ we have
\begin{align}
\Delta_c {\mc F}(u) \geq \sigma \big\langle \nabla {\mc F}(u), J{\mc X}_{\mu^*(u)} \big\rangle \geq {\sigma \over c_0} \big( \big|\mu(u) \big|^2 - c_0^2 \big) > 0.
\end{align}
So ${\mc F}(u)$ is subharmonic near $w$ and hence ${\mc F}(u)$ is constant on $B_r(z_j)$. It contradicts with the definition of $K_j(A, u)$. On the other hand, if $w = z_j$, choose $\tau\in (0, r)$ such that $\inf_{B_\tau(z_j)} {\mc F}(u) \geq K'$. Let $\wt\Delta$ be the Laplacian with respect to the smooth metric on $\wt{B}_r(z_j)$. Then  for any $\kappa \in (0, \tau)$, by the divergence formula, we have
\begin{align}\label{equation810}
\int_{\wt{B}_r(z_j) \setminus \wt{B}_\kappa(z_j)} \wt\Delta {\mc F}(u) = \int_{\partial \wt{B}_r(z_j)} {\partial \over \partial \rho} {\mc F}(u) - \int_{\partial \wt{B}_\kappa (z_j)} {\partial \over \partial \rho} {\mc F}(u).
\end{align}
Here $\rho$ is the radial coordinate on $\wt{B}_r(z_j)$. By the exponential convergence of $u$ near each puncture (Theorem \ref{thm43}), there is an $\alpha>0$ such that $| \partial_\rho {\mc F}(u) | \leq \rho^{\alpha -1}$. Therefore let $\kappa$ go to zero in (\ref{equation810}), we see 
\begin{align*}
\int_{\partial \wt{B}_\tau (z_j)} {\partial \over \partial \rho} {\mc F}(u) = \int_{\wt{B}_\tau (z_j)} \wt\Delta {\mc F}(u) \geq \int_{\wt{B}_r(z_j)} \big\langle \nabla {\mc F}, J{\mc X}_{\mu^*} \big\rangle \geq {1\over c_0} \int_{\wt{B}_\tau(z_j)} \big( \big| \mu(u)\big|^2 -c_0^2 \big) \geq 0.
\end{align*}
Then since ${\mc F}(u)$ attains maximum at $z_j$, ${\mc F}(u)$ is a constant on $B_r(z_j)$. This constant is bigger than $K_j(A, u)$, which is a contradiction. Therefore (\ref{equation89}) holds.
\end{proof}

\subsection{Proof of Theorem \ref{thm65}}

\begin{prop}\label{prop84}
For every $E>0$, with abuse of notation, there exists $K(E) > 0$ such that for every solution $(A, u)$ to the perturbed gauged Witten equation (\ref{equation221}) with $E(A, u) \leq E$, we have
\begin{align*}
\big\| d_A u \big\|_{L^\infty(\Sigma^*)} \leq K(E).
\end{align*}
\end{prop}
\begin{proof}
The uniform bound on the section $u$ implies that the inhomogeneous term of the Witten equation is uniformly bounded. Then it is a standard argument to extract a subsequence from any sequence of solutions with energy uniform bound, such that the subsequence bubbles off a non-constant holomorphic sphere. However, since the target space $\wt{X}$ is symplectically aspherical, this is impossible. 
\end{proof}

Now we can prove Theorem \ref{thm65}. Suppose $(A^{(i)}, u^{(i)})$ is a sequence of smooth solutions to the perturbed gauged Witten equation over $\vec{\mc C}$ with 
\begin{align*}
\sup_i E\big( A^{(i)}, u^{(i)} \big) = E  < \infty.
\end{align*}
As we did in the proof of Proposition \ref{prop216}, there is a sequence of (smooth) gauge transformations $g^{(i)}\in {\mz G}$ such that $(g^{(i)})^* A^{(i)}$ is in Coulomb gauge, relative to some reference connection $A_0$. We replace $A^{(i)}$ by $(g^{(i)})^* A^{(i)}$. On the other hand, by Theorem \ref{thm83}, there is a $G$-invariant compact subset $\wt{K} \subset \wt{X}$ such that the images of $u^{(i)}$ are contained in $ P\times_G \wt{K}$. So by the equation $* F_{A_i} + \mu^* (u_i) = 0$, the curvature form has uniformly bounded $L^\infty$-norm. Therefore elliptic estimate shows that $A^{(i)}$ converges to some $A \in {\mz A}$ in weak $W^{1, p}$-topology. In particular, the monodromy and residue of $A$ at each $z_j$ is the same as that of each $A^{(i)}$. The weak convergence implies that $A$ is also in Coulomb gauge relative to $A_0$.

Therefore, by the continuous dependence of $\wt{\mc W}_A$ on $A \in {\mz A}$, for any compact subset $\Sigma_c \subset \Sigma^*$ and any $G$-invariant compact subset $\wt{K}$, $ \lim_{i \to \infty} \wt{\mc W}_{A^{(i)}} = \wt{\mc W}_A$ uniformly on $(P|_{\Sigma_c}) \times_G \wt{K}$. By the basic compactness about inhomogeneous Cauchy-Riemann equation with Proposition \ref{prop84}, there is a subsequence (still indexed by $i$) and a section $u \in \Gamma_{loc}^{1, p}(Y)$ such that $u_i$ converges to $u$ in $W_{loc}^{1, p}$-topology. Moreover, the pair $(A, u)$ satisfies the perturbed gauged Witten equation on $\Sigma^*$. By Proposition \ref{prop216} and its proof, the Coulomb gauge condition on $A$ implies that $(A, u)$ is smooth. Bootstrapping shows that the convergence of $(A^{(i)}, u^{(i)})$ to $(A, u)$ is uniform on any compact subset of $\Sigma^*$, together with all derivatives. 

Now near each puncture $z_j$, we consider the corresponding cylindrical models on a cylindrical end. By Theorem \ref{thm42}, for each $z_j$, there exists $\upkappa_j \in \wt{X}_{\upgamma_j}$, such that 
\begin{align*}
\lim_{z \to z_j} e^{\lambda_j t} u(z) = \upkappa_j.
\end{align*}

%If $z_j$ is narrow, we prove that the condition (4) of Definition \ref{defn64} is satisfied. Fix $S_0$ sufficiently large and identify $U_j(S_0)$ with $\Theta_+$. We first show that
%\begin{align}\label{equation811}
%\lim_{s \to +\infty} \limsup_{i \to \infty} \sup_{[s, +\infty)\times S^1} e \big( A^{(i)}, u^{(i)} \big) = 0,
%\end{align}
%where the energy density $e^{(i)}:= e \big( A^{(i)}, u^{(i)} \big)$ is taken with respect to the cylindrical metric. Suppose (\ref{equation811}) is not true, then %there exist $\epsilon_0>0$, a subsequence (still indexed by $i$) and a sequence $\big(s^{(i)}, t^{(i)} \big)\in \Theta_+$ with $s^{(i)} \to +\infty$ such that
%\begin{align}\label{equation812}
%e^{(i)}  \big( s^{(i)}, t^{(i)} \big) \geq \epsilon_0.
%\end{align}
%Assume that $t^{(i)}$ converges to $t_0 \in S^1$. Denote by $v^{(i)}: [ - s^{(i)}/2,  s^{(i)}/2 ] \times S^1 \to \wt{X}$ the map
%\begin{align*}
%v^{(i)}(s, t) = u^{(i)} \big( s + s^{(i)} , t \big).
%\end{align*}
%Then it is easy to see that there is a subsequence of $v^{(i)}$, still indexed by $i$, which converges to a soliton uniformly on compact subset of $\Theta$ (with all derivatives). (\ref{equation812}) implies that the limit soliton has nonzero energy density at $(0, t_0)$, which contradicts with Lemma \ref{lemma62}. Therefore (\ref{equation811}) holds and Theorem \ref{thm43} implies that the condition (4) of Definition \ref{defn64} holds.

It is possible that a sequence of solutions degenerate to a stable solution with a sequence of solitons ``attached'' at the broad punctures. The situation is almost the same as the situation in Floer theory where a sequence of connecting orbits degenerate to a stable connecting orbits in the limit. 

For each broad puncture $z_j$, identify $U_j \simeq \Theta_+$, we express $(A^{(i)}, u^{(i)})$ as a solution ${\bm u}^{(i)} = (u^{(i)}, h^{(i)})$ to a cylindrical model where the connection form is $\phi^{(i)} ds + \psi^{(i)} dt$. We also write $A = d + \phi ds + \psi dt$. Since $A^{(i)}$ converges to $A$ uniformly on any compact subset of $\Sigma^*$, $(\phi^{(i)}, \psi^{(i)})$ converges to $(\phi, \psi)$ uniformly on any compact subset. We define 
\begin{align}
f^{(i)}(s, t) = \int_s^{+\infty} \phi^{(i)}(v, t) dv,\ f(s, t) = \int_s^{+\infty} \phi(v, t) dv. 
\end{align}
By the symplectic vortex equation we see that $f^{(i)}, f \in W_\delta^{2, p}(\Theta_+)$ for some $\delta>0$. Denote $\cancel{g}^{(i)} = \exp f^{(i)}$, $\cancel{g} = \exp f$. Therefore we can extend $\cancel{g}^{(i)}$ and $\cancel{g}$ to gauge transformations in ${\mz G}$ so that $\cancel{g}^{(i)}$ converges to $\cancel{g}$ uniformly on any compact subset of $\Sigma^*$. Therefore, we can absorb $\cancel{g}^{(i)}$ into $g^{(i)}$, and without loss of generality, we may assume that ${\bm u}^{(i)}$ is in temporal gauge. 

Now if there is a subsequence (still indexed by $i$) and $S_0>0$ such that 
\begin{align*}
\sup_{[S_0, +\infty)\times S^1} e \big( {\bm u}^{(i)} \big) \leq \epsilon_2,
\end{align*}
where $\epsilon_2 = \epsilon_2 ( \wt{K}, \ud\delta)$ is the one in Theorem \ref{thm43}. Then by Theorem \ref{thm43} we see that 
\begin{align*}
\lim_{S \to +\infty} \lim_{i \to \infty} E \big( {\bm u}^{(i)}; \Theta_+(S) \big) = 0.
\end{align*}
Therefore in this situation, Theorem \ref{thm65} holds. 

On the other hand, suppose there is a sequence $(s^{(i)}, t^{(i)})\in \Theta_+$ such that 
\begin{align}\label{equation813}
\lim_{i \to \infty} s^{(i)} \to +\infty,\ e \big( {\bm u}^{(i)}\big) \big( s^{(i)}, t^{(i)} \big) > \epsilon_2.
\end{align}
Then we can extract a subsequence (still indexed by $i$) such that the sequence 
\begin{align*}
\big( u^{(i)} \big( s + s^{(i)}, t \big), h^{(i)} \big( s + s^{(i)}, t \big) \big)
\end{align*}
converges uniformly on compact subsets of $\Theta$ with all derivatives to $(u, 0)$, where $u: \Theta \to \wt{X}_{\upgamma_j}$ is a soliton with nonzero energy. Indeed, $(u^{(i)}, h^{(i)})$ satisfies the equation
\begin{align*}
\partial_s u^{(i)} + J\big( \partial_t u^{(i)} + {\mc X}_{\psi^{(i)}}(u^{(i)}) \big) + \nabla \wt{W}_{h^{(i)}, \lambda_j}^{(\delta_j^{(i)})}(u^{(i)}) = 0,\ \Delta (h^{(i)})'' + \sigma \mu^* (u^{(i)}) = 0.
\end{align*}
Here $\delta_j^{(i)} = \delta_{j, A^{(i)}} \in (0,1]$. Since all derivatives of $u^{(i)}$ are uniformly bounded, the second equation implies that a subsequence of $h^{(i)}(s + s^{(i)}, t)$ (still indexed by $i$) converges uniformly on any compact subset of $\Theta$ to $0$, together all derivatives. Then $\psi^{(i)}(s + s^{(i)}, t)$ converges uniformly on any compact subset of $\Theta$ to the constant $\lambda$ together will all derivatives. Moreover, $\delta_j^{(i)}$ converges to $\delta_j := \delta_{j, A}$. Therefore the standard argument shows that $u^{(i)}$ converges uniformly with all derivatives to a $(\lambda_j, \delta_j)$-soliton. (\ref{equation813}) implies that this soliton must be nontrivial. 

More generally, using the same trick as in \cite[Section 8.5]{Mundet_Tian_2009}, we define an {\bf $\epsilon_2 $-bubbling list} to be a sequence of lists $\big\{ z_1^{(i)}, \ldots, z_\alpha^{(i)} \big\}$, satisfying
\begin{itemize}
\item $z_l^{(i)} = \big( s^{(i)}_l, t^{(i)}_l \big) \in U_j\simeq \Theta_=$ and $\lim_{i \to \infty} s^{(i)}_l \to +\infty$;

\item for $l_1 \neq l_2$, $ d \big( z_{l_1}^{(i)}, z_{l_2}^{(i)} \big) \to +\infty$;

\item for each $l$ we have $\displaystyle \liminf_{i \to \infty} e \big( {\bm u}^{(i)} \big)\big( z_l^{(i)} \big) \geq \epsilon_2$.
\end{itemize}
We call $\alpha\geq 1$ the length of an $\epsilon_2$-bubbling list. Then if we have an $\epsilon_2$-bubbling list, we can extract a subsequence (still indexed by $i$) for which locally near $z_l^{(i)}$ the sequence converges to a nontrivial soliton. It is easy to see, because there are only finitely many critical points of $\wt{W}_{\upgamma_j}|_{\wt{X}_{\upgamma_j}}$, the length of any $\epsilon$-bubbling list is uniformly bounded from above. So we take an $\epsilon_2$-bubbling list $\big\{ z_1^{(i)}, \ldots, z_{\alpha_j}^{(i)} \big\}$ of maximal length $\alpha_j$. By taking a subsequence, we may assume that 
\begin{align*}
l_1 < l_2 \Longrightarrow s_{l_1}^{(i)} < s_{l_2}^{(i)}.
\end{align*}
Then we define 
\begin{align*}
\big( u^{(i)}_l (s, t) , h^{(i)}_l (s, t) \big) = \big( u^{(i)} \big( s + s^{(i)}_l, t \big), h^{(i)} \big( s + s^{(i)}_l, t \big) \big),\ l=1, \ldots, \alpha.
\end{align*}
By choosing a further subsequence, we may assume that each $u^{(i)}_l$ converges to a nontrivial soliton $u_l: \Theta \to \wt{X}_{\upgamma_j}$. Denote ${\bm u}_j =(u_1, \ldots, u_{\alpha_j})$. Then for any sequence of points $z^{(i)} = \big( s^{(i)}, t^{(i)} \big)$, we have
\begin{align*}
\lim_{i \to \infty} \min_l d \big( z^{(i)}, z_l^{(i)} \big) = \infty \Longrightarrow \limsup_{i \to \infty} e \big( u^{(i)}, h^{(i)} \big) \big( z^{(i)} \big) < \epsilon_2.
\end{align*}
Otherwise it will contradicts with the fact that the $\epsilon_2$-bubbling list is of maximal length. Then it is easy to see that by Theorem \ref{thm43}, $(u_l)_+ = (u_{l+1})_-$ for $l = 1, \ldots, \alpha_j-1$ and $ev_j(A, u) = (u_1)_-$. 

Therefore, the collection $\big( A, u, \{ {\bm u}_j\}_{z_j\ {\rm broad}}\big)$ forms a stable solution to the perturbed gauged Witten equation over $\vec{\mc C}$. Moreover, the subsequence constructed above and $\big(A, u , \{ {\bm u}_j\}_{z_j\ {\rm broad}} \big)$ satisfy Definition \ref{defn64}. Hence the proof of Theorem \ref{thm65} is complete.

\appendix

\section{Epsilon-regularity, etc.}\label{appendixa}

\subsection{Epsilon-regularity for Cauchy-Riemann equations}

The Witten equation is an inhomogeneous Cauchy-Riemann equation. In this appendix we recall some basic estimates about Cauchy-Riemann equations. 

We first recall the $\epsilon$-regularity result of \cite{IS_compactness} in the case of $J$-holomorphic curves with a continuous $J$. Let $Y$ be a manifold of dimension $2N$ and $Y' \subset Y$ be a subset. Let $h_0$ be a smooth Riemannian metric on $Y$ which we used as a reference to define the norms on function spaces on $Y$. For any $x\in Y$ and $\delta>0$, we use $B_\delta(x)$ to denote the open geodesic ball centered at $x$ with radius $\delta$.

\begin{defn}\label{defna1}(\cite[Definition 1.1]{IS_compactness})
A continuous almost complex structure $J$ on $Y$ is said to be {\bf uniformly continuous} on $Y'$ (with respect to $h_0$), if the following is true. 1) $\left\| J \right\|_{L^\infty(Y')} < \infty$; 2) For any $\epsilon>0$, there is a number $\delta>0$ such that for any $x \in Y'$, there exists a $C^1$-diffeomorphism $\phi: B(x, \delta) \to B(0, \delta) \subset {\mb C}^N$ such that
\begin{align}\label{equationa1}
\big\| J - \phi^* J_{st} \big\|_{C^0( B_\delta(x)) \cap Y')} + \big\| h_0 - \phi^* h_{st} \big\|_{C^0(B_\delta(x) \cap Y'} < \epsilon,
\end{align}
where $J_{st}$ is the standard complex structure and $h_{st}$ is the standard metric on ${\mb C}^N$.
\end{defn}
For each $\epsilon>0$, the largest $\delta$ for which (\ref{equationa1}) is true is called the modulus of uniform continuity, and is denoted by a function $\mu_J(\epsilon)$.

\begin{lemma}\label{lemmaa2}\cite[Lemma 1.1]{IS_compactness} Let $J^*$ be a continuous almost complex structure on $Y$ which is uniformly continuous on $A\subset Y$. For every $p \in (2, +\infty)$, there exist constants ${\bm \epsilon_1} = {\bm \epsilon_1} (\mu_{J*}, A, h_0 )>0$, ${\bm \epsilon_p}>0$, ${\bm C_p} = C(p, \mu_{J^*}, A, h_0 )<\infty$ with the following property.

For any continuous almost complex structure $J$ on $Y$ with $\left\| J- J^* \right\|_{L^\infty(A)}< {\bm \epsilon_p}$ and for any $J$-holomorphic map $u \in C^0 \cap L_1^2(B_1, Y)$ such that $u(B_1) \subset A$ and $\left\| du \right\|_{L^2(B_1)} < {\bm \epsilon_1}$, we have
\begin{align}\label{equationa2}
\big\| du \big\|_{L^p(B_{1\over 2})} \leq {\bm C_p}  \big\| du \big\|_{L^2(B_1)}.
\end{align}
\end{lemma}
By Sobolev embedding $L_1^p \to C^{0, {2\over p} - 1}$, (\ref{equationa2}) implies (using the same constant ${\bm C_p}$)
\begin{align}\label{equationa3}
{\rm diam} \big( u(B_{1\over 2}) \big) \leq {\bm C_p} \big\| du \big\|_{L^2(B_1)}.
\end{align}

A consequence of Lemma \ref{lemmaa2} is the following.
\begin{lemma}\label{lemmaa4}
Let $(X, h_0)$ be a Riemannian manifold of dimension $2n$ and $J$ be a continuous almost complex structure on $X$ which is uniformly continuous on the whole (noncompact) manifold $X$ with respect to $h_0$. Then there exists ${\bm \epsilon_2} = {\bm \epsilon_2}( \mu_J, X, h_0 )>0$ satisfying the following condition.

Suppose $\rho \in (0, 1]$, $\nu\in C^0 ( B_\rho \times X, TX)$ and $u: B_\rho \to X$ satisfies the inhomogeneous equation
\begin{align}\label{equationa4}
{\partial u \over \partial \ov{z}} + \nu(u) = 0.
\end{align}
If
\begin{align*}
\big\| du \big\|_{L^2(B_\rho)} \leq {\bm \epsilon_2},\ \rho \big\| \nu(u) \big\|_{L^\infty(B_\rho)} \leq {\bm \epsilon_p} {\bm \epsilon_2},
\end{align*}
then 
\begin{align*}
{\rm diam} \big( u ( B_{\rho \over 2}) \big)  \leq {\bm C_p} \Big( \big\| du \big\|_{L^2(B_\rho)} + \rho \big\| \nu(u) \big\|_{L^\infty(B_\rho)} \Big).
\end{align*}
\end{lemma}

\begin{proof}
Denote $Y = {\mb C} \times X$. Let $\wt{J}_0 = (J_0, J_{st})$ be the product almost complex structure on $Y$. Then $\wt{J}_0$ is uniformly continuous on $Y$ with respect to the product metric $\wt{h}_0 = (h_0, h_{st})$. We take
\begin{align*}
{\bm \epsilon_2} = {1\over 1 + \sqrt{\pi}} {\bm \epsilon_1} \Big( \mu_{\wt{J}_0}, {\mb C}\times X, \wt{h}_0 \Big).
\end{align*}

Indeed, denote $\kappa = \left\| \nu(u) \right\|_{L^\infty(B_\rho)}$, we define
\begin{align*}
\begin{array}{cccc}
\wt{v}: & B_1 & \to & B_{\rho \kappa \over {\bm \epsilon}_p} \times X\\
       &  w & \mapsto &  \displaystyle  \big( { \rho \kappa \over {\bm \epsilon_p}} w, u(\rho w) \big);
\end{array} \hspace{0.3cm} 
\begin{array}{cccc}
 \wt{\nu}: & \displaystyle B_{{\rho \kappa \over {\bm \epsilon}_p}}  \times X & \to & TX \\
   & (w, x) & \mapsto & \displaystyle   {{\bm \epsilon}_p \over \kappa } \nu \big( {{\bm \epsilon}_p \over \kappa} w, x \big).
\end{array}
\end{align*}
Then define an almost complex structure $\wt{J}_{\wt{\nu}}$ on $B_{\rho\kappa \over {\bm \epsilon}_p} \times X$ by
\begin{align*}
\wt{J}_{\wt{\nu}} (\partial_s, X) = (\partial_t, J_0 X + \wt\nu ),\ \wt{J}_{\wt\nu}(\partial_t, X) = ( - \partial_s, J_0 X - J_0 \wt\nu).
\end{align*}
Then (\ref{equationa4}) implies that $\wt{u}$ is holomorphic with respect to $\wt{J}_{\wt\nu}$. On the other hand, we have
\begin{align*}
\big\| \wt{J}_{\wt\nu} - \wt{J}_0 \big\|_{L^\infty(\wt{v}(B_1))} \leq {\bm \epsilon_p},\ \big\| d \wt{v} \big\|_{L^2(B_1)} \leq \big\| du \big\|_{L^2(B_\rho)} + \sqrt{\pi} {\rho\kappa \over {\bm \epsilon}_p} \leq {\bm \epsilon_1}.
\end{align*}
Then Lemma \ref{lemmaa2} and (\ref{equationa3}) imply that
\begin{align*}
\big\| dv \big\|_{L^p(B_{1\over 2})} \leq {\bm C_p} \Big( \big\| d u \big\|_{L^2(B_\rho)}  +  \rho \big\| \nu(u) \big\|_{L^\infty(B_\rho)}\Big).
\end{align*}
\begin{align*}
{\rm diam} \big( u ( B_{\rho\over 2} ) \big) = {\rm diam} \big( v ( B_{1\over 2}) \big) \leq {\bm C_p} \Big( \big\| d u \big\|_{L^2(B_\rho)}  +  \rho \big\| \nu(u) \big\|_{L^\infty(B_\rho)}\Big).
\end{align*}
The rescaling relation of $L^p$-norms implies that
\begin{align*}
\big\| d u \big\|_{L^p(B_{\rho \over 2})} \leq {\bm C_p} \rho^{{2\over p} - 1} \Big( \big\| d u \big\|_{L^2(B_\rho)}  +  \rho \big\| \nu(u) \big\|_{L^\infty(B_\rho)}\Big).
\end{align*}
\end{proof}

\subsection{Mean value estimates}

We quote several important mean value estimates for differential inequalities of the Laplace operator on the plane. Let $B_r$ be the radius $r$ open disk in ${\mb C}$ centered at the origin, with the standard coordinates $(s, t)$. Let $\Delta = \partial_s^2 + \partial_t^2$. 

\begin{lemma}\label{lemmaa5}(\cite[Page 156]{Salamon_lecture})
Suppose $f: B_r \to {\mb R}$ with $f(z) \geq 0$ be a smooth function, satisfying
\begin{align*}
\Delta f\geq -A - B f^2
\end{align*}
where $A\geq 0$, $B >0$. Then
\begin{align*}
\int_{B_r} f \leq {\pi \over 16 B} \Longrightarrow f(0) \leq {8 \over \pi r^2} \int_{B_r} f + {A r^2 \over 4}.
\end{align*}
\end{lemma}

\subsection{Hofer's lemma}

In proving compactness we used the following lemma, which is due to Hofer.

\begin{lemma}\cite[Lemma 4.6.4]{McDuff_Salamon_2004}\label{lemmaa7} Let $(X, d)$ be a metric space, $f: X \to {\mb R}$ be a non-negative continuous function. Suppose $x \in X$, $\delta>0$ and the closed ball $\ov{B}_{2\delta}(x) \subset X$ is complete. Then there exists $\xi \in X$, $\epsilon \in (0, \delta]$ such that
\begin{align*}
d(x, \xi) < 2\delta,\ \sup_{B_\epsilon(\xi)} \leq 2 f(\xi),\ \epsilon f(\xi) \geq \delta f(x).
\end{align*}
\end{lemma}

\section{Equivariant topology}\label{appenxib}

Suppose $G$ is a compact Lie group, $N$ is a $G$-manifold and $P\to M$ is a principal $G$-bundle over a closed oriented manifold $M$, then any continuous section $s$ of the associated bundle $P\times_K N$ defines a cycle in the Borel construction $N_G$, which represents an equivariant homology class
\begin{align*}
s_*[M] \in H_{{\rm dim}M}^G (N; {\mb Z}).
\end{align*}
In this current paper, we would like to define such an equivariant fundamental class for any solution $(A, u)$ to the perturbed gauged Witten equation by using the section $u$. However, since the monodromy of the $r$-spin structure at the punctures could be nontrivial, the image of the section $u$ is an equivariant cycle in $X$ only in the orbifold sense. So the contribution from the cylindrical ends $U_j$ should be weighted by a rational weight, and the fundamental class of a solution $(A, u)$ should be a class
\begin{align*}
\big[ A, u \big] \in H_2^G \big(  \wt{X}; {\mb Z}[r^{-1} ] \big).
\end{align*}
We will carry this out explicitly in this subsection.

We first recall a general way of defining a rational fundamental class of an orbifold section of an associated bundle over an orbicurve. We assume that the reader is familiar with the notion of orbicurves (orbifold Riemann surfaces) and orbifold bundles over an orbicurve, so we will be sketchy when referring to such structures. 

We assume that we have a compact Riemann surface $\Sigma$ with several distinct punctures $z_1, \ldots, z_k$. An orbifold chart near $z_j$ with local group $\Gamma_j \simeq {\mb Z}_{r_j}$ is a holomorphic map
\begin{align*}
\pi_j: {\mb D} \to \Sigma
\end{align*}
which maps $0$ to $z_j$ and can be expressed as $\zeta \mapsto \zeta^{r_j} $ in local coordinates. A collection of orbifold charts $\{\pi_j\}_{j=1}^k$ define an orbicurve structure. An equivalence relation can be defined among orbifold charts, and an equivalence class is called an orbicurve ${\mc C}$. 

Now suppose for each $j$, we have an injective homomorphism $\chi_j: {\mb Z}_{r_j} \to G$. An orbifold $G$-bundle over ${\mc C}$ is a usual $G$-bundle over $\Sigma^*:= \Sigma \setminus \{z_1, \ldots, z_k\}$, together with a collection of ``bundle charts'' 
\begin{align*}
\left( \wt\pi_j, \pi_j \right) : \left( {\mb D}^*\times G, {\mb D}^* \right) \to \left( P|_{\Sigma^*}, \Sigma^* \right),\ j=1, \ldots, k,
\end{align*}
where $\pi_j: {\mb D}^* \to \Sigma^*$ extends to an orbifold chart near $p_j$ and $\wt\pi_j$ covers $\pi_j$; moreover, $\wt\pi_j$ is invariant under the $\Gamma_j$-action on the left by $\upgamma \cdot( \zeta, k) = \left( \upgamma \zeta, \chi_j (\upgamma) k \right)$. An equivalence class of orbifold bundle charts defines an orbifold $G$-bundle ${\mc P}\to {\mc C}$. As a topological space, ${\mc P}$ is
\begin{align*}
{\mc P}:= P^* \cup \big( \bigcup_{j=1}^k {\mb D} \times G \big) /\sim
\end{align*}
with the equivalence relation generated by $p \sim (\zeta, k)$ if $\wt\pi_j(\zeta, k) = p$. 

Now if $N$ is a $G$-manifold, we can have an ``orbifold associated bundle'' ${\mc Y}:= {\mc P}\times_G N$, which contains the usual associated bundle $Y^*:= P^* \times_G N$ as a proper subset. Each bundle chart $\wt\pi_j$ induces a chart $\wt\pi_j^N: {\mb D}^* \times N \to Y^*$ by
\begin{align*}
\wt\pi_j^N( \zeta, x ) = \left[ \wt\pi_j(\zeta, 1), x \right],
\end{align*}
which is invariant under the $\Gamma_j$-action $\upgamma(\zeta, x) = ( \upgamma \zeta, \upgamma x)$.

Suppose we have a continuous section $u: \Sigma^* \to Y^*$, identified with an equivariant map $U: P^* \to N$. Then the composition
\begin{align*}
U \circ \wt\pi_j: {\mb D}^* \times G \to N
\end{align*}
is again a $G$-equivariant map and invariant under the $\Gamma_j$-action. It can be viewed as a continuous section over the chart ${\mb D}^* \times N$. If it extends continuous to the origin $0\in {\mb D}_j$ for all $j$, then we have an orbifold section of ${\mc Y}\to {\mc C}$. 

Now we can define the rational fundamental class of a continuous orbifold section of ${\mc Y}$. First, we construct a CW complex out of the orbicurve. The complement $\Sigma \setminus U$ is a surface with boundary, hence we can regard it as a CW complex in such a way that $\partial U$ is a subset of the 1-skeleton of $\Sigma \setminus U$. Then we take $k$ copies of 2-cells ${\mb D}_j$ and attach it to $\partial U$ by the $r_j$-to-1 map $\zeta_j \mapsto \zeta_j^{r_j}$. This CW complex is denoted by $|{\mc C}|$. Then, it is easy to see that the singular chain
\begin{align*}
\big[ {\mc C} \big]:= \big[ \Sigma \setminus U \big] + \sum_{j=1}^k {1\over r_j} \big| {\mb D}_j \big|
\end{align*}
defines a rational homology class in $H_2 \big( |{\mc C}|; {\mb Z}[ r^{-1}] \big)$, if $r$ is divisible by all $r_j$. 

Moreover, the orbibundle charts defines a continuous $G$-bundle $|{\mc P}| \to |{\mc C}|$ (in the usual sense); the orbifold section $s$ defines a continuous section $|s|: |{\mc P}| \to N$. Hence we obtained a continuous map (up to homotopy) $|{\mc C}| \to N_G$. The pushforward of the rational class $[{\mc C}]$ is then a class
\begin{align*}
s_* \big[ {\mc C} \big] \in H_2 \big( N_G; {\mb Z}[r^{-1}] \big) = H_2^G \big( N; {\mb Z}[r^{-1}] \big).
\end{align*}

\bibliography{symplectic_ref}
	
\bibliographystyle{amsalpha}

\end{document}